\documentclass[11pt,a4paper]{amsart}
\usepackage{amssymb,amsmath,amsthm}
\usepackage{mathrsfs}
\usepackage[all,cmtip]{xy}
\usepackage[table]{xcolor}
\usepackage{pdflscape}
\newcommand\scalemath[2]{\scalebox{#1}{\mbox{\ensuremath{\displaystyle #2}}}}
\usepackage{hyperref}
\usepackage[msc-links]{amsrefs}
\hypersetup{
  colorlinks   = true,	
  urlcolor     = blue,	
  linkcolor    = blue,	
  citecolor   = red	
}
\calclayout
\newtheorem{theorem}{Theorem}
\newtheorem{lemma}[theorem]{Lemma}
\newtheorem{remark}[theorem]{Remark}
\newtheorem{corollary}[theorem]{Corollary}
\newtheorem{definition}[theorem]{Definition}
\newtheorem{proposition}[theorem]{Proposition}
\numberwithin{theorem}{section} \numberwithin{equation}{section}
\begin{document}
\title[On isogenies among abelian surfaces]{On isogenies among certain abelian surfaces}
\author{Adrian Clingher}
\address{Dept.\!~of Mathematics, University of Missouri - St. Louis, MO 63121}
\email{clinghera@umsl.edu}
\author{Andreas Malmendier}
\address{Dept.\!~of Mathematics \& Statistics, Utah State University, Logan, UT 84322}
\email{andreas.malmendier@usu.edu}
%
\author{Tony Shaska}
\address{Dept.~of Mathematics \& Statistics,  Oakland University, Rochester, MI 48309}
\email{shaska@oakland.edu}
\begin{abstract}
We construct a three-parameter family of non-hyperelliptic and bielliptic plane genus-three curves whose associated Prym variety is two-isogenous to the Jacobian variety of a general hyperelliptic genus-two curve. Our construction is based on the existence of special elliptic fibrations with the section on the associated Kummer surfaces that provide a simple geometric interpretation for the rational double cover induced by the two-isogeny between the abelian surfaces. 
\end{abstract}
\maketitle
\section{Introduction}
A smooth, projective curve is called hyperelliptic if it admits a map of degree two onto a curve of genus zero. Within the (coarse) moduli space of irreducible, projective curves of genus three $\mathcal{M}_3$ we denote the hyperelliptic locus by $\mathcal{M}_3^h$ and the isomorphism class of such hyperelliptic curve $\mathcal{H}$ by $ [\mathcal{H}] \in \mathcal{M}^h_3$. It is known that  $\mathcal{M}_3^h$ is an irreducible five-dimensional sub-variety\footnote{The hyperelliptic involution on an irreducible, smooth, projective curve of genus $g$ is unique if $g \ge 2$.} of $\mathcal{M}_3$. Within the moduli space $\mathcal{M}_3$, we also define the bielliptic locus
\[
 \mathcal{M}_3^{b} = \left\lbrace [\mathcal{D}] \in \mathcal{M}_3 \Big| \; \mathcal{D} \; \text{is bielliptic} \right\rbrace \;,
 \]
 where bielliptic means that irreducible, projective curve $\mathcal{D}$ of genus three admits a degree-two morphism $\pi^{\mathcal{D}}_\mathcal{E}: \mathcal{D} \to \mathcal{E}$ onto an elliptic curve $\mathcal{E}$. We denote by $[\mathcal{D}]\in \mathcal{M}_3$ the isomorphism class of $\mathcal{D}$, and by $\tau$ the involution, i.e., the element of $\operatorname{Aut}(\mathcal{D})$ which interchanges the sheets of $\pi^{\mathcal{D}}_\mathcal{E}$, so that  $\mathcal{E} \cong \mathcal{D} /\langle\tau\rangle$.  For such a bielliptic genus-three curve $\mathcal{D}$ with a bielliptic involution $\tau$, the Prym variety $\operatorname{Prym}(\mathcal{D},\pi^{\mathcal{D}}_\mathcal{E})$ is defined to be the connected component of the kernel of the induced norm map $\pi^{\mathcal{D}}_{\mathcal{E}, \star}$.   %
 \par We recall from \cite{MR932781} that $\mathcal{M}_3^{b}$ is an irreducible four-dimensional sub-variety of  $\mathcal{M}_3$, and it is the unique component of maximal dimension of the singular locus\footnote{By the Castelnuovo-Severi inequality it follows that bielliptic curves of genus $g \ge 6$ admit precisely one bielliptic structure and that bielliptic curves of genus $g \ge 4$ cannot be hyperelliptic.} of $\mathcal{M}_3$. It was proven in \cite{MR1816214} that:
\begin{proposition}
\label{thm:genus3-hyperelliptc-bielliptic}
\hspace{1em}
\begin{enumerate}
\item $[\mathcal{H}] \in \mathcal{M}_3^{b} \cap \mathcal{M}_3^h$ iff $\mathcal{H}$ is a double cover of a genus-two curve $\mathcal{C}$.
\item $\mathcal{M}_3^{b} \cap \mathcal{M}_3^h$ is an irreducible, 3-dimensional, rational sub-variety of $\mathcal{M}_3^{b}$.
\end{enumerate}
\end{proposition} 
\par On the other hand, among the smooth genus-three curves, there are the ones that are given as plane quartics in $\mathbb{P}^2$.  However, smooth plane quartics are never hyperelliptic. We can ask whether the following abelian surfaces are related by isogeny: (1) the Jacobian variety $\operatorname{Jac}(\mathcal{C})$ associated with a hyperelliptic and bielliptic curve $\mathcal{H}$ in $\mathcal{M}_3^{b} \cap \mathcal{M}_3^h$ covering a smooth genus-two curve $\mathcal{C}$, and (2) the Prym variety $\operatorname{Prym}(\mathcal{D},\pi^{\mathcal{D}}_\mathcal{E})$ associated with a bielliptic plane genus-three curve $\mathcal{D}$ in $\mathcal{M}_3^{b}$. That is, we ask for what curves $\mathcal{H}$ and $\mathcal{D}$, there is an isogeny $\Psi: \operatorname{Prym}(\mathcal{D},\pi^{\mathcal{D}}_\mathcal{E}) \to \operatorname{Jac}(\mathcal{C})$, where curves are embedded as divisors representing the respective polarization.
\par In~\cite{MR946234} Barth studied abelian surfaces $\mathfrak{A}$ with $(1,2)$-polarization line bundle $\mathcal{L}$ and proved their close connection with Prym varieties of smooth bielliptic genus-three curves. An excellent summary of Barth's construction was given by Garbagnati in \cites{Garbagnati08, MR3010125, MR3563178}. Abelian surfaces with $(1,2)$-polarization were also discussed in \cites{MR2306633, MR2804549, MR2729013}. Bielliptic genus-three curves and abelian surfaces with $(1,2)$-polarization have also appeared as spectral curves of Lax representations of certain algebraic integrable systems, most importantly (for us) the Kovalevskaya top \cites{MR912838,MR923636, MR990136, MR3798190}. On the other hand, Kovalevskaya presented in her celebrated paper \cite{MR1554772} a separation of variables of the corresponding integrable system using a certain (hyperelliptic) genus-two curve, nowadays commonly referred to as \emph{Kowalewski curve}, whose Jacobian is associated with the integrals of motion of the Kovalevskaya top. 
\par Barth's seminal work proved that the linear system $|\mathcal{L}|$ is a pencil on $\mathfrak{A}$ of bielliptic genus-three curves. Horozov and van Moerbeke \cite{MR990136} wrote down a specific Lefschetz pencil of bielliptic genus-three curves $\mathcal{D}_\lambda$ over $\mathbb{P}^1 \ni \lambda$, generically smooth and with twelve double points. However, the members of the pencil are generically \emph{not} plane genus-three curves. The construction of the pencil is based on Barth's elegant geometric description for Prym varieties of bielliptic genus-three curves as intersection of quadrics in projective space~\cites{MR946234, Clingher:2018aa}. However, less attention has been given in this context to the elliptic fibrations with section that the associated Kummer surfaces admit. The first two authors studied several of the elliptic fibrations on the Kummer surface associated with an abelian surface with $(1,2)$-polarization in \cite{Clingher:2017aa}, using the results of Mehran \cites{MR2804549, MR2306633, MR2708983} and Garbagnati \cite{Garbagnati08}. Among these fibrations is a special elliptic fibration with twelve singular fibers which is directly induced by the linear system $|\mathcal{L}|$ on $\mathfrak{A}$.
\par In this article we will construct a three-parameter non-hyperelliptic and bielliptic genus-three curve whose associated Prym variety is two-isogenous to the Jacobian variety of the general three-parameter hyperelliptic genus-two curve.  We will consider the genus-two curve $\mathcal{C}$ to be in Rosenhain form
\begin{equation}
\label{Eq:Rosenhain_g2_intro}
  \mathcal{C}: \quad Y^2 = X \big(X- 1\big) \big(X-\lambda_1 \big) \big(X-\lambda_2 \big) \big(X-\lambda_3 \big)  \;,
\end{equation} 
with parameters $\lambda_1, \lambda_2, \lambda_3$. We define the subgroup $\Gamma_2(2n) = \lbrace M \in \Gamma_2 | \, M \equiv \mathbb{I} \mod{2n}\rbrace$ and Igusa's congruence subgroups $\Gamma_2(2n, 4n) = \lbrace M \in \Gamma_2(2n) | \, \operatorname{diag}(B) =  \operatorname{diag}(C) \equiv \mathbb{I} \mod{4n}\rbrace$ of the Siegel modular group $\Gamma_2= \{ M =\bigl(\begin{smallmatrix}A&B\\ C&D \end{smallmatrix} \bigr) \in \operatorname{Sp}_4(\mathbb{Z}) \}$, such that
\begin{equation}
 \Gamma_2/\Gamma_2(2)\cong S_6, \quad  \Gamma_2(2)/\Gamma_2(2,4)\cong (\mathbb{Z}/2\mathbb{Z})^4, \quad \Gamma_2(2,4)/\Gamma_2(4,8)\cong (\mathbb{Z}/2\mathbb{Z})^9,
\end{equation}
where $S_6$ is the permutation group of six elements.  Then, $\lambda_1, \lambda_2, \lambda_3$ are modular with respect to $\Gamma_2(2)$. We define a modular form $l$ such that $l^2=\lambda_1\lambda_2\lambda_3$, and three modular forms $m^{(i,j,k)}$ such that $(m^{(i,j,k)})^2=(\lambda_i-\lambda_j)(\lambda_i-\lambda_k)/[(1-\lambda_j)(1-\lambda_k)]$ with $\{i, j, k\} = \{1,2,3\}$. Then, $l$ is a modular with respect to $\Gamma_2(2,4)$, and $m^{(i,j,k)}$ are modular with respect to $\Gamma_2(4,8)$. 
\par The main theorem of this article is the following:
\begin{theorem}\label{thm:main}
Consider the plane bielliptic genus-three curves $\mathscr{D}_{[s^*_0:s^*_1]}$, given by
\begin{equation}
\left( w^2 - u^2 - \frac{C(s^*_0,s^*_1)}{E(s^*_0,s^*_1)} \, uv - \frac{D(s^*_0,s^*_1)}{E(s^*_0,s^*_1)}\,  v^2 \right)^2 = u^4 + B(s^*_0,s^*_1) \, u^2 v^2 + A^2(s^*_0,s^*_1) \, v^4 \,,
\end{equation}
where $[u:v:w] \in \mathbb{P}^2$, $A, B, C, D, E$ are polynomials in $[s_0:s_1]\in \mathbb{P}^1$ with coefficients in $\mathbb{Z}[l, \lambda_1, \lambda_2, \lambda_3]$ defined in Appendix~\ref{App:coeffs}, and $[s^*_0:s^*_1] \in \mathbb{P}^1$ is one of the six special points given by
\begin{equation}
\label{eqn:sol_s_intro}
  [s^*_0:s^*_1]=\Big[  (1+\lambda_i - \lambda_j - \lambda_k) l \; : \; ( \lambda_i - \lambda_j \lambda_k)
  \pm  m^{(i,j,k)}  (1-\lambda_j)(\lambda_0-\lambda_k)  \Big]\;,
\end{equation}  
with $\{ i, j, k\} = \{1,2,3\}$ such that $E(s^*_0,s^*_1) \not =0$.
\par Then, the curves $\mathscr{D}_{[s^*_0:s^*_1]}$ are smooth and irreducible, and admit a degree-two covering $\pi^{\mathscr{D}}_{\mathscr{E}}: \; \mathscr{D}_{[s^*_0:s^*_1]} \to \mathscr{E}_{[s^*_0:s^*_1]}$ onto a smooth elliptic curve $\mathscr{E}_{[s^*_0:s^*_1]}\cong \mathcal{D}^*/\langle \tau \rangle$ where $\tau$ is the bielliptic involution. Moreover, the Prym variety $\operatorname{Prym}( \mathscr{D}_{[s^*_0:s^*_1]} , \pi^{\mathscr{D}}_{\mathscr{E}})$ is an abelian surface that admits a $(1,2)$-isogeny 
\[ 
\Psi: \operatorname{Prym}( \mathscr{D}_{[s^*_0:s^*_1]} , \pi^{\mathscr{D}}_{\mathscr{E}}) \to \operatorname{Jac}(\mathcal{C})\]
onto the principally polarized abelian surface $\operatorname{Jac}(\mathcal{C})$ of $\mathcal{C}$ in Equation~(\ref{Eq:Rosenhain_g2_intro}).
\end{theorem}
\par  The geometry underlying Theorem~\ref{thm:main} is the following: if one chooses 6 points in $\mathbb{P}^1$, partitioned into 2 and 4, then one obtains three double covers of $\mathbb{P}^1$, branched respectively at the marked sets of 2, 4, and all 6 points. We will label them $\mathcal{R}$, $\mathcal{E}$, $\mathcal{C}$ with genus $0$, $1$, and $2$, respectively. These three curves have a common double cover $\mathcal{H}$, which can be obtained as the fiber product over $\mathbb{P}^1$ of any two of the three. This is a Galois cover of $\mathbb{P}^1$ with group $\mathbb{Z}/2\mathbb{Z} \times \mathbb{Z}/2\mathbb{Z}$, and the three intermediate curves $\mathcal{R}$, $\mathcal{E}$, $\mathcal{H}$ are the quotients by the three $\mathbb{Z}/2\mathbb{Z}$ subgroups. $\mathcal{H}$ is hyperelliptic via the map to $\mathcal{R}$ and bielliptic via the map to $\mathcal{E}$. Its Jacobian decomposes as $\operatorname{Jac}{\mathcal{H}} \cong \operatorname{Jac}{\mathcal{C}} \times \mathcal{E}$. On the other hand, Recillas’ famous trigonal construction \cite{MR0480505} relates to such a tower $\{ \mathcal{R}, \mathcal{E}, \mathcal{H} \}$ a non-hyperelliptic and bielliptic genus-three curve $\mathcal{D}$ such that the Prym of the latter is two-isogenous to $\operatorname{Jac}{\mathcal{C}}$; see~\cite{MR1188194}.
\par Amazingly, one can use two special elliptic fibrations with section on the Kummer surfaces associated with $\operatorname{Jac}{\mathcal{C}}$ and $\operatorname{Prym}(\mathcal{D},\pi^{\mathcal{D}}_\mathcal{E})$ to provide a simple geometric interpretation for the rational double cover induced by the two-isogeny between the abelian surfaces.  We then work backwards and obtain explicit expressions for the coefficients of suitable normal forms for $\mathcal{D}$ and $\mathcal{C}$ in terms of Siegel modular forms. This is the content of Theorem~\ref{thm:main}. Applications of isogenies of Pryms to hyperelliptic Jacobians are of central importance in cryptography; see \cite{frey-sh}*{Section 9} for further details. 
\par The paper is structured as follows: in Section~\ref{sec2} we consider an abelian surface $\mathfrak{A}$ with polarization of type $(1,2)$. On the Kummer surface $\operatorname{Kum}(\mathfrak{A})$ we identify a special elliptic fibration, alongside with a set of generators for the Mordell-Weil group and symplectic automorphisms in Theorem~\ref{thm:involutions} that turn out to be  crucial for the proof of Theorem~\ref{thm:main}.  In Section~\ref{sec3}  we determine a convenient normal form for a hyperelliptic and bielliptic genus-three curve that is the double cover of a general genus-two curve. We then generalize this construction to pencils and establish the connection to the aforementioned elliptic pencil on $\operatorname{Kum}(\mathfrak{A})$, providing explicit formulas for the coefficients of all normal forms in terms of suitable modular forms. In Section~\ref{sec4} we give a geometric description of plane bielliptic genus-three curves and determine a criterion for the quotient (elliptic) curves to have a rational level-two structure and branch locus. In Section~\ref{sec:proof} we carry out the proof of Theorem~\ref{thm:main}: using the results of Section~\ref{sec2}, we identify six special members of the fibration induced by the pencil $|\mathcal{L}|$ on $\mathfrak{A}$ where the elliptic fiber satisfies the conditions of Proposition~\ref{cor:bielliptic_quotient} and its double cover is a smooth bielliptic plane quartic curve. Using the results of Section~\ref{sec3} the plane bielliptic genus-three curve can then be related back to the Rosenhain normal form of a general genus-two curve to prove our theorem.

\subsection*{Acknowledgments}
We would like to thank the reviewers for their thoughtful comments and efforts towards improving our manuscript. A.M. acknowledges support from the Simons Foundation through grant no.~202367. 
\bigskip

\section{Abelian and Kummer surfaces}
\label{sec2}
Polarizations on an abelian surface $\mathfrak{A}\cong \mathbb{C}^2/\Lambda$ are known to correspond to positive definite hermitian forms $H$ on $\mathbb{C}^2$,  satisfying $E = \operatorname{Im} H(\Lambda,\Lambda) \subset \mathbb{Z}$.  In turn, such a hermitian form determines a line bundle $\mathcal{L}$ in the N\'eron-Severi group $\mathrm{NS}(\mathfrak{A})$. One may always then choose a basis of $\Lambda$ such that $E$ is given  by a matrix $\bigl(\begin{smallmatrix} 0&D\\ -D&0 \end{smallmatrix} \bigr)$ with $D=\bigl(\begin{smallmatrix}d_1&0\\ 0&d_2 \end{smallmatrix} \bigr)$ where $d_1, d_2 \in \mathbb{N}$, $d_1, d_2 \ge 0 $, and $d_1$ divides $d_2$. The pair $(d_1, d_2)$ gives the {\it type} of the polarization. 
\par If $\mathfrak{A}=\operatorname{Jac}(\mathcal{C})$ is the Jacobian of a smooth curve $\mathcal{C}$ of genus two, the hermitian  form associated to the divisor class $[\mathcal{C}]$ is a polarization of type $(1, 1)$ - a \emph{principal polarization}.  Conversely, a principally polarized abelian surface is either the Jacobian of a smooth curve of genus two or the product of two complex elliptic curves, with the product polarization. 
\par Let $\mathfrak{A}$ an abelian surface defined over $\mathbb{C}$ and $-\mathbb{I}$ be the minus identity involution on $\mathfrak{A}$.  The quotient $\mathfrak{A}/\langle -\mathbb{I} \rangle$ has sixteen ordinary double points and its minimum resolution, denoted $\operatorname{Kum}(\mathfrak{A})$, is known as the \emph{Kummer surface} of  $\mathfrak{A}$. Thus, there is an even set of 16 disjoint rational curves $K_i$ for $0\le i \le 15$ such that $K_i \circ K_j = -2 \delta_{ij}$.   The double points are the images of the order-two points $\{ \mathsf{P}_0, \dots, \mathsf{P}_{15}\}$ on $\mathfrak{A}$, i.e., elements of $\mathfrak{A}[2]$, and the disjoint rational curves $\{ K_0, \dots, K_{15} \}$ are the exceptional divisors introduced in the blow-up process. The minimal primitive sub-lattice which contains all these curves in called \emph{Kummer lattice}. In particular, they form an even set in the N\'eron-Severi lattice. We recall that an \emph{even set} of rational curves is a set of disjoint $(-2)$-rational smooth curves $\{ K_0, \dots, K_{15} \}$ such that there exists a divisor $\delta$ in the N\'eron-Severi lattice with $K_0 + \dots +  K_{15} \sim 2 \delta$ where $\sim$ denotes linear equivalence. Since they form an even set, the class $\hat{K} = \frac{1}{2}(K_0 + \dots + K_{15} )$ is an element of this lattice with $\hat{K}^2=-8$. However, the classes $K_i$ and $\hat{K}$ do not generate, over $\mathbb{Z}$, the minimal primitive lattice containing these curves. The N\'eron-Severi lattice $\mathrm{NS}(\operatorname{Kum}\mathfrak{A})$ is generated over $\mathbb{Q}$ by the classes $K_i$, and one additional class $H$ with $H^2=8$ and $H \circ K_i=0$ for $0\le i \le 15$.
\subsection{Abelian surfaces with \texorpdfstring{$(1,2)$}{(1,2)}-polarization}
\label{ssec:AbSrfc}
Let us now consider the generic abelian surface $\mathfrak{A}$ with a $(1,2)$-polarization.  Let this polarization of type $(d_1,d_2)=(1,2)$ be given by an ample symmetric line bundle $\mathcal{L}$ such that $\mathcal{L}^2 =4$. We also assume that the Picard number $\rho(\mathfrak{A})=1$ such that the N\'eron-Severi group of $\mathfrak{A}$ is generated by $\mathcal{L}$ \cite{MR2062673}. The line bundle $\mathcal{L}$ defines an associated rational map $\phi=\phi_{\mathcal{L}}: \mathfrak{A} \to \mathbb{P}^{d_1d_2-1}=\mathbb{P}^1$. Since $h^0(\mathfrak{A},\mathcal{L})=2$, the linear system $|\mathcal{L}|$ is a pencil on $\mathfrak{A}$ and the map $\phi_{\mathcal{L}}$ is a rational map $\phi_{\mathcal{L}}: \mathfrak{A} \to \mathbb{P}^1$. As $\mathcal{L}^2 = 4$, each curve in $|\mathcal{L}|$ has self-intersection equal to $4$.  Since we assumed $\rho(\mathfrak{A})=1$, the abelian surface $\mathfrak{A}$ cannot be a product of two elliptic curves or isogenous to a product of two elliptic curves. 
\par It was proven in \cite{MR2062673}*{Prop.~4.1.6, Lemma 10.1.2} that the linear system $|\mathcal{L}|$ has exactly four base points if $(d_1,d_2)=(1,2)$. To characterize these four base points, Barth proves in~\cite{MR946234} that the base points form the translation group $T(\mathcal{L})=\{P \in \mathfrak{A} \mid \, t_P^*\mathcal{L}=\mathcal{L}\}$ where elements of $\mathfrak{A}$ act by translation $t_p(x)=x+P$. Moreover, he proves $T(\mathcal{L}) \cong (\mathbb{Z}/2\mathbb{Z})^2$ and that the base points all have order two on the abelian surface $\mathfrak{A}$, and we denote them by $\{ \mathsf{P}_0, \mathsf{P}_1, \mathsf{P}_2, \mathsf{P}_3 \}$. A curve in the pencil $|\mathcal{L}|$ is never singular at any of the base points $\{ \mathsf{P}_0, \mathsf{P}_1, \mathsf{P}_2, \mathsf{P}_3 \}$; see \cite{MR2729013}*{Lemma~3.2}. 
Barth's seminal duality theorem in~\cite{MR946234}  can then be stated as follows:
\begin{theorem}[Barth]
\label{thm:Barth}
In the situation above, let $\mathcal{D} \in |\mathcal{L}|$ be a smooth genus-three curve in the pencil $|\mathcal{L}|$. There exists a bielliptic  involution $\tau$ on $\mathcal{D}$ with degree-two quotient map $\pi^{\mathcal{D}}_{\mathcal{E}}: \mathcal{D} \to \mathcal{E} =\mathcal{D}/\langle \tau \rangle$ such that $\mathfrak{A}$ is naturally isomorphic to the Prym variety $\operatorname{Prym}(\mathcal{D},\pi^{\mathcal{D}}_{\mathcal{E}})$ and the involution $-\mathbb{I}$ restrict to $\tau$.
\par Conversely, if $\mathcal{D}$ is a smooth bielliptic genus-three curve with degree-two quotient map $\pi^{\mathcal{D}}_{\mathcal{E}}: \mathcal{D} \to \mathcal{E} =\mathcal{D}/\langle \tau \rangle$ then $\mathcal{D}$ is embedded in $\operatorname{Prym}(\mathcal{D},\pi^{\mathcal{D}}_{\mathcal{E}})$ as a curve of self-intersection four. The Prym variety $\operatorname{Prym}(\mathcal{D},\pi^{\mathcal{D}}_{\mathcal{E}})$ is an abelian surface with a polarization of type $(1,2)$.
\end{theorem}
\subsection{An elliptic fibration on \texorpdfstring{$(1,2)$}{(1,2)}-polarized Kummer surfaces}
\label{ssec:EF}
We will denote the exceptional curves associated with the base points on $\operatorname{Kum}(\mathfrak{A})$ by $\{K_0, K_1, K_2, K_3\}$.  The map $\phi_{\mathcal{L}}: \mathfrak{A} \to \mathbb{P}^1$ induces an elliptic fibration $\pi: \operatorname{Kum}(\mathfrak{A}) \to \mathbb{P}^1$ with section $\mathsf{O}$ as follows: first, a fibration is obtained by blowing up the base points of the pencil $|\mathcal{L}|$. The fibers of this fibration are the strict transform of the curves $\mathcal{D} \in |\mathcal{L}|$ and so general fiber is a smooth genus-three curve. The involution $\tau$ lifts to an involution on this fibration whose fixed points are the exceptional curves over $\{ \mathsf{P}_0, \mathsf{P}_1, \mathsf{P}_2, \mathsf{P}_3 \}$. We then take as the general fiber of $\pi$ the quotient of the general fiber of $\phi_{\mathcal{L}}$ by the bielliptic involution. Since the a curve in the pencil $|\mathcal{L}|$ is never singular at any of the base points $\{ \mathsf{P}_0, \mathsf{P}_1, \mathsf{P}_2, \mathsf{P}_3 \}$, we can take as zero-section $\mathsf{O}$ the exceptional curve over $\mathsf{P}_0$ such that the divisor class of the section is $[\mathsf{O}]=K_0$. Garbagnati \cites{Garbagnati08, MR3010125, MR3563178} proved:
\begin{proposition}[Garbagnati]
\label{prop:Garbagnati}
The fibration $\pi$ has twelve singular fibers of Kodaira type $I_2$ and no other singular fibers. The Mordell Weil group satisfies $\operatorname{MW}(\pi, \mathsf{O})_{\mathrm{tor}}=(\mathbb{Z}/2\mathbb{Z})^2$ and $\operatorname{rank} \operatorname{MW}(\pi, \mathsf{O})=3$. The smooth fiber class $F$ with $F^2=0$ and $F \circ K_0$=1 is given by
\[
 F = \frac{H-K_0 - K_1 - K_2 - K_3}{2} \,.
\] 
\end{proposition}
The twelve non-neutral components of the reducible fibers of Kodaira type $A_1$ represent the classes $K_4, \dots, K_{15}$ of the Kummer lattice and are not intersected by the class of the zero section given by $K_0$. The remaining four classes $K_i$ with $0 \le i\le3$ satisfy $F\circ K_i=1$ and $K_j \circ K_i=0$ with $4\le j \le 15$. Thus, they represent sections of the elliptic fibration with section $(\pi,\mathsf{O})$ that we still denote by $K_i$ and intersect only neutral components of the reducible fibers, given by the divisor classes $F-K_j$ with $1 \le i\le3$ and $4\le j \le 15$.
\par In \cites{Clingher:2017aa,MR3995925} the authors introduced explicit normal forms for the elliptic fibration with section $(\pi,\mathsf{O})$, given as the affine Weierstrass model
\begin{equation}\label{eqn:EFS}
 Y^2  = X \Big( X^2 - 2 B(s) \, X + \big(B(s)^2 - 4 A(s)^2\big)  \Big) \,,
\end{equation}
where $A(s)$ and $B(s)$ are certain \emph{even} polynomials of degree four in $s$ -- we will determine them in Corollary~\ref{cor:pencilQ} and Equation~(\ref{eqn:coeffs_affine}), -- such that there are no singular fibers over $s=0, \infty$, and
\[
 A(s) = s^4 A(1/s)\,, \quad B(s) = s^4 B(1/s)\,,
\]
and the discriminant of the elliptic fiber is given by $\Delta=16 A(s)^2 \big(B(s)^2-4A(s)^2\big)^2$ has twelve roots of order two. Moreover, the elliptic fibration is invariant under the action of the hyperelliptic involution $(s,X,Y) \mapsto (s,X,-Y)$ -- which we denote by $p \mapsto -p$ for a point $p \in F$ in a fiber $F$ given by Equation~(\ref{eqn:EFS}) -- and three additional involutions given by
\begin{equation}
\begin{split}\label{eqn:involutions}
\jmath_1:& \quad (s,X,Y) \mapsto \Big(s' = -s,X,Y\Big)\,,\\
\jmath_2:& \quad (s,X,Y) \mapsto \Big(s''= \frac1s,\frac{X}{s^4},\frac{Y}{s^6}\Big) \,,\\
\jmath_3:& \quad (s,X,Y) \mapsto \Big(s''' = -\frac1s,\frac{X}{s^4},-\frac{Y}{s^6}\Big) \,.
\end{split}
\end{equation}
The involutions $s \mapsto -s$ and $s \mapsto 1/s$ and their composition map singular fibers of Equation~(\ref{eqn:EFS}) to singular fibers, and smooth fibers to smooth fibers. The zero section $\mathsf{O}$, given as the point at infinity in each fiber, and the two-torsion sections $\mathsf{T}_1, \mathsf{T}_2, \mathsf{T}_3$, given by
\begin{equation}\label{eqn:zero_sections}
 \mathsf{T}_1: (X,Y)=(0,0)\,, \quad \mathsf{T}_2: (X,Y)=(B-2A,0)\,, \quad \mathsf{T}_3: (X,Y)=(B+2A,0) \,,
\end{equation}
are invariant under the involutions $\jmath_1, \jmath_2, \jmath_3$, and the hyperelliptic involution. The two-torsion sections intersect the non-neutral components of eight reducible fibers of type $A_1$ each -- which we represent as sets $W_k = \{K_i \mid i \in I_k\}$ for index sets $I_k$ such that $|W_k|=8$ for $k= 1, 2, 3$ -- partitioning the twelve rational curves $K_j$ with $4 \le j \le 15$ into three sets of eight curves with pairwise intersections consisting of four curves, i.e., $|W_j \cap W_k|=4$ and $W_1 \cap W_2 \cap W_3=\emptyset$. None of the twelve reducible fibers are invariant under the action of the involutions $\jmath_1, \jmath_2$.  However, the sets $W_k$ and $W_j \cap W_k$  for $1\le j,k \le 3$ are invariant under $\jmath_1, \jmath_2$.  We may define divisors $\bar{K}_{W_k} = \frac{1}{2} \sum_{n \in I_k} K_n$ with $1 \le k \le 3$, which are known to be elements of the Kummer lattice \cites{Garbagnati08, MR3010125, MR3563178}, with $\bar{K}_{W_j}\circ \bar{K}_{W_k}=-2-2\delta_{jk}$ for $1\le j,k \le 3$. We also define divisors $\bar{K}_{W_j \cap W_k} = \frac{1}{2} \sum_{n \in I_j \cap I_k} K_n$ with $\bar{K}_{W_j \cap W_k}^2=-2$. By construction, the elements $\bar{K}_{W_k}$ and $\bar{K}_{W_j \cap W_k}$ for $1 \le j, k \le 3$ are invariant under the action of the involutions $\jmath_1, \jmath_2$. The twelve singular fibers of the fibration~(\ref{eqn:EFS}) arise when two-torsion sections collide. This happens as follows:
\begin{center}
\begin{tabular}{crcc}
colliding sections & equation & \# of points & fiber components\\
\hline
$\mathsf{T}_1=\mathsf{T}_2$ & $B-2A=0$ & $4$ & $W_1 \cap W_2$ \\
$\mathsf{T}_1=\mathsf{T}_3$ & $B+2A=0$ & $4$ & $W_1 \cap W_3$ \\
$\mathsf{T}_2=\mathsf{T}_3$ & $A=0$ & $4$ & $W_2 \cap W_3$
\end{tabular}
\end{center}
We have the following:
\begin{corollary}
The divisor classes of the two-torsion sections $\mathsf{T}_k$ are given by
\begin{equation}\label{eqn:secT}
 [ \mathsf{T}_k ] = 2 F + K_0 - \bar{K}_{W_k} \quad \text{for  $1 \le k \le 3$.}
\end{equation}
\end{corollary}
\begin{proof}
The proof follows from $[ \mathsf{T}_k ] \circ F =1$, $[ \mathsf{T}_k ] \circ K_0 =0$, $[ \mathsf{T}_k ] \circ K_j =1$ for $j \in I_k$ and $[ \mathsf{T}_k ] \circ K_j =0$ for $j \not \in I_k$, and $[ \mathsf{T}_k ] \circ K_l =2$ for $1 \le l \le 3$.
\end{proof}
\par In \cites{Clingher:2017aa,MR3995925} three non-torsion sections $\mathsf{S}_1, \mathsf{S}_2, \mathsf{S}_3$ of the elliptic fibration $(\pi,\mathsf{O})$ of minimal height were constructed explicitly. We will review the explicit construction of these sections in Section~\ref{ssec:KummerPencils}. For two arbitrary sections $S'$ and $S''$ of the elliptic fibration, one defines the \emph{height pairing} using the formula
  \begin{equation}
   \langle S', S'' \rangle = \chi^{\text{hol}} + \mathsf{O}\circ S' + \mathsf{O}\circ S'' - S' \circ S'' - \sum_{\{s|\Delta=0\}} C_s^{-1}(S',S'') \,,
    \end{equation}  
where the holomorphic Euler characteristic is $\chi^{\text{hol}}=2$, and the inverse Cartan matrix $C_s^{-1}$ of a fibre of type $A_1$ located over point $s$ of the discriminant locus $\Delta=0$ contributes $(\frac{1}{2})$ if and only if both $S'$ and $S''$ intersect the non-neutral component. It turns out, that the sections $\mathsf{S}_1$ and $\mathsf{S}_2$ do not intersect the zero section $\mathsf{O}$ and intersect the non-neutral components of six reducible fibers of type $A_1$ each -- which we represent as complementary sets $V_k = \{K_i \mid i \in J_k\}$ for index sets $J_k$ such that $|V_k|=6$ for $k= 1, 2$ -- partitioning the twelve rational curves $K_j$ with $4 \le j \le 15$ into two disjoint sets of six curves. We also set $W'_1 = V_1 \cup \{ K_0, K_1\}, I'_1=J_1 \cup \{0,1\}$ and $W'_2 = V_2 \cup \{ K_2, K_3\}, I'_2=J_2 \cup \{2,3\}$ and define the divisors $\bar{K}_{W'_k} = \frac{1}{2} \sum_{n \in I'_k} K_n$ with $1 \le k \le 2$.   The sets $V_1$ and $V_2$ are invariant under the action of the involution $\jmath_1$ and interchanged under the action of $\jmath_2$.  The section $\mathsf{S}_3$ intersects the non-neutral components of all reducible fibers, and the zero section such that $\mathsf{S}_3 \circ \mathsf{O} =2$.
\par We have the following:
\begin{proposition}
\label{prop:MW}
The sections $\{\mathsf{O}, \mathsf{T}_1, \mathsf{T}_2, \mathsf{T}_3, \mathsf{S}_1, \mathsf{S}_2, \mathsf{S}_3\}$ form a basis of the Mordell-Weil group of sections. In particular, we have
\begin{equation}\label{eq:MWlattice}
 \operatorname{MW}(\pi, \mathsf{O}) =(\mathbb{Z}/2\mathbb{Z})^2 \oplus \langle 1 \rangle^{\oplus 2}  \oplus \langle 2 \rangle \,.
\end{equation}
\end{proposition}
\begin{proof}
Given the explicit form of the sections $\{\mathsf{O}, \mathsf{T}_1, \mathsf{T}_2, \mathsf{T}_3, \mathsf{S}_1, \mathsf{S}_2, \mathsf{S}_3\}$, we computed the intersection pairings for their divisor classes. The results are part of the Table~\ref{tab:intersection}. The height pairings of the corresponding sections of the elliptic fibration $(\pi,\mathsf{O})$ are given in Table~\ref{tab:height}.  We observe from Table~\ref{tab:height} that the pairwise orthogonal sections $\mathsf{S}_1, \mathsf{S}_2, \mathsf{S}_3$ of height less or equal two, generate a rank-three sub-lattice of the Mordell-Weil group of sections. It was proved in \cite{Garbagnati08}*{Prop.~\!2.2.4} that the transcendental lattice of the Kummer surface $\operatorname{Kum}(\mathfrak{A})$ is isometric to $H(2)\oplus H(2) \oplus \langle -8 \rangle$ such that the determinant of the discriminant form equals $2^7$. This is in numerical agreement with the determinant of the discriminant form for the N\'eron-Severi lattice obtained from an elliptic fibration with section, twelve singular fibers of Kodaira type $I_2$, and a Mordell Weil group of sections given by Equation~(\ref{eq:MWlattice}).
\end{proof}
We recall that an automorphism of finite order on a complex K3 surface is called symplectic if it acts trivially on the holomorphic two-form of the K3 surface, and it is called anti-symplectic if it acts as multiplication by $(-1)$. These notions were introduced by Nikulin in \cite{MR544937}. We have the following:
\begin{lemma}
\label{lem:antisymplectic}
The involutions $\jmath_1, \jmath_2, \jmath_3$ are three commuting anti-symplectic involutions of the elliptic fibration with section $(\pi,\mathsf{O})$ with $\jmath_3=-\jmath_1\jmath_2$. The  involutions $\jmath_l$ for $1\le l \le3$ act on the sections $\{\mathsf{O}, \mathsf{T}_1, \mathsf{T}_2, \mathsf{T}_3, \mathsf{S}_1, \mathsf{S}_2, \mathsf{S}_3\}$ as follows:
\begin{center}
\begin{tabular}{c|ccccrrr}
			& $\mathsf{O}$	& $\mathsf{T}_1$	& $\mathsf{T}_2$	& $\mathsf{T}_3$	& $\mathsf{S}_1$	& $\mathsf{S}_2$	& $\mathsf{S}_3$\\
\hline
$\jmath_1$	& $\mathsf{O}$	& $\mathsf{T}_1$	& $\mathsf{T}_2$	& $\mathsf{T}_3$	& $\mathsf{S}_1$	& $\mathsf{S}_2$	& $-\mathsf{S}_3$\\
$\jmath_2$	& $\mathsf{O}$	& $\mathsf{T}_1$	& $\mathsf{T}_2$	& $\mathsf{T}_3$	& $\mathsf{S}_2$	& $\mathsf{S}_1$	& $\mathsf{S}_3$\\
$\jmath_3$	& $\mathsf{O}$	& $\mathsf{T}_1$	& $\mathsf{T}_2$	& $\mathsf{T}_3$	& $-\mathsf{S}_2$	& $-\mathsf{S}_1$	& $\mathsf{S}_3$
\end{tabular}
\end{center}
\end{lemma}
\begin{proof}
We check that the involutions are anti-symplectic by using an explicit representative of the holomorphic two-form for the affine Weierstrass model in Equation~(\ref{eqn:EFS}) given by $ds\wedge dX/Y$. The rest of the statement follows by explicit computation.
\end{proof}
\par We should emphasize that the operations $+$ and $-$, when used with sections of a Jacobian elliptic fibration, are operations with respect to the group law in the Mordell-Weil group $\operatorname{MW}(\pi, \mathsf{O})$, i.e., the fiberwise application of the elliptic curve group law.  In contrast, before the symbols were used in the context of divisors in the N\'eron-Severi group. We have the following:
\begin{proposition}
\label{prop:Divisors}
There are four possible choices for sections $\{\mathsf{S}'_1 ,\mathsf{S}'_2 , \mathsf{S}'_3\}$ of the elliptic fibration with section $(\pi,\mathsf{O})$ (up to permutation and the action of the hyperelliptic involution), such that the divisor classes $K_0, K_1, K_2, K_3$ are represented as
\begin{equation}\label{eqn:KummerCurves}
 K_0 = [\mathsf{O}]\, \quad K_1 = [ \mathsf{S}'_1] \,,\quad K_2 = [\mathsf{S}'_2] \,, \quad K_3 = [\mathsf{S}_3'] \,.
\end{equation}
The sections are obtained as linear combinations of the non-torsion sections $\mathsf{S}_1, \mathsf{S}_2, \mathsf{S}_3$ generating $\operatorname{MW}(\pi, \mathsf{O})$, using the elliptic-curve group law in each fiber $F$ given by Equation~(\ref{eqn:EFS}) as follows:
\begin{center}
\begin{tabular}{c|c||l|l|l||l}
\# 	& action & $\mathsf{S}'_1$	& $\mathsf{S}'_2$ 	& $\mathsf{S}'_3$ & $\sum_{i=1}^3 \mathsf{S}'_i$\\[0.2em]
\hline\hline
1	& $\pm \operatorname{id}$		& $\pm 2\mathsf{S}_1$			&  $\pm (\mathsf{S}_1+\mathsf{S}_2+\mathsf{S}_3)$		& $\pm (\mathsf{S}_1-\mathsf{S}_2+\mathsf{S}_3)$	& $\pm 2(2\mathsf{S}_1+\mathsf{S}_3)$\\
2	& $\pm \jmath_1$ & $\pm 2\mathsf{S}_1$	& $\pm (\mathsf{S}_1+\mathsf{S}_2-\mathsf{S}_3)$	&  $\pm (\mathsf{S}_1-\mathsf{S}_2-\mathsf{S}_3)$		& $\pm 2(2\mathsf{S}_1-\mathsf{S}_3)$\\
3	& $\pm \jmath_2$ & $\pm 2\mathsf{S}_2$	& $\pm (\mathsf{S}_1+\mathsf{S}_2+\mathsf{S}_3)$	&  $\pm (-\mathsf{S}_1+\mathsf{S}_2+\mathsf{S}_3)$	& $\pm 2(2\mathsf{S}_2+\mathsf{S}_3)$\\
4	& $\mp \jmath_3$ & $\pm 2\mathsf{S}_2$	& $\pm (\mathsf{S}_1+\mathsf{S}_2-\mathsf{S}_3)$	&  $\pm (-\mathsf{S}_1+\mathsf{S}_2-\mathsf{S}_3)$		& $\pm 2(2\mathsf{S}_2-\mathsf{S}_3)$
\end{tabular}
\label{tab:choice_sections}
\end{center}
\end{proposition}
\begin{proof}
We explicitly compute $2 \mathsf{S}_1, \mathsf{S}_1 + \mathsf{S}_2 + \mathsf{S}_3, \mathsf{S}_1 - \mathsf{S}_2 + \mathsf{S}_3$ using the elliptic-curve group law. Since these are sections of the elliptic fibration, we find that the intersection pairing with the smooth fiber $F$ always equals one. We then check that the three sections intersect only neutral components of the reducible fiber, i.e., the components $F-K_j$ for $4\le j \le 15$. We finally check that the three sections do not mutually intersect nor intersect the zero section $\mathsf{O}$. For $\mathsf{S}'_1=2\mathsf{S}_1$, $\mathsf{S}'_2=\mathsf{S}_1+\mathsf{S}_2+\mathsf{S}_3$, $\mathsf{S}'_3=\mathsf{S}_1-\mathsf{S}_2+\mathsf{S}_3$, the intersection pairings of all aforementioned divisor classes and height pairings of the corresponding sections are given in Table~\ref{tab:intersection}.  The sections of the table are then obtained by acting with involutions $\jmath_1, \dots, \jmath_3$ and the hyperelliptic involution. Using the height pairing one checks that these are the only possibilities.
\end{proof}
\begin{table}
\parbox{.45\linewidth}{
\scalemath{0.62}{
\begin{tabular}{c||r|r|r|r|r|r|r|r|r|r|r}
$\circ$ 	& $F$ 	& $\mathsf{O}$	& $\mathsf{T}_1$	& $\mathsf{T}_2$	& $\mathsf{T}_3$	& $\mathsf{S}'_1$	& $\mathsf{S}'_2$ 	& $\mathsf{S}'_3$ 	& $\mathsf{S}_1$	& $\mathsf{S}_2$	& $\mathsf{S}_3$\\
\hline\hline
$F$		& 0		& $1$	& $1$	& $1$	& $1$	& $1$	& $1$	& $1$	& $1$	& $1$	& $1$\\
$\mathsf{O}$	& $1$	& $-2$	& $0$	& $0$	& $0$	& $0$	& $0$	& $0$	& $0$	& $0$	& $2$\\
$\mathsf{T}_1$	& $1$	& $0$	& $-2$	& $0$	& $0$	& $2$	& $2$	& $2$	& $0$	& $0$	& $0$\\
$\mathsf{T}_2$	& $1$	& $0$	& $0$	& $-2$	& $0$	& $2$	& $2$	& $2$	& $0$	& $0$	& $0$\\
$\mathsf{T}_3$	& $1$	& $0$	& $0$	& $0$	& $-2$	& $2$	& $2$	& $2$	& $0$	& $0$	& $0$\\
$\mathsf{S}'_1$	& $1$	& $0$	& $2$	& $2$	& $2$	& $-2$	& $0$	& $0$	& $0$	& $2$	& $4$\\
$\mathsf{S}'_2$	& $1$	& $0$	& $2$	& $2$	& $2$	& $0$	& $-2$	& $0$	& $1$	& $1$	& $2$\\
$\mathsf{S}'_3$	& $1$	& $0$	& $2$	& $2$	& $2$	& $0$	& $0$	& $-2$	& $1$	& $3$	& $2$\\
$\mathsf{S}_1$	& $1$	& $0$	& $0$	& $0$	& $0$	& $0$	& $1$	& $1$	& $-2$	& $2$	& $1$\\
$\mathsf{S}_2$	& $1$	& $0$	& $0$	& $0$	& $0$	& $2$	& $1$	& $3$	& $2$	& $-2$	& $1$\\
$\mathsf{S}_3$	& $1$	& $2$	& $0$	& $0$	& $0$	& $4$	& $2$	& $2$	& $1$	& $1$	& $-2$
\end{tabular}}}
\quad
\parbox{.45\linewidth}{
\scalemath{0.62}{
\begin{tabular}{c||r|r|r|r|r|r|r|r|r|r}
$\langle\bullet,\bullet\rangle$ 	& $\mathsf{O}$	& $\mathsf{T}_1$	& $\mathsf{T}_2$	& $\mathsf{T}_3$	& $\mathsf{S}'_1$	& $\mathsf{S}'_2$ 	& $\mathsf{S}'_3$	& $\mathsf{S}_1$	& $\mathsf{S}_2$	& $\mathsf{S}_3$\\
\hline\hline
$\mathsf{O}$	& $0$	& $0$ 	& $0$	& $0$ 	& $0$	& $0$	& $0$	& $0$ 	& $0$	& $0$\\
$\mathsf{T}_1$	& $0$	& $0$ 	& $0$	& $0$	& $0$	& $0$	& $0$	& $0$ 	& $0$	& $0$\\
$\mathsf{T}_2$	& $0$	& $0$ 	& $0$	& $0$ 	& $0$	& $0$	& $0$	& $0$ 	& $0$	& $0$\\
$\mathsf{T}_3$	& $0$	& $0$ 	& $0$	& $0$ 	& $0$	& $0$	& $0$	& $0$ 	& $0$	& $0$\\
$\mathsf{S}'_1$	& $0$	& $0$ 	& $0$	& $0$	& $4$	& $2$	& $2$	& $2$ 	& $0$	& $0$\\
$\mathsf{S}'_2$	& $0$	& $0$ 	& $0$	& $0$ 	& $2$	& $4$	& $2$	& $1$ 	& $1$	& $2$\\
$\mathsf{S}'_3$	& $0$	& $0$ 	& $0$	& $0$ 	& $2$	& $2$	& $4$	& $1$ 	& $-1$	& $2$\\
$\mathsf{S}_1$	& $0$ 	& $0$	& $0$	& $0$	& $2$ 	& $1$	& $1$	& $1$	& $0$	& $0$\\
$\mathsf{S}_2$	& $0$ 	& $0$	& $0$	& $0$	& $0$ 	& $1$	& $-1$	& $0$	& $1$	& $0$\\
$\mathsf{S}_3$	& $0$ 	& $0$	& $0$	& $0$	& $0$ 	& $2$	& $2$	& $0$	& $0$	& $2$
\end{tabular}}}
\caption{Intersection and Height Pairings}
\label{tab:height}\label{tab:intersection}
\end{table} 
\begin{remark}
The different choices in Prop.~\ref{prop:Divisors} are permuted by automorphisms which fix the ample class; see Thm.~\ref{thm:involutions}.
\end{remark}
\par Using the elliptic-curve group law on each fiber $F_s \ni (s,X,Y) $ given by Equation~(\ref{eqn:EFS}), the three involutions in Equation~(\ref{eqn:involutions}), and a choice of sections $\{\mathsf{S}'_1 ,\mathsf{S}'_2 , \mathsf{S}'_3\}$ in Proposition~\ref{prop:MW}, we define involutions of the elliptic fibration with section $(\pi,\mathsf{O})$ mapping smooth or singular fibers to smooth or singular fibers, respectively,
\begin{equation}
\label{eqn:symplectic_involutions}
\begin{array}{llll}
 \imath_1 : \quad & (s,X,Y) \mapsto  (s',X',Y') &= -\jmath_1\Big(s, X, Y\Big) &+ \; \mathsf{S}'_1 \Big|_{F_{s'}}   \,,\\
 \imath_2 : \quad & (s,X,Y) \mapsto  (s'',X'',Y'') &= -\jmath_2\Big(s, X, Y\Big) &+  \; \mathsf{S}'_2 \Big|_{F_{s''}} \,,\\
 \imath_3 : \quad & (s,X,Y) \mapsto  (s''',X''',Y''') &= -\jmath_3\Big(s, X, Y\Big) &+   \; \mathsf{S}'_3 \Big|_{F_{s'''}} \,.
\end{array}
\end{equation}
By a slight abuse of notation we also denote the involutions more intuitively by $p \mapsto \imath_l(p) := - \jmath_l(p) + K_l$ for $p\in F$ and $1\le l \le 3$.
We have the following:
\begin{theorem}
\label{thm:involutions}
The involutions $\imath_1, \imath_2, \imath_3$ are three commuting symplectic involutions of the elliptic fibration with section $(\pi,\mathsf{O})$ on $\operatorname{Kum}(\mathfrak{A})$ such that $\imath_3=\imath_1\circ\imath_2$. The involutions act on the divisor classes $\{F, K_0, K_1, K_2, K_3 \}$ as follows:
\begin{center}
\begin{tabular}{c|ccccc}
			& $F$	& $K_0$	& $K_1$	& $K_2$	& $K_3$ \\
\hline
$\imath_1$ 	& $F$	& $K_1$	& $K_0$	& $K_3$	& $K_2$ \\
$\imath_2$ 	& $F$	& $K_2$	& $K_3$	& $K_0$	& $K_1$ \\
$\imath_3$ 	& $F$	& $K_3$	& $K_2$	& $K_1$	& $K_0$ 
\end{tabular}
\end{center}
\end{theorem}
\begin{proof}
Each involution $\imath_l$ is a composition of the involution $\jmath_l$, an inversion given by the hyperelliptic involution, and a shift on the fiber. Since $\jmath_l$ is anti-symplectic by Lemma~\ref{lem:antisymplectic}, the involution $\imath_l$ is symplectic. One checks by explicit computation that the involutions $\imath_l$ commute and satisfy $\imath_3=\imath_1\circ\imath_2$.  The rest of the statement follows using the explicit representation of each class $K_l$ for $1\le l \le3$ in Equation~(\ref{eqn:KummerCurves}).
\end{proof}
We have the following consequence:
\begin{corollary}
For the abelian surface $\mathfrak{A}$ with polarization of type $(1,2)$ given by a line bundle $\mathcal{L}$, the translation group $T(\mathcal{L})=\{P \in \mathfrak{A} \mid \, t_P^*\mathcal{L}=\mathcal{L}\} \cong (\mathbb{Z}/2\mathbb{Z})^2$ induces the group of symplectic involutions $\{ \mathrm{id}, \imath_1, \imath_2, \imath_3 \}$ given by Equation~(\ref{eqn:symplectic_involutions}) on the elliptic fibration with section $(\pi,\mathsf{O})$ on the Kummer surface $\operatorname{Kum}(\mathfrak{A})$.
\end{corollary}
\begin{proof}
If we denote the four base points of the linear system $|\mathcal{L}|$ by $\{\mathsf{P}_0, \mathsf{P}_1, \mathsf{P}_2, \mathsf{P}_3\}$ and identify $\mathsf{P}_0 =0$ and the action by translation as follows:
\begin{center}
\begin{tabular}{c|cccc}
		& $0$	& $\mathsf{P}_1$	& $\mathsf{P}_2$	& $\mathsf{P}_3$ \\
\hline
$t_{\mathsf{P}_1}$ 	& $\mathsf{P}_1$	& $0$	& $\mathsf{P}_3$	& $\mathsf{P}_2$ \\
$t_{\mathsf{P}_2}$ 	& $\mathsf{P}_2$	& $\mathsf{P}_3$	& $0$	& $\mathsf{P}_1$ \\
$t_{\mathsf{P}_3}$ 	& $\mathsf{P}_3$	& $\mathsf{P}_2$	& $\mathsf{P}_1$	& $0$ 
\end{tabular}
\end{center}
The action of $t_{\mathsf{P}_i}$ on the abelian surface descends to a symplectic automorphism of $\operatorname{Kum}(\mathfrak{A})$. Since $\mathsf{P}_i \in T(\mathcal{L})$, the action of $t_{\mathsf{P}_i}$ on the abelian surface descends to an automorphism that preserves the elliptic fibration with section $(\pi,\mathsf{O})$ and map the zero section $\mathsf{O}$ to the section representing the image of the base point $\mathsf{P}_i$ on $\operatorname{Kum}(\mathfrak{A})$. By Theorem~\ref{thm:involutions}, this is  the group of symplectic involutions $\{ \mathrm{id}, \imath_1, \imath_2, \imath_3 \}$ given by Equation~(\ref{eqn:symplectic_involutions}).
\end{proof}
\section{Bielliptic and hyperelliptic genus-three curves}
\label{sec3}
In this section we construct a bi-double cover of $\mathbb{P}^1$ introducing the curves $\mathcal{H}$, $\mathcal{C}$, and $\mathcal{E} \cong \mathcal{Q}$, which are used in the construction of the bielliptic curve $\mathcal{D}$ and in Section~\ref{sec:proof} to show that the Prym variety of $\mathcal{D}$ is two-isogenous to the Jacobian variety of $\mathcal{C}$.
\par The intersection $\mathcal{M}_3^h \cap \mathcal{M}_3^b$ is exactly the locus of curves with automorphism group $V_4$ (the Klein 4-group) inside the hyperelliptic locus. Such curves are usually called \emph{hyperelliptic curves with extra involutions}. In \cite{MR3118614} the locus  $\mathcal{M}_3^{b} \cap \mathcal{M}_3^h$ was explicitly described in terms of invariants.  We shall construct a curve $\mathcal{H}\in \mathcal{M}_3^{b} \cap \mathcal{M}_3^h$  by choosing 4 out  of the six Weierstrass points of $\mathcal{C}$ to be the images of 4 pairs of points on the curve $\mathcal{H}$ such that all eight Weierstrass points of $\mathcal{H}$ in the preimage are fixed under the hyperelliptic involution, and each pair is kept fixed by the bielliptic involution. For a genus-two curve $\mathcal{C}$ given as sextic $Y^2 = f_6(X,Z)$, a class in $\mathcal{M}_2(2)$, i.e., the moduli space of genus-two curves with level-two structure,  is given by the ordered tuple $(\lambda_1, \lambda_2, \lambda_3)$ after we sent the three remaining roots to $0, \infty, 1$. We then choose the points $(1,\lambda_1, \lambda_2, \lambda_3)$ to be the images of the eight ramification points of $\mathcal{H}$. 
\subsection{A normal form}
We assume that the smooth genus-two curve $\mathcal{C}$ in Proposition~\ref{thm:genus3-hyperelliptc-bielliptic} is in Rosenhain normal form, i.e., for $[X:Z:Y] \in \mathbb{P}(1,1,3)$ the curve is given by
\begin{equation}
\label{Eq:Rosenhain_g2}
  Y^2 = XZ  \prod_{i=0}^3 \big(X-\lambda_i Z\big)  \;,
\end{equation} 
with the hyperelliptic map $\pi^{\mathcal{C}}:  \mathcal{C} \to \mathbb{P}^1$ given by $[X:Z:Y]  \mapsto [X:Z]$. The hyperelliptic involution on $\mathcal{C}$ has the 6 fixed points  $\mathfrak{p}_i=[\lambda_i:1:0]$ for $i=0, \dots,3$, $\mathfrak{p}_4=[0:1:0]$, and $\mathfrak{p}_5=[1:0:0]$.  To simplify our discussion in the situation of pencils, we will use $\lambda_0$ rather than choosing $\lambda_0=1$. Since $\mathcal{C}$ is smooth, we will assume that $\lambda_i \not= 0$ and $\lambda_i \not= \lambda_j$ for $0 \le i <j \le 3$.  The lambdas are ratios of squares of even theta functions $\theta_i^2=\theta_i^2(0,\tau)$  with zero elliptic argument, modular argument $\tau \in \mathbb{H}_2/\Gamma_2(2)$, and $1\le i \le10$ where we are using the same standard notation for even theta functions as in \cites{MR0141643, MR0168805, Clingher:2018aa}. We have a choice of $6!=720$ such expressions. In each case, there is a ratio of \emph{squares} of theta functions such that $l^2=\lambda_0\lambda_1\lambda_3\lambda_3$.  In the following we use the convention from \cite{MR2367218}:
\begin{lemma} 
\label{lem:Picard}
If $\mathcal{C}$ is a genus-two curve with period matrix $\tau$ and non-vanishing discriminant, then $\mathcal{C}$  is equivalent to the curve in Equation~(\ref{Eq:Rosenhain_g2}) with Rosenhain parameters $\lambda_0, \lambda_1, \lambda_2, \lambda_3$  given by 
\begin{equation}\label{Picard}
\lambda_0=1 \,, \quad \lambda_1 = \frac{\theta_1^2\theta_3^2}{\theta_2^2\theta_4^2} \,, \quad \lambda_2 = \frac{\theta_3^2\theta_8^2}{\theta_4^2\theta_{10}^2}\,, \quad \lambda_3 = \frac{\theta_1^2\theta_8^2}{\theta_2^2\theta_{10}^2}\,.
\end{equation}
Conversely, given three distinct complex numbers $(\lambda_1, \lambda_2, \lambda_3)$ different from $0, 1, \infty$ the complex abelian surface $\operatorname{Jac} (\mathcal{C})$ has the period matrix $[\mathbb{I}_2 | \tau]$ where $\mathcal{C}$ is the genus-two curve with period matrix $\tau$. 
\end{lemma}
\begin{remark}
\label{rem:mu}
We define
\begin{equation}
 l =  \frac{\theta_1^2\theta_3^2\theta_8^2}{\theta_2^2\theta_4^2\theta_{10}^2}  \,,\quad
 m^{(1,2,3)} =  \frac{\theta_1\theta_3\theta^2_6}{\theta_2\theta_4\theta^2_5}\,, \quad
  m^{(2,1,3)} =  i \frac{\theta_3\theta_8\theta^2_6}{\theta_4\theta_{10}\theta^2_7}\,,\quad
 m^{(3,1,2)} =  \frac{\theta_1\theta_8\theta^2_6}{\theta_2\theta_{10}\theta^2_9}\,,
\end{equation}
such that $l^2=\lambda_0\lambda_1\lambda_3\lambda_3$ and $(m^{(i,j,k)})^2=(\lambda_i-\lambda_j)(\lambda_i-\lambda_k)/[(\lambda_0-\lambda_i)(\lambda_0-\lambda_j)]$ with $\{i, j ,k\} =\{1,2,3\}$. The latter identities follow from the well known Frobenius identities for theta functions; see \cites{MR3712162, MR3731039}.
\end{remark}
We define the subgroup $\Gamma_2(2n) = \lbrace M \in \Gamma_2 | \, M \equiv \mathbb{I} \mod{2n}\rbrace$ and Igusa's congruence subgroups $\Gamma_2(2n, 4n) = \lbrace M \in \Gamma_2(2n) | \, \operatorname{diag}(B) =  \operatorname{diag}(C) \equiv \mathbb{I} \mod{4n}\rbrace$ of the Siegel modular group $\Gamma_2= \{ M =\bigl(\begin{smallmatrix}A&B\\ C&D \end{smallmatrix} \bigr) \in \operatorname{Sp}_4(\mathbb{Z}) \}$ such that
\begin{equation}
 \Gamma_2/\Gamma_2(2)\cong S_6, \quad  \Gamma_2(2)/\Gamma_2(2,4)\cong (\mathbb{Z}/2\mathbb{Z})^4, \quad \Gamma_2(2,4)/\Gamma_2(4,8)\cong (\mathbb{Z}/2\mathbb{Z})^9,
\end{equation}
where $S_6$ is the permutation group of six elements. The following lemma was proven in \cite{Clingher:2018aa}:
\begin{lemma}
$\lambda_1, \lambda_2, \lambda_3$ are modular with respect to $\Gamma_2(2)$, $l$ is a modular with respect to $\Gamma_2(2,4)$, and $m^{(i,j,k)}$ is modular with respect to $\Gamma_2(4,8)$ for $\{i, j ,k\} =\{1,2,3\}$.
\end{lemma}
\qed
\par By Proposition~\ref{thm:genus3-hyperelliptc-bielliptic}, a hyperelliptic and bielliptic genus-three curve $\mathcal{H}$ in the preimage of the curve $\mathcal{C}$ defined in Equation~(\ref{Eq:Rosenhain_g2}), i.e., the parameters $\lambda_i$ in Equations~(\ref{Picard}) under the map  $\mathcal{M}_3^{b} \cap \mathcal{M}_3^h \to \mathcal{M}_2$ is given by the equation
\begin{equation}
\label{Eq:Rosenhain_g3}
 y^2 = \prod_{i=0}^3 \big(x^2-\lambda_i z^2\big) \,,
\end{equation} 
with $[x:z:y] \in \mathbb{P}(1,1,4)$. On $\mathcal{H}$, there are \emph{two} involutions: the hyperelliptic involution $\imath^{\mathcal{H}}: [x:z:y] \mapsto [x:z:-y]$ and the bielliptic involution $\tau^{\mathcal{H}}: [x:z:y] \mapsto [-x:z:y]$. 
\par It is easy to check that the composition $\tau^{\mathcal{H}} \circ \imath^{\mathcal{H}}$ is fixed-point-free. An unramified double cover $\pi^\mathcal{H}_\mathcal{C}:\mathcal{H}  \to \mathcal{C}$ is given by
\begin{equation}
\label{mapQC}
\begin{split}
  \pi^\mathcal{H}_\mathcal{C}: \quad [x:z:y] & \mapsto  [X:Z:Y]=[x^2:z^2:xyz] \;.
\end{split} 
\end{equation}
The images of the four pairs of hyperelliptic fixed points and the two pairs of bielliptic fixed points under $\pi^\mathcal{H}_\mathcal{C}$ are exactly the Weierstrass points of the genus-two curve $\mathcal{C}$. It is easily proved that every unramified  double cover of a hyperelliptic genus two curve is obtained in this way \cite{MR990136}*{p.~\!387}; in particular, the cover is always hyperelliptic.
\par The quotient genus-one curve $\mathcal{Q} =  \mathcal{H} /\langle\tau^{\mathcal{H}} \rangle$ obtained from the bielliptic involution is the quartic curve
\begin{equation}
\label{Eq:Rosenhain_g1}
  y^2 = \prod_{i=0}^3 \big(X-\lambda_i Z\big) \,,
\end{equation} 
with $[X:Z:y] \in \mathbb{P}(1,1,2)$, and the double cover $\pi^{\mathcal{H}}_{\mathcal{Q}}: \mathcal{H} \to \mathcal{Q}$ is given by
\[
 \pi^{\mathcal{H}}_{\mathcal{Q}}: \quad [ x:\pm z:y] = [- x:\mp z:y] \mapsto  [X:Z:y]=[x^2:z^2:y] \;.
 \]
The four branch points of $ \pi^{\mathcal{H}}_{\mathcal{Q}}$ are precisely the images of the bielliptic fixed points. The situation is summarized in Figure~\ref{Relations_Curves}.  Here, the map $\mathbb{P}^1 \to \mathbb{P}^1$ is given by $[x:z] \mapsto [X:Z]=[x^2:z^2]$. Moreover, in the introduction the genus-one curve in the bi-double cover is called $\mathcal{E}$. Here, it is called $\mathcal{Q}$ and we prove in Section~\ref{sec:proof} that it is isomorphic to a curve $\mathcal{E}$ with a certain given equation.
\begin{figure}[ht]
\centerline{
\xymatrix{
& \mathcal{H} 	\ar[d]^{\pi^{\mathcal{H}}_{\mathcal{Q}}} \ar[ld]_{\pi^{\mathcal{H}}_{\mathcal{C}}}   \ar[rd] \\
\mathcal{C}	\ar[rd]_{\pi^{\mathcal{C}}} &  \mathcal{Q}	\ar[d]^{\pi^{\mathcal{Q}}}  & \mathbb{P}^1 	\ar[ld] \\
& \mathbb{P}^1
}}
 \caption{\label{Relations_Curves}}
\end{figure}

\par We have the following:
\begin{proposition}
\label{prop:normal_form}
The quotient $\mathcal{Q} =  \mathcal{H} /\langle\tau^{\mathcal{H}} \rangle$ in Equation~(\ref{Eq:Rosenhain_g1}) of the hyperelliptic and bielliptic genus-three curve in Equation~(\ref{Eq:Rosenhain_g3}) is isomorphic to the elliptic curve
\begin{equation}\label{eq:EC}
 \mathcal{E}: \quad \rho^2 \eta = \xi \Big( \xi^2 - 2 b \xi \eta +  (b^2-4a^2)\,  \eta^2 \Big) \,,
\end{equation} 
with $[\xi:\eta:\rho] \in \mathbb{P}^2$ and coefficients
\begin{equation}
\label{eqn:coeffs_ab}
\begin{split}
 a =(\lambda_0-\lambda_1)(\lambda_2-\lambda_3) \,, \quad
 b = 4 \lambda_0 \lambda_1+ 4 \lambda_2 \lambda_3-2 \lambda_0 \lambda_2 - 2  \lambda_0 \lambda_3 -  2\lambda_1 \lambda_2  - 2 \lambda_1 \lambda_3 \,.
\end{split}
\end{equation}
The elliptic curve~(\ref{eq:EC}) has two-torsion points $[\xi:\eta:\rho] = [0:1:0], [b \pm 2a:1:0]$, and the neutral element $[0:0:1]$.
\end{proposition}
\begin{proof}
The proof follows by an explicit computation. 
\end{proof}
Later, we will also use the existence of certain rational points on $\mathcal{E}$ in Equation~(\ref{eq:EC}) that stems from the fact that $\mathcal{E}$ is isomorphic to the genus-one curve $\mathcal{Q} =  \mathcal{H} /\langle\tau^{\mathcal{H}} \rangle$ with four bielliptic branch points. We have the following:
\begin{lemma}\label{lem:adding}
On the elliptic curve $\mathcal{E}$ in Proposition~\ref{prop:normal_form}, there are the rational points $\mathsf{p}_1$ with coordinates given by
\begin{equation}
\begin{split}
   [\xi:\eta:\rho] & = \left[ 4(\lambda_0-\lambda_2)( \lambda_0-\lambda_3) : 1: 8(\lambda_0-\lambda_1)(\lambda_0-\lambda_2)( \lambda_0-\lambda_3)\right] \,,
\end{split} 
\end{equation}
and $\mathsf{p}_2$ with
\begin{equation}
\begin{split}
  [\xi:\eta:\rho]  = \left[ 4\lambda_0\lambda_1(\lambda_0-\lambda_2)( \lambda_0-\lambda_3) :\lambda_0^2:  8l(\lambda_0-\lambda_1)(\lambda_0-\lambda_2)( \lambda_0-\lambda_3)\right]\,.
\end{split} 
\end{equation}
Using the group law on $\mathcal{E}$, we obtain the rational points $2\mathsf{p}_1$ with coordinates
\begin{equation}
\label{eqn:2p1}
\begin{aligned}
  \xi  & =  (\lambda_0+\lambda_1-\lambda_2-\lambda_3)^2 \,,\quad  \eta = 1\,,\\
  \rho & = (\lambda_0+\lambda_1-\lambda_2-\lambda_3)(\lambda_0-\lambda_1-\lambda_2+\lambda_3)(\lambda_0-\lambda_1+\lambda_2-\lambda_3) \,,
\end{aligned} 
\end{equation}
and rational points $\mathsf{p}_1 \pm \mathsf{p}_2$ with coordinates
\begin{equation}
\label{eqn:p1+p2}
\begin{split}
  \xi  & = 4 ( \lambda_0\lambda_1+\lambda_2\lambda_3 \mp 2l)\,, \quad \eta = 1\,,\\
  \rho& = 8 \big( \pm l (\lambda_0 + \lambda_1+\lambda_2 + \lambda_3) - \lambda_0 \lambda_1\lambda_2 -\lambda_0 \lambda_1\lambda_3-\lambda_0 \lambda_2\lambda_3-\lambda_1 \lambda_2\lambda_3\big)\,.
\end{split}
\end{equation}
\end{lemma}
\begin{proof}
The points $\pm \mathsf{p}_1$ and $\pm \mathsf{p}_2$ are the images of the four branch points of $\pi^{\mathcal{H}}_{\mathcal{Q}}$, namely  $[X:Z:y]=[1:0:\pm1]$ and $[0:1:\pm l]$ on the genus-one curve in Equation~(\ref{Eq:Rosenhain_g1}), respectively. The rest of the proof follows by explicit computation.
\end{proof}
Moreover, we have the following:
\begin{proposition}
\label{prop:JacH}
Given a smooth genus-two curve $\mathcal{C}$, the hyperelliptic and bielliptic genus-three curve $\mathcal{H}$ in Equation~(\ref{Eq:Rosenhain_g3}) and the elliptic curve $\mathcal{E}$ in Equation~(\ref{eq:EC}) satisfy
\[
 \operatorname{Jac}(\mathcal{H}) \cong \operatorname{Prym}(\mathcal{H},\pi^{\mathcal{H}}_{\mathcal{E}}) \times \mathcal{E}\,, \qquad \operatorname{Prym}(\mathcal{H},\pi^{\mathcal{H}}_{\mathcal{E}}) \cong \operatorname{Jac}(\mathcal{C}) \,,
\]
where $\operatorname{Prym}(\mathcal{H},\pi^{\mathcal{H}}_{\mathcal{E}})$ is the Prym variety associated with $\pi^{\mathcal{H}}_{\mathcal{E}}$.
\end{proposition}
\begin{proof}
The involution $\tau^{\mathcal{H}}$ extends to the Jacobian variety $\operatorname{Jac}(\mathcal{H})$. Therefore, it contains two abelian sub-varieties, the elliptic curve $\mathcal{E}$ and the two-dimensional Prym-variety $\operatorname{Prym}(\mathcal{H},\pi^{\mathcal{H}}_{\mathcal{E}})$ which is anti-symmetric with respect to the extended involution. On the other hand, the \'etale double cover $\pi^\mathcal{H}_\mathcal{C}:\mathcal{H}  \to \mathcal{C}$ satisfies $\pi^\mathcal{H}_\mathcal{C} \circ \tau^{\mathcal{H}} = \imath^{\mathcal{C}}$, i.e., it is equivariant with respect to the bielliptic involution on $\mathcal{H}$ and the hyperelliptic involution on $\mathcal{C}$. The claim follows.
\end{proof}
\subsection{G\"opel groups and double covers}
We denote the space of two-torsion points on an abelian variety $\mathfrak{A}$ by $\mathfrak{A}[2]$. In the case of the Jacobian of a genus-two curve, every nontrivial two-torsion point can be expressed using differences of Weierstrass points of $\mathcal{C}$. Concretely, the sixteen order-two points of $\operatorname{Jac}(\mathcal{C})[2]$ are obtained using the embedding of the curve into the connected component of the identity in the Picard group, i.e., $\mathcal{C} \hookrightarrow \operatorname{Jac}(\mathcal{C}) \cong \operatorname{Pic}^0(\mathcal{C})$ with $\mathfrak{p} \mapsto [\mathfrak{p} -\mathfrak{p}_5]$. We obtain 15 elements $\mathsf{P}_{ij} \in \operatorname{Jac}(\mathcal{C})[2]$ with $0 \le i < j \le 5$ as
 \begin{equation}
 \label{oder2points}
  \mathsf{P}_{i5} = [ \mathfrak{p}_i - \mathfrak{p}_5] \; \text{for $0 \le i < 5$}\,, \qquad 
  \mathsf{P}_{ij}=[ \mathfrak{p}_i + \mathfrak{p}_j - 2 \, \mathfrak{p}_5]  \; \text{for $0 \le i < j \le 4$}\,, 
 \end{equation}
and set $\mathsf{P}_{0} = \mathsf{P}_{55}= [0]$. For $\{i,j,k,l,m,n\}=\{0, \dots, 5\}$,  the group law on $\operatorname{Jac}(\mathcal{C})[2]$ is given by the relations
 \begin{equation}
 \label{group_law}
    \mathsf{P}_0 +  \mathsf{P}_{ij} =  \mathsf{P}_{ij}, \quad  \mathsf{P}_{ij} +  \mathsf{P}_{ij} =  \mathsf{P}_{0}, \quad 
    \mathsf{P}_{ij} + \mathsf{P}_{kl} =  \mathsf{P}_{mn}, \quad \mathsf{P}_{ij} +
    \mathsf{P}_{jk} =  \mathsf{P}_{ik}.
 \end{equation}
\par The  space $A[2]$ of two-torsion points on an abelian variety $\mathfrak{A}$ admits a symplectic bilinear form, called the \emph{Weil pairing}. The Weil pairing is induced by the pairing
\[
 \langle [ \mathfrak{p}_i - \mathfrak{p}_j  ] ,[ \mathfrak{p}_k - \mathfrak{p}_l] \rangle =\#\{  \mathfrak{p}_{i}, \mathfrak{p}_{j}\}\cap \{ \mathfrak{p}_{k}, \mathfrak{p}_{l}\} \mod{2},
\]
We call a two-dimensional, maximal isotropic subspace of $A[2]$ with respect to the Weil pairing, i.e., a subspace such that the symplectic form vanishes on it, a \emph{G\"opel group} in $A[2]$. Such a maximal subgroup is isomorphic to $( \mathbb{Z}/2\mathbb{Z})^2$.
\par We give the following characterization of the choices involved in our construction of the curves $\mathcal{H}$ and $\mathcal{E}$ in Figure~\ref{Relations_Curves}:
\begin{proposition}
For a smooth genus-two curve $\mathcal{C}$, there are 15 inequivalent hyperelliptic and bielliptic genus-three curves $\mathcal{H}_{ij}$ for $0 \le i < j \le 5$ that are unramified double covers of $\mathcal{C}$. The double covers $\mathcal{H}_{ij} \to \mathcal{C}$ are in one-to-one correspondence with non-trivial elements of $\mathsf{P}_{ij} \in \operatorname{Jac}(\mathcal{C})[2]$. Moreover, isomorphisms $\mathcal{Q}_{ij} \cong \mathcal{E}$ -- understood as isomorphisms between genus-one curves with marked Weierstrass points -- are in one-to-one correspondence with G\"opel groups $G \subset \operatorname{Jac}(\mathcal{C})[2]$ such that $\mathsf{P}_{ij} \in G$.
\end{proposition}
\begin{proof}
We constructed the curve $\mathcal{H}\in \mathcal{M}_3^{b} \cap \mathcal{M}_3^h$  by choosing 4 out  of the six Weierstrass points of $\mathcal{C}$ to be the images of 4 pairs of points on the curve $\mathcal{H}$ such that all eight Weierstrass points of $\mathcal{H}$ in the preimage are fixed under the hyperelliptic involution, and each pair is kept fixed by the bielliptic involution. That is, the construction of $\mathcal{H}$ was determined by $\{\mathfrak{p}_4, \mathfrak{p}_5\}$. The unordered pair represents a divisor class $[\mathfrak{p}_4 - \mathfrak{p}_5]$ with $2[\mathfrak{p}_4- \mathfrak{p}_5]\equiv 0$. Therefore, $[\mathfrak{p}_4 - \mathfrak{p}_5] \in \operatorname{Jac}(\mathcal{C})[2]$. One checks that the resulting curve for any two different Weierstrass points also has a different $j$-invariant. It is easy to see, that the elliptic curve $\mathcal{E}$ together with the set of two-torsion points $\{[0:1:0], [B \pm 2A:1:0]\}$, depends on the partition of Weierstrass points of $\mathcal{E}$ or, equivalently, on a partition of the Weierstrass points of $\mathcal{C}$. From every partition of Weierstrass points, we obtain three elements $\mathsf{P}_{ij}, \mathsf{P}_{kl}, \mathsf{P}_{mn} \in \operatorname{Jac}(\mathcal{C})[2]$ with $\{i,j,k,l,m,n\}=\{0, \dots, 5\}$, each generating a $(\mathbb{Z}/2\mathbb{Z})$ subgroup.  Because the only relation between these classes is given by $\mathsf{P}_{ij} + \mathsf{P}_{kl} + \mathsf{P}_{mn} =0$, the classes generate a subgroup in $\operatorname{Jac}(\mathcal{C})[2]$ isomorphic to $(\mathbb{Z}/2\mathbb{Z})^2$. Because the pairs of Weierstrass points are all disjoint, the associated subgroup is in fact isotropic with respect to the Weil pairing.
\end{proof}
\begin{remark}
\label{rem:element}
For the smooth genus-two curve $\mathcal{C}$ in Equation~(\ref{Eq:Rosenhain_g2}), the hyperelliptic and bielliptic genus-three curve $\mathcal{H}$ in Figure~\ref{Relations_Curves} corresponds to the divisor $\mathsf{P}_{45} \in  \operatorname{Jac}(\mathcal{C})[2]$. 
\end{remark}
\subsection{Pencils of hyperelliptic curves}\label{pencils}
We start with the hyperelliptic and bielliptic genus-three curve $\mathcal{H}$ in the preimage of $\mathcal{M}_3^{b} \cap \mathcal{M}_3^h \to \mathcal{M}_2$ given by Equation~(\ref{Eq:Rosenhain_g3}).
The automorphism $\imath^\mathcal{H} \times  \imath^\mathcal{H}$ of $\mathcal{H} \times \mathcal{H}$ induces an automorphism on the symmetric square $\operatorname{Sym}^2(\mathcal{H})$ which by a slight abuse of notation we will denote by $\imath^\mathcal{H} \times  \imath^\mathcal{H}$ as well. We have the following:
\begin{lemma}
On the variety $\mathscr{H}=\operatorname{Sym}^2(\mathcal{H})/\langle \imath^\mathcal{H} \times  \imath^\mathcal{H} \rangle$, there is a pencil over $\mathbb{P}^1 \ni [s_0:s_1]$ of hyperelliptic and bielliptic genus-three curves $\mathscr{H}_{[s_0:s_1]}$ given by
\begin{equation}
\label{eqn:hyperpencil}
 \mathscr{H}_{[s_0:s_1]}: \quad y^2 =\prod_{i=0}^3  \left( x^2 - \frac{(s_0+ \lambda_i s_1)^2}{\lambda_i}  z^2\right),
\end{equation}
with $[x:z:y] \in \mathbb{P}(1,1,4)$. In particular, the central fiber over $[s_0:s_1]=[0:1]$ is isomorphic to $\mathcal{H}$.
\end{lemma}
\begin{proof}
If we set $y=y^{(1)}y^{(2)}/l$, $x=x^{(1)} z^{(2)}+x^{(2)} z^{(1)}$, and $s_0 z=x^{(1)}x^{(2)}$ and $s_1 z=z^{(1)}z^{(2)}$, Equation~(\ref{eqn:hyperpencil}) becomes the product of two copies of Equation~(\ref{Eq:Rosenhain_g3}). Since the variables are invariant under the product of the hyperelliptic involutions on each copy of $\mathcal{H}$, the statements follows.
\end{proof}
We make the following:
\begin{remark}
The bielliptic and hyperelliptic involution on the curve in Equation~(\ref{Eq:Rosenhain_g3}) both lift to involutions on the fibers of the pencil~(\ref{eqn:hyperpencil}).
\end{remark}
We define two pencils $\mathscr{Q}_{[s_0:s_1]}$ and $\mathscr{C}_{[s_0:s_1]}$ of genus-one and genus-two curves over $\mathbb{P}^1 \ni [s_0:s_1]$, respectively. They are given by
\begin{equation}
\label{Eq:Rosenhain_g1_pencil}
\begin{split}
  \mathscr{Q}_{[s_0:s_1]}: \quad y^2 = \prod_{i=0}^3 \left(X- \frac{(s_0+ \lambda_i s_1)^2}{\lambda_i}  Z\right) \,, \\
  \mathscr{C}_{[s_0:s_1]}: \quad Y^2 = X Z \prod_{i=0}^3 \left(X- \frac{(s_0+ \lambda_i s_1)^2}{\lambda_i}  Z\right),
 \end{split}
\end{equation} 
with $[X:Z:y] \in \mathbb{P}(1,1,2)$ and $[X:Z:Y] \in \mathbb{P}(1,1,3)$.  The pencils are constructed such that the diagram of Figure~\ref{Relations_Curves} holds for every fiber over $[s_0:s_1]$, and central fibers over $[s_0:s_1]=[0:1]$ are exactly the curves $\mathcal{H}$, $\mathcal{C}$, and $\mathcal{Q}$, respectively. That is, we have
\[
 \mathscr{Q}_{[0:1]} = \mathcal{Q} \,, \qquad \mathscr{C}_{[0:1]} = \mathcal{C} \,, \qquad \mathscr{H}_{[0:1]} = \mathcal{H} \,.
\] 
We define another pencil $\mathscr{Q}'_{[t_0:t_1]}$ of genus-one curves over $\mathbb{P}^1 \ni [t_0:t_1]$ given by
\begin{equation}
\label{Eq:Rosenhain_g1_pencil_Kummer}
\begin{split}
  \mathscr{Q}'_{[t_0:t_1]}: \quad Y^2 = t_0 t_1 \prod_{i=0}^3 \left(x- \frac{t_0 + \lambda_i^2 t_1}{\lambda_i}  z\right) \,, 
\end{split}
\end{equation} 
with $[x:z:Y] \in \mathbb{P}(1,1,2)$, and a two-to-one map $\mathscr{Q} \to \mathscr{Q}'$ by setting
\begin{equation}
\label{eqn:psi}
\begin{split}
 \Big( [s_0:s_1], [X:Z:y]\Big) \mapsto  & \Big([t_0:t_1], [x:z:Y]\Big)= \Big([s_0^2:s_1^2], [X-2s_0s_1Z:Z:s_0s_1y]\Big)\,.
\end{split}
\end{equation}
From these pencils, we obtain the total spaces of fibrations (without multiple fibers)
\[
\mathscr{C} = \coprod_{[s_0:s_1]\in \mathbb{P}^1} \mathscr{C}_{[s_0:s_1]}\,, \qquad
 \mathscr{Q} = \coprod_{[s_0:s_1]\in \mathbb{P}^1} \mathscr{Q}_{[s_0:s_1]}\,, \qquad 
  \mathscr{Q}' = \coprod_{[t_0:t_1]\in \mathbb{P}^1} \mathscr{Q}'_{[t_0:t_1]}\,.
\] 
In the next section, we will show that the total space $\mathscr{Q}$ and $\mathscr{Q}'$ are in fact singular models for certain Kummer surfaces. Singular fibers for pencils of genus-two curves were classified by Namikawa and Ueno in \cites{MR0319996,MR0369362,MR0384794,MR0384795}. We have the following immediate:
\begin{proposition}
The pencil  $\mathscr{C} \to \mathbb{P}^1$  has twelve singular fiber of Namikawa-Ueno type $I_{2-0-0}$ and four singular fibers of type $I_{4-0-0}$ with modulus point $\bigl( \begin{smallmatrix} \tau_1 & * \\ * &\infty \end{smallmatrix} \bigr)$.
\end{proposition}
\par Comparing Equation~(\ref{Eq:Rosenhain_g1_pencil}) with Equation~(\ref{Eq:Rosenhain_g1}), we introduce the functions
\begin{equation}
\label{eqn:params_Lambdas}
 \Lambda_i(s_0,s_1) =   \frac{(s_0+ \lambda_i s_1)^2}{\lambda_i}  \,, \qquad  L(s_0,s_1) = \frac{\prod_{i=0}^3 (s_0+ \lambda_i s_1)}{l} \,,
\end{equation}
for $0 \le i \le 3$ such that $L^2=\Lambda_0\Lambda_1\Lambda_2\Lambda_3$.
Using Proposition~\ref{prop:normal_form} we have the immediate:
\begin{corollary}
\label{cor:pencilQ}
The pencil $\mathscr{Q}$ is isomorphic to the elliptic fibration $\pi: \mathscr{E} \to \mathbb{P}^1$,
\begin{equation}
\label{eqn:Kummer_EF}
 \mathscr{E}_{[s_0:s_1]}: \quad \rho^2 \eta = \xi \Big( \xi^2 - 2 B(s_0,s_1)\, \xi \eta +  \big(B^2(s_0,s_1)-4A^2(s_0,s_1)\big)\,  \eta^2 \Big) \,,
\end{equation}
with section $\mathsf{O}: [\xi:\eta:\rho] = [0:0:1]$ and
\begin{equation}
\label{eqn:coeffs_AB}
\begin{split}
 A(s_0,s_1) & =\big(\Lambda_0(s_0,s_1)-\Lambda_1(s_0,s_1)\big)\big(\Lambda_2(s_0,s_1)-\Lambda_3(s_0,s_1)\big) \,,\\
 B(s_0,s_1) & = 4 \Lambda_0(s_0,s_1) \, \Lambda_1(s_0,s_1)+ 4 \Lambda_2(s_0,s_1) \, \Lambda_3(s_0,s_1)-2 \Lambda_0(s_0,s_1) \, \Lambda_2(s_0,s_1) \\
 &  - 2  \Lambda_0(s_0,s_1)  \, \Lambda_3(s_0,s_1)  -  2\Lambda_1(s_0,s_1) \, \Lambda_2(s_0,s_1)   - 2 \Lambda_1(s_0,s_1) \, \Lambda_3(s_0,s_1)  \,.
\end{split}
\end{equation}
In particular, $A$ and $B$ are even polynomials of degree four, such that there are no singular fibers over $[s_0:s_1]=[0:1], [1:0]$, and
\[
 l^2  A(s_0,s_1) = A(l s_1,s_0)\,, \quad l^2  B(s_0,s_1) =  B(l s_1,s_0)\,.
\]
\end{corollary}
\qed
\par Similarly, we obtain the following:
\begin{corollary}
\label{cor:pencilQp}
The pencil $\mathscr{Q}'$ is isomorphic to the elliptic fibration $\pi': \mathscr{E}' \to \mathbb{P}^1$,
\begin{equation}
\label{eqn:Ep}
 \mathscr{E}'_{[t_0:t_1]}: \quad \rho^{\prime 2} \eta' = \xi' \Big( \xi^{\prime 2} - 2 B'(t_0,t_1)\, \xi' \eta' +  \big(B^{\prime 2}(t_0,t_1)-4A^{\prime 2}(t_0,t_1)\big)\,  \eta^{\prime 2} \Big) \,,
\end{equation}
with section $\mathsf{O}': [\xi':\eta':\rho'] = [0:0:1]$, and polynomials
\[
  A'(t_0,t_1) = t_0 t_1 \, A\big(\sqrt{t_0}, \sqrt{t_1}\big) \,, \quad  B'(t_0,t_1) = t_0 t_1 \, B\big(\sqrt{t_0}, \sqrt{t_1}\big)\,,
\]
which are well defined polynomials because of Corollary~\ref{cor:pencilQ}. Moreover, the two-to-one map in Equation~(\ref{eqn:psi}) extends to a  double cover $\psi: \mathscr{E} \to \mathscr{E}'$ given by
\begin{equation}
\label{eqn:psi2}
\begin{split}
\psi: \; \Big( [s_0:s_1], [\xi:\eta:\rho]\Big) \mapsto  & \Big([t_0:t_1],[\xi':\eta':\rho'] \Big)
= \Big([s_0^2:s_1^2], [s^2_0 s^2_1\xi:\eta:s^3_0s^3_1\rho]\Big)\,.
\end{split}
\end{equation}
\end{corollary}
\begin{proof}
Making the point $[x:z:Y]=[t_0+ \lambda^2_0 t_1:\lambda_0:0]$ the neutral element of an elliptic curve, the point $[x:z:Y]=[t_0 + \lambda^2_1 t_1:\lambda_1:0]$ a two-torsion point, we can bring Equation~(\ref{Eq:Rosenhain_g1_pencil_Kummer}) into the normal form in Equation~(\ref{eqn:Ep}).
\end{proof}
\subsection{Relation between elliptic pencils and Kummer surfaces}
\label{ssec:KummerPencils}
Corollary~\ref{cor:pencilQ} proves that the elliptic fibration with section $(\pi: \mathscr{E} \to \mathbb{P}^1,\mathsf{O})$ is equivalent to the pencil $\mathscr{Q}$ given by
\begin{equation}
\label{Eq:Rosenhain_g1_pencil_b}
\begin{split}
  \mathscr{Q}_{[s_0:s_1]}: \quad y^2 = \prod_{i=0}^3 \Big(X- \Lambda_i(s_0,s_1) \,  Z\Big) \,.
 \end{split}
\end{equation} 
Therefore, Proposition~\ref{prop:normal_form} and Lemma~\ref{lem:adding} can be applied in \emph{each} fiber by replacing $\lambda_i \mapsto \Lambda_i(s_0,s_1)$ for $0 \le i \le 3$ and $l \mapsto L(s_0,s_1)$. Three two-torsion sections for the elliptic fibration $(\pi ,\mathsf{O})$ are given by $\mathsf{T}_1, \mathsf{T}_2, \mathsf{T}_3: [\xi:\eta:\rho] = [0:1:0], [B \pm 2A:1:0]$. Two non-torsion sections -- which by slight abuse of notation we will still denote by $\mathsf{p}_1, \mathsf{p}_2$ -- are obtained by assigning the points $\mathsf{p}_1$ and $\mathsf{p}_2$ in Lemma~\ref{lem:adding} in each fiber. The existence of a third non-torsion section $\mathsf{p}_3$ in the pencil is easily verified by assigning the point $[X:Z:y]=[4s_0s_1: 1:  L(-s_0,s_1)]$ in each fiber $\mathscr{Q}_{[s_0:s_1]}$ and then converting to coordinates $[\xi:\eta:\rho]$. 
\par We also define sections $\{\mathsf{S}_1, \mathsf{S}_2, \mathsf{S}_3\}$ as follows:
\begin{equation}
\scalemath{0.8}{
\begin{array}{c|l}
\text{sec.} & [\xi:\eta:\rho] \\[0.2em]
\hline
& \\[-0.9em]
\mathsf{S}_1 &\big[ 
4\lambda_0\lambda_1 \prod_{i=2}^3 (\lambda_i-\lambda_0)(s_0^2-\lambda_0\lambda_i s_1^2) :
8 \prod_{i=1}^3 (\lambda_i-\lambda_0)(s_0^2-\lambda_0\lambda_i s_1^2) : \lambda_0^2l^2\big] \\[0.2em]
\mathsf{S}_2 & \big[ 
4 l  \prod_{i=2}^3 (\lambda_i-\lambda_1)(s_0^2-\lambda_i \lambda_1 s_1^2) :
8 \prod_{i=1}^3 (\lambda_0-\lambda_i) \prod_{1\le j < k \le3} (s_0^2-\lambda_j  \lambda_k s_1^2) : l^3\big] \\[0.2em]
\mathsf{S}_3 & \big[ 
 4l s_0s_1  \prod_{i=0}^1 \prod_{j=2}^3 (s_0^2-\lambda_i \lambda_j s_1^2) :
-8 \prod_{i=1}^3 (s_0^2-\lambda_0  \lambda_i s_1^2) \prod_{1\le j < k \le3}  (s_0^2-\lambda_j  \lambda_k s_1^2) : l^3 s_0^3s_1^3\big]\\[0.2em]
\end{array}}
\end{equation}
It follows:
\begin{proposition}
\label{prop:Garbagnati2}
For the elliptic fibration with section $(\pi: \mathscr{E} \to \mathbb{P}^1,\mathsf{O})$ in Equation~(\ref{eqn:Kummer_EF}) and in Proposition~\ref{prop:Garbagnati}, the sections $\{\mathsf{O}, \mathsf{T}_1, \mathsf{T}_2, \mathsf{T}_3, \mathsf{S}_1, \mathsf{S}_2, \mathsf{S}_3\}$  are generators of the Mordell-Weil group $\operatorname{MW}(\pi,\mathsf{O})\cong (\mathbb{Z}/2\mathbb{Z})^2 \oplus \langle 1 \rangle^{\oplus 2} \oplus \langle 2 \rangle$. In particular, we have
\begin{equation}
\label{eqn:sections_EF}
 \mathsf{p}_1 = \mathsf{S}_1 \,, \quad \mathsf{p}_2 = -\mathsf{S}_2 + \mathsf{S}_3 \,, \quad \mathsf{p}_3 = \mathsf{S}_2 + \mathsf{S}_3 \,.
\end{equation}
\end{proposition}
\begin{proof}
Using our previous definitions and results in Lemma~\ref{lem:Picard} we set
\begin{equation}
 \mu = \frac{\theta_1\theta_3\theta_8}{\theta_2\theta_4\theta_{10}} \,,
\end{equation}
such that $\mu^4=\lambda_0\lambda_1\lambda_2\lambda_3$. If we use the affine chart given by $s_0=1$, $s_1 = s/\mu$, $\xi=X$, $\eta=1$, $\rho=Y$ in Corollary~\ref{cor:pencilQ}, we obtain the Weierstrass model~(\ref{eqn:EFS}) and two involutions $s\mapsto -s$ and $s\mapsto 1/s$; see Equations~(\ref{eqn:involutions}). In fact, the coefficients in Equation~(\ref{eqn:EFS}) are obtained from Corollary~\ref{cor:pencilQ} by setting
\begin{equation}
\label{eqn:coeffs_affine}
 A(s) := A\left(s_0=1, \ s_1 = \frac{s}{\mu}\right) \,,\quad B(s) := B\left(s_0=1, \ s_1 = \frac{s}{\mu}\right) \,.
\end{equation}
A straight forward computation shows that sections $\{\mathsf{O}, \mathsf{T}_1, \mathsf{T}_2, \mathsf{T}_3, \mathsf{S}_1, \mathsf{S}_2, \mathsf{S}_3\}$ have exactly the intersection and height pairings given by Table~\ref{tab:height} and form a basis of the Mordell-Weil group $\operatorname{MW}(\pi,\mathsf{O})$.  Using the elliptic-curve group law in each fiber, one finds that the sections $\{\mathsf{S}_1, \mathsf{S}_2, \mathsf{S}_3\}$ satisfy relations~(\ref{eqn:sections_EF}).
\end{proof}
\par We turn to the symmetric square $\operatorname{Sym}^2(\mathcal{C})$ associated with a smooth genus-two curve $\mathcal{C}$.  The automorphism $\imath^\mathcal{C} \times  \imath^\mathcal{C}$ of $\mathcal{C} \times \mathcal{C}$ again induces an automorphism on the symmetric square $\operatorname{Sym}^2(\mathcal{C})$ which by a slight abuse of notation we will denote by $\imath^\mathcal{C} \times  \imath^\mathcal{C}$ as well. The variety $\operatorname{Sym}^2(\mathcal{C})/\langle \imath^\mathcal{C} \times  \imath^\mathcal{C} \rangle$ admits a birational model that can be easily derived: in terms of the variables $z_1=Z^{(1)}Z^{(2)}$, $z_2=X^{(1)}Z^{(2)}+X^{(2)}Z^{(1)}$, $z_3=X^{(1)}X^{(2)}$, and $z_4=Y^{(1)}Y^{(2)}$  with $[z_1:z_2:z_3:z_4] \in \mathbb{P}(1,1,1,3)$, it is given by the equation
\begin{equation}
\label{kummer_middle}
  z_4^2 = z_1 z_3  \prod_{i=0}^3 \big( \lambda_i^2 \, z_1  -  \lambda_i \, z_2 +  z_3 \big) \;.
\end{equation}
\begin{definition}
The hypersurface in $\mathbb{P}(1,1,1,3)$ given by Equation~(\ref{kummer_middle}) is called Shioda sextic and was described in \cite{MR2296439}.
\end{definition}
\par One easily checks the following:
\begin{lemma}
\label{lem:Shioda}
The Shioda sextic in Equation~(\ref{kummer_middle}) is birational to the Kummer surface $\operatorname{Kum}(\operatorname{Jac}\mathcal{C})$ associated with the Jacobian $\operatorname{Jac}(\mathcal{C})$ of a genus-two curve~$\mathcal{C}$ in Rosenhain normal form~(\ref{Eq:Rosenhain_g2}).
\end{lemma}
\begin{remark}
\label{rem:6lines}
Equation~(\ref{kummer_middle}) defines a double cover of $\mathbb{P}^2 \ni [z_1:z_2:z_3]$ branched along six lines given by
\begin{equation}
\label{eqn:6lines}
    \lambda_i^2 \, z_1  -  \lambda_i \, z_2 +  z_3 =0 \quad \text{with $0\le i \le 3$}\,, \quad  z_1=0, \quad z_3=0,  \quad
 \end{equation}
The six lines are tangent to the common conic $\mathcal{K}: z_2^2 - 4 \, z_1 z_3=0$. Conversely, any six lines tangent to a common conic can always be brought into the form of Equations~\ref{eqn:6lines}. A picture is provided in Figure~\ref{fig:6Lines}.
\end{remark}
\begin{figure}[ht]
\scalemath{0.8}{
$$
  \begin{xy}
   <0cm,0cm>;<1.5cm,0cm>:
    (2,0.75)*++!D\hbox{$\mathcal{K}=0$},
    (.5,0.18);(3.5,0.18)**@{-},
    (0.5,2);(1.6,0)**@{-},
    (3.1,2);(2.6,0)**@{-},
    (0.7,.5);(2,2.5)**@{-},
    (0.5,1.82);(3.5,1.82)**@{-},
    (2,2.7);(3.45,0)**@{-},
    (2,1)*\xycircle(.8,.8){},
  \end{xy}
  $$}
\caption{Double cover branched along reducible sextic}
\label{fig:6Lines}
\end{figure}
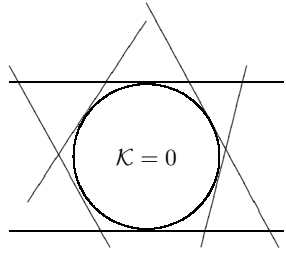
\par Equation~(\ref{kummer_middle}) is birationally equivalent to Equation~(\ref{Eq:Rosenhain_g1_pencil_Kummer}) as can be seen by setting
\[
  [z_1:z_2:z_3:z_4] = [ t_1 z \, : \, x \, : \, t_0 z \, : \, l zY ] \,,
\]
and is in turn is equivalent to the elliptic fibration with section $(\pi': \mathscr{E}' \to \mathbb{P}^1,\mathsf{O}')$ in Corollary~\ref{cor:pencilQp}. We make the following:
\begin{remark}
\label{rem:Kummer}
The elliptic fibration with section $(\pi': \mathscr{E}' \to \mathbb{P}^1,\mathsf{O}')$ has six singular fibers of Kodaira type $I_2$, two singular fibers of type $I_0^*$ (over $[t_0:t_1]=[0:1],[1:0]$), and a Mordell-Weil group $\operatorname{MW}(\pi',\mathsf{O}')\cong (\mathbb{Z}/2\mathbb{Z})^2 \oplus \langle 1 \rangle$. The elliptic fibration is induced by a pencil of lines in $\mathbb{P}^2$ passing through one of intersection point between two lines in Figure~\ref{fig:6Lines}. This is precisely the fibration on $\operatorname{Kum}(\operatorname{Jac}\mathcal{C})$ described in \cite{MR2296439}. 
\end{remark}
We have established that $\mathscr{E}'$ defines a pencil on $\operatorname{Sym}^2(\mathcal{C})/\langle \imath^\mathcal{C} \times  \imath^\mathcal{C} \rangle$. The minimal resolution of $\mathscr{E}'$ is the Kummer surface $\operatorname{Kum}(\operatorname{Jac}\mathcal{C})$ associated with the Jacobian $\operatorname{Jac}(\mathcal{C})$ of a smooth genus-two curve~$\mathcal{C}$. The involution $-\mathbb{I}$ on $\operatorname{Jac}(\mathcal{C})$ restricts to the hyperelliptic involution on each factor of $\mathcal{C}$ in $\operatorname{Sym}^2(\mathcal{C})$. The Weierstrass model in Remark~\ref{rem:Kummer} with two singular fibers of Kodaira type $I_0^*$ over $[t_0:t_1]=[0:1]$ and $[t_0:t_1]=[1:0]$ extends to an elliptic fibration with section on the Kummer surface and two reducible fibers of type $D_4$. We then have the following:
\begin{lemma}
\label{lem:branching}
The map $\psi: \mathscr{E} \to \mathscr{E}'$ in Equation~(\ref{eqn:psi2}) extends to rational double cover between the minimal resolutions $\widehat{\psi}: \widehat{\mathscr{E}} \dasharrow \widehat{\mathscr{E}}' \cong \operatorname{Kum}(\operatorname{Jac}\mathcal{C})$ that is branched along the eight non-central components of the two reducible fibers of type $D_4$.
\end{lemma}
\begin{proof} The proof is straight forward since the rational map is explicitly given.
\end{proof}
Mehran proved in \cite{MR2804549} that there are 15 distinct isomorphism classes of rational double covers of $\operatorname{Kum}(\operatorname{Jac} \mathcal{C})$ and computed the fifteen even eights (up to taking complements) on the Kummer surface $\operatorname{Kum}(\operatorname{Jac} \mathcal{C})$ that give rise to all distinct 15 isomorphism classes of rational double  covers \cite{MR2804549}*{Prop.~4.2}. An \emph{even eight} is an even set (as defined in Section~\ref{sec2}) of eight exceptional curves. Each even eight is enumerated by points $\mathsf{P}_{ij} \in \operatorname{Jac}(\mathcal{C})[2]$ with $0 \le i < j \le5$, and given as a sum
\[
 \Delta_{ij} = K_{0i} + \dots + \widehat{K_{ij}} + \dots + K_{i5} + K_{0j} + \dots + \widehat{K_{ij}} + \dots + K_{j5} \;,
 \]
where $K_{00}=0$, and $K_{ij}$ are the exceptional divisors obtained by resolving the nodes $p_{ij}$, i.e., the images of the points $\mathsf{P}_{ij}$, and the hat indicates divisors that are not part of the even eight. Moreover, Mehran proved that every rational map $\psi_{\Delta}: \operatorname{Kum}(\mathfrak{A}) \dashrightarrow \operatorname{Kum}(\operatorname{Jac} \mathcal{C})$ from a $(1,2)$-polarized to a principally polarized Kummer surface is induced by an isogeny $\Psi_{\Delta}:  \mathfrak{A} \to \operatorname{Jac}(\mathcal{C})$ of abelian surfaces of  degree two~\cite{MR2804549}, and that all inequivalent $(1,2)$-polarized abelian surfaces $\mathfrak{A}$ are obtained in this way. We have the following:
\begin{proposition}
\label{prop:Kum}
There exists an abelian surface $\mathfrak{A}$ with a polarization of type $(1,2)$ such that the variety $\mathscr{E}$ is birational to the Kummer surface $\operatorname{Kum}(\mathfrak{A})$. In particular, the map $\psi$ is induced by an isogeny $\Psi: \mathfrak{A} \to \operatorname{Jac}(\mathcal{C})$  of degree two.
\end{proposition}
\begin{proof}
It was shown  in \cites{Clingher:2017aa,MR3995925} that the eight non-central components of the two reducible fibers of type $D_4$ form an even eight on $ \operatorname{Kum}(\operatorname{Jac}\mathcal{C})$. In fact, the sum of the components in the even eight label by $p_{45}$ that forms the ramification locus of $\widehat{\psi}$ is given by
 \[
 \Delta_{45} = K_{04} +  K_{14} + K_{24} + K_{34} + K_{05} + K_{15}+ K_{25} + K_{35}  \,.
\] 
The result then follows from Lemma~\ref{lem:branching} and \cite{MR2804549}*{Prop.~5.1}.
\end{proof}
\begin{remark}
The construction in Proposition~\ref{prop:Kum} was based on a double cover branched along the even eight $\Delta_{45}$ labelled by the point $\mathsf{P}_{45} \in \operatorname{Jac}(\mathcal{C})[2]$. This is in agreement with the construction of the hyperelliptic and bielliptic genus-three curves $\mathcal{H}$ in Figure~\ref{Relations_Curves} which was based on the divisor $\mathsf{P}_{45} \in  \operatorname{Jac}(\mathcal{C})[2]$; see Remark~\ref{rem:element}.
\end{remark}
We make the following crucial remark:
\begin{remark}
The involution $-\mathbb{I}$ on the abelian surface $\mathfrak{A}$ with $(1,2)$-polarization does \underline{not} restrict to the bielliptic involution on each factor of $\mathcal{H}$ in $\mathscr{H}$. Therefore, the generic Prym variety associated with the bielliptic quotient map $\mathscr{H}_{[s_0:s_1]} \to \mathscr{Q}_{[s_0:s_1]}$ of a general fiber is not isomorphic to $\mathfrak{A}$. Instead it is isomorphic to $\operatorname{Jac}(\mathcal{C})$ by Proposition~\ref{prop:JacH} which is only two-isogenous to $\mathfrak{A}$ by Proposition~\ref{prop:Kum}.
\end{remark}
\section{Plane  bielliptic curves}
\label{sec4}
In this section we provide a geometric characterization of plane bielliptic genus-three curves $\mathcal{D}$ and their bielliptic quotients. The precise characterization of the associated branch loci turns out to be critical to relate the bielliptic genus-three curve $\mathcal{D}$ to a genus-two curve $\mathcal{C}$ such that the Prym variety of the former is isogenous to the Jacobian variety of the latter.
\par Let $\mathcal{D}$ be a bielliptic curve.  Then there is an involution $\tau \in \operatorname{Aut} (\mathcal{D})$ such that $\mathcal{D}/\langle \tau \rangle$ is a genus one curve.  For $g=3$ there are two loci in the moduli space $\mathcal{M}_3$ such that the automorphism group has precisely order two, namely the hyperelliptic locus $\mathcal{M}_3^h$ and the $\mathcal{M}_3^b$ of dimension five and four respectively. In \cite{MR1816214} a normal form for bielliptic genus-three curves was determined. From \cite{kyoto}*{Table~1} we see that a generic curve $[\mathcal{D}] \in \mathcal{M}_3^b$ has a degree two cover $\pi^{\mathcal{D}}: \mathcal{D} \to \mathbb{P}^1$ ramified at four points. The curve has the equation
\begin{equation}\label{bielliptic}
w^4 + w^2 (u^2+av^2) + bu^4 + cu^3v + d u^2 v^2 + euv^3 + gv^4 =0\, ,
\end{equation}
where $e=1$ or $g=1$. 
\par Precise equations, in terms of invariants of binary sextics, describing the locus $\mathcal{M}_3^h \cap \mathcal{M}_3^b$ can be easily obtained; see \cite{MR3118614}. The same can not be said for the locus $\mathcal{M}_3^b$; see \cite{genus3}. However, there is a geometric description of the locus $\mathcal{M}_3^b$ which seems to have been known from the XIX century and it was pointed out to us by I. Dolgachev. 
\subsection{Characterization of plane bielliptic curves}
Let $\mathcal{D}$ be a canonical curve of genus 3 over $\mathbb{C}$ with a bielliptic involution $\tau:  \mathcal{D} \to \mathcal{D}$. In its canonical plane model given in \cite{MR1816214}, $\tau$ is induced by a projective involution $\tilde{\tau}$ whose set of fixed points consists of a point $u_0 \in\mathbb{P}^2$ and a line $\ell_0$. The intersection $\ell_0 \cap \mathcal{D}$ are the fixed points of $\tau$ on $\mathcal{D}$, namely the branch points of the degree 2 cover $\pi^{\mathcal{D}}: \mathcal{D} \to \mathbb{P}^1$. The following characterization is originally due to Kowalevskaya; see Dolgachev \cite{dolga}:
\begin{theorem}[Kowalevskaya] 
The point $u_0$ is the intersection point of four distinct bitangents of $\mathcal{D}$.  Conversely, if a plane quartic has four  bitangents intersecting at a point $u_0$, then there exists a bielliptic involution $\tau$ of $\mathcal{D}$ such that the projective involution $\tilde{\tau}$ has $u_0$ as its isolated fixed point.
\end{theorem}

\par We give another characterization of bielliptic quartic curves:
\begin{theorem} \label{thm:charcterization}
$\mathcal{D}$ is bielliptic if and only if the following conditions are satisfied:
\begin{itemize} 
\item[(i)] There exists  a line $\ell_0$ intersecting $\mathcal{D}$ at four distinct points $p_1,\ldots, p_4$ with tangent lines $\ell_i$ at the points $p_i$ that intersect at one point $u_0$.

\item[(ii)] Let $P_{u_0}(\mathcal{D})$ be the cubic polar of $\mathcal{D}$ with respect to the point $u_0$ and let $Q$ be the conic component of $P_{u_0}(\mathcal{D})$ (note that the line $\ell_0$ from above is a line component of $P_{u_0}(\mathcal{D})$). Then $\ell_0$ is the polar line of $\mathcal{D}$ with respect to $u_0$.
\end{itemize}
\end{theorem}
\begin{proof} Suppose $\mathcal{D}$ is bielliptic so that its equation $f(u,v,w) = 0$ can be written in the form
\begin{equation}\label{eq1}
\begin{split}
f(u,v,w) & = w^4-2a_2(u,v)w^2 + a_4(u,v) \\
& = (w^2-a_2(u,v))^2+(a_4(u,v)-a_2(u,v)^2) = 0 \,.
\end{split}
\end{equation}
One checks that for the curve in Equation~(\ref{eq1}) the involution $\tau: [u:v:w] \mapsto [u:v:-w]$ is induced by a projective involution $\tilde{\tau}$ whose set of fixed points consists of the point $u_0=[0:0:1] \in\mathbb{P}^2$ and the line $\ell_0= V(w)$. The polar cubic $P_{u_0}(\mathcal{D})$ has the equation $q = w(w^2-a_2(u,v)) = 0$. It is the union of the line $\ell_0$ and the conic $Q = V(w^2-a_2(u,v))$. The line $\ell_0$ intersect $\mathcal{D}$ at the points $p_i: [\beta_i:\alpha_i:0]$, where $a_4(\beta_i,\alpha_i) -a_2(u,v)^2 =0$. The tangent lines at the points $p_i$ are $\ell_i = V(\alpha_i u - \beta_i v)$.  By the main property of polars, $P_{u_0}(\mathcal{D})$ intersects $\mathcal{D}$ at the points $p$ such that the tangent line of $\mathcal{D}$ at $p$ contains the point $u_0$. Thus, the tangent lines $\ell_i$ at $p_i \in \mathcal{D}$ all pass through the point $u_0$, which -- given the normalization of the curve $\mathcal{D}$ in Equation~(\ref{eq1}) -- is $u_0 = [0:0:1]$. Thus,  part i) is verified.  
\par Using Equation~\eqref{eq1}  we compute the line polar $P_{u_0^3}(\mathcal{D}) = V \left(\frac{\partial^3}{\partial w^3}(F) \right)$ of $\mathcal{D}$. It coincides with the line $\ell_0$. On other hand 
\[ 
P_{u_0^3}(\mathcal{D}) = P_{u_0^2}(P_{u_0}(\mathcal{D})) = P_{u_0^2}(qw) = P_{u_0}(q+P_{u_0}(q)w) = 2P_{u_0}(q)+P_{u_0^2}(q)w = w\,,
\]
where we identify the polar curves with the corresponding partial derivatives. This implies that $V \left(P_{u_0}(q) \right) = V(w) = \ell_0$. This checks property (ii).
\par Let us prove the converse. Choose coordinates to assume that $\ell_0 = V(w)$ and the intersection point of the four tangent lines is $u_0 =  [0:0:1]$. The cubic polar $P_{u_0}(\mathcal{D})$ must contain the line component equal to $\ell_0$. Write the equation of $\mathcal{D}$ in the form
\[ 
a_0w^4+a_1(u,v)w^3+a_2(u,v)w^2+a_3(u,v)w+a_4(u,v) = 0 \,.
\]
Then we get  
\[ 
\begin{split}
P_{u_0}(\mathcal{D}) & = V(4a_0w^3+3a_1(u,v)w^2+a_2(u,v)w+a_3(u,v))\, , \\
 P_{u_0^2}(\mathcal{D}) & = V(12a_0w^2+6a_1(u,v)w+a_2(u,v))\, ,\\
   P_{u_0^3}(\mathcal{D}) & = 24a_0w+6a_1(u,v) \,.
\end{split} 
\]
Since $w$ divides the equation of the  polar cubic, we obtain $a_3(u,v) = 0$. If $a_0 = 0$, then $u_0\in \mathcal{D}$, and the line polar  $P_{u_0^3}(\mathcal{D})$ vanishes at $u_0$. But this polar is the tangent line of $\mathcal{D}$ at $u_0$. This implies that $\mathcal{D}$ is singular at $u_0$. So, we may assume that $a_0 \ne 0$.
Thus, the first condition implies that $\mathcal{D}$ can be written in the form
\[w^4+a_1(u,v)w^3+a_2(u,v)w^2+a_4(u,v) = 0 \,.\]
As in the first part of the proof, we obtain that $a_1(u,v) = 0$ if and only if condition (ii) is satisfied. Thus $\mathcal{D}$ can be written in the form Equation~\eqref{eq1}, and hence is bielliptic.
\end{proof}
\par For any general line $\ell$ let $\ell_1, \ldots, \ell_4$ be the tangents of $\mathcal{D}$ at the points $a_1 +  \dots a_4 = \mathcal{D} \cap \ell$ with $\ell_i \cap \mathcal{D} = 2a_i + c_i + d_i$. Adding up, we see that 
\[ 
\sum (c_i+d_i) \sim 4K_\mathcal{D} - 2\sum a_i \sim 4K_\mathcal{D} -2K_\mathcal{D} = 2K_\mathcal{D} \, ,
\]
where $K_\mathcal{D}$ is a canonical divisor. This shows that there exists a conic $S(\ell)$ that cuts out on $\mathcal{D}$ the  divisor $\sum (c_i+d_i)$ of degree 8. This conic is called the \emph{satellite conic} of $\ell$ (see \cite{Cohen2}). The map 
\begin{equation}
\begin{split}
S : \mathbb{P}^2 & \longrightarrow \mathbb{P}^5\, , \\
\ell &  \mapsto S(\ell)\,,
\end{split}
\end{equation}
is given by polynomials of degree 10 whose coefficients are polynomials in coefficients of  $\mathcal{D}$ of degree 7. Since $2\ell+S(\ell)$ and $T = \ell_1+\cdots+\ell_4$ cut out on $\mathcal{D}$ the same divisor, we obtain that the equation of $\mathcal{D}$ can  be written in the form
\[  F = l_1\cdots l_4+l^2q = 0\,,\]
where $\ell_i = V(l_i), \ell = V(l)$, and $S(\ell) = V(q)$.
\par Using an automorphism of $\mathbb{P}^2$ we can assume that the line $\ell$ is given by $\ell=V(w)$. For a general quartic, given by
\[
\begin{split}
& a{x}^{4}+b{y}^{4}+c{z}^{4}+6\,f{y}^{2}{z}^{2}+6\,g{x}^{2}{z}^{2}+6\,h{x}^{2}{y}^{2}+12\,l{x}^{2}yz+12\,mx{y}^{2}z+12\,nxy{z}^{2}\\
& +4\,{x}^{3}y a_{{1}} +4\,{x}^{3}za_{{2}}+4\,x{y}^{3}b_{{0}}+4\,x{z}^{3}c_{{0}}+4\,{y
}^{3}zb_{{2}}+4\,y{z}^{3}c_{{1}} =0 \,,
\end{split}
\]
an expression of the satellite conic is
\[
\begin{split}
S & = {x}^{2} \left( 9\,a{f}^{2}-16\,ab_{{2}}c_{{1}} \right) +2\,xy \left( 
18\,a_{{1}}{f}^{2}-32\,a_{{1}}b_{{2}}c_{{1}} \right) +2\,xz \left( 18
\,a_{{2}}{f}^{2}-32\,a_{{2}}b_{{2}}c_{{1}} \right) \\
& +{y}^{2} \left( 54
\,{f}^{2}h-96\,b_{{2}}c_{{1}}h \right) +2\,xy \left( 54\,{f}^{2}l-108
\,fmn-96\,b_{{2}}c_{{1}}l+72\,c_{{1}}{m}^{2}+72\,b_{{2}}{n}^{2}
 \right) \\
 & +{z}^{2} \left( 54\,f{g}^{2}-96\,b_{{2}}c_{{1}}g \right) \,.
 \end{split}
\]
For the bielliptic curve in \eqref{bielliptic} this satellite conic is 
\[
S=-{u}^{2}ace+ \frac 1 4\,{u}^{2}a{d}^{2}-{v}^{2}ce+ \frac 1 4\,{v}^{2}{d}^{2}-{w}^{2}cea+ \frac 1 4\,{w}^{2}d{a}^{2} \,.
\]
A line  $\ell$ is called a \emph{bielliptic line} if  the four tangents $\ell_i$ intersect at a common point $u_0$.   Choose the coordinates such that $u_0 = [0: 0 :1]$ and $l = w$. Then the equation of $\mathcal{D}$ is of the form
\[  
F = w^2\big(a_0w^2+a_1(u,v)w+a_2(u,v)\big)+a_4(u,v) = 0\,. 
\]
It is a bielliptic curve if and only if $a_1(u,v) = 0$. This is equivalent to $P_{u_0}(S(\ell)) = \ell$. Thus, we have obtained the following:
\begin{theorem} 
Suppose   $\ell$ is a bielliptic line.  Then $\mathcal{D}$ is bielliptic if and only if the polar line of the satellite conic $S(\ell)$ with respect to the point $u_0$ coincides with $\ell$.
\end{theorem}
\par Let $\ell$ be a bielliptic line. The polar cubic of $P_{u_0}(\mathcal{D})$ passes through $\mathcal{D} \cap \ell$, hence it contains $\ell$ as an  irreducible component. In particular, $P_{u_0}(\mathcal{D})$ is singular. Recall that the locus  of points $u\in \mathbb{P}^2$ such that  $P_u(\mathcal{D})$ is a singular cubic  is the Steinerian curve $\operatorname{St}( \mathcal{D} )$. If $\mathcal{D}$ is  general enough,  the degree of $\operatorname{St}(\mathcal{D})$ is equal to 12 and it has 
24 cusps and 21 nodes. The cusps correspond to points such that the polar cubic is cuspidal, the nodes correspond to  points such that the polar cubic is reducible. The line components define the set of 21 bielliptic lines. In \cite{Cohen2} the 21 lines are described as singular points of multiplicity 4 of the curve of degree 24 in the dual plane parameterizing lines that the tangents to $\mathcal{D}$ at three intersection points of $\mathcal{D}$ and $\ell$ are concurrent.
\par According to  \cite{Cohen2}*{p.~327}, the equation of the satellite conic $S(\ell)$ is equal to 
\[ S = \mathcal{D}_{7,2,10} + \ell \cdot \mathcal{D}_{7,1,9}  +  \ell^2 \cdot \mathcal{D}_{7,0,8} = 0\,,\]
where $\mathcal{D}_{a,b,c}\in S^{a}(S^4(V^*)^*)\otimes S^b(V)\otimes S^c (V^*)$ is a comitant of degree $a$ in coefficients of $\mathcal{D}$, of degree $b$ in  coordinates, in  the plane $\mathbb{P}(V^*)$ (we use Grothendieck's notation) and of degree $c$ in  coordinates of the dual plane. Thus the vanishing of $a_1(u,v)$ from above is equivalent to the vanishing of the comitant $\mathcal{D}_{7,1,9}$. In  \cite{Cohen2} Cohen gives an explicit equation of $\mathcal{D}_{7,1,9}$:
\begin{theorem} $\mathcal{D}$ is bielliptic if and only if  $\mathcal{D}_{7,1,9}$, considered as a map $\mathbb{P}(V)\longrightarrow \mathbb{P}(V^*)$ has one  of the 21 lines corresponding to the nodes of $\operatorname{St}(\mathcal{D})$ as its indeterminacy point.  The rational map is given by polynomials of degree 9 with polynomial coefficients in coefficients of $\mathcal{D}$ of degree $7$. 
\end{theorem}
\par This gives the equations of the locus of bielliptic curves in $\mathcal{M}_3$. It is unknown to us if this locus has ever been explicitly determined in terms of the invariants of the ternary quartics. 
\subsection{Ramification locus}
In this section we determine explicit equations for plane bielliptic genus-three curves based on their characterization in Theorem~\ref{thm:charcterization}. All computations in this section are carried out over an arbitrary field $K$. We start with the plane bielliptic genus-three curve $\mathcal{D}$ given in Equation~(\ref{eq1}), i.e.,
\begin{equation}
  \mathcal{D}: \quad w^4 - 2 \, a_2(u,v) \, w^2 + a_4(u,v) = 0 \,,
\end{equation}
with $[u:v:w] \in \mathbb{P}^2$ and general homogeneous polynomials $a_2$ and $a_4$ of degree two and four, respectively, and the bielliptic involution $\tau:[u:v:w] \mapsto [u:v:-w]$. It follows from \cite{MR1816214}*{Corollary~2.2} that any such smooth curve $\mathcal{D}$ is the canonical model of a bielliptic non-hyperelliptic curve of genus three. The bielliptic quotient $\mathcal{D}/\langle \tau \rangle$ is the genus-one curve $\mathcal{Q}$ given by
\begin{equation}
   \mathcal{Q}: \quad W^2 = a_2(u,v)^2 - a_4(u,v) =  c_4 u^4 + c_3 u^3v + c_2 u^2v^2 + c_1 uv^3 + c_0 v^4 \,,
\end{equation}
with $W=w^2 - a_2(u,v)$ and $[u,v,W]\in\mathbb{P}(1,1,2)$. Using a standard technique, as explained for example in \cite{MR3263663}*{App.~A}, we convert this genus-one curve to Weierstrass form given any $K$-rational point on the curve. We have the following:
\begin{lemma}
\label{lem:normal_form_BEC}
Given a $K$-rational point, the bielliptic quotient $\mathcal{D}/\langle \tau \rangle$ is isomorphic to the elliptic curve $\mathcal{E}$ given by
\begin{equation}\label{eqn:quotientEC0}
 \mathcal{E}: \quad \rho^2 \eta = \xi^3 + f  \xi^2 \eta + g \xi \eta^2 + h \eta^3 \,,
\end{equation}
with $[\xi:\eta:\rho] \in \mathbb{P}^2$ and 
\begin{align*}
 f = 3 \,c_1^2 - 8 \, c_0 c_2 \,,\quad h = (c_1^3 - 4 \, c_0 c_1 c_2 + 8 \, c_0^2 c_3)^2\,,\\
 g = 3 \, c_1^4 -16 \, c_0 c_1^2 c_2 + 16 \, c_0^2 (c_2^2 + c_1c_3) -64 c_0^3 c_4\,.
\end{align*}
\end{lemma}
\begin{proof}
By a change of coordinates we can assume that the $K$-rational point is given by $[u,v,W]=[0:1:\alpha]$, that is, $c_0=\alpha^2$.  Using the transformation
\[
 u = - \frac{4 c_0 \xi \eta v}{\tilde{\rho}} \,, \quad w = \alpha v^2 \, \frac{\tilde{\rho}^2- 2c_1 \xi \eta \tilde{\rho} -2  \xi^3\eta - 2 (c_1^2-4 c_0c_2) \xi^2\eta^2  }{\tilde{\rho}^2} \,,
\]
followed by the transformation
\[
 \tilde{\rho} = \rho \eta + c_1 \xi \eta+ (c_1^3 - 4 \, c_0 c_1 c_2 + 8 \, c_0^2 c_3) \eta^2 \,,
\] 
proves the lemma.
\end{proof}
\begin{remark}
The elliptic curve~(\ref{eqn:quotientEC0}) remains well defined, independently of the existence of a $K$-rational point. However, in general there is only an isomorphism 
\[
 \operatorname{Jac}(\mathcal{D}/\langle \tau \rangle) \cong \mathcal{E} \,.
 \] 
 The existence of a $K$-rational point is required for an isomorphism $\mathcal{D}/\langle \tau \rangle \cong \mathcal{E}$. The Jacobian was first found by Hermite as the determinant of a symmetric matrix that defines a conic bundle which degenerates over $\mathcal{E}$; see \cite{MR2004218}. 
\end{remark}
The following lemma is easily verified:
\begin{lemma}
\label{cor:ECgeneral}
The elliptic curve $\mathcal{E}$ with full $K$-rational two-torsion given by
\begin{equation}\label{eqn:quotientEC}
 \mathcal{E}: \quad \rho^2 \eta = \xi \Big(\xi^2 - 2b \, \xi \eta + (b^2-4 a^2) \, \eta^2\Big) \,,
\end{equation}
with $[\xi:\eta:\rho] \in \mathbb{P}^2$, is isomorphic to the genus-one curve
\begin{equation}
\label{eqn:Qbiquadratic}
 \mathcal{Q}: \quad W^2 = u^4 + b \, u^2v^2 + a^2  v^4 \,,
\end{equation}
where $[u:v:W] \in \mathbb{P}(1,1,2)$. An isomorphism $\varphi: \mathcal{E}\overset{\cong}{\longrightarrow} \mathcal{Q}$ is given by
\begin{equation}
\label{eqn:iso_curves}
[\xi:\eta:\rho] \mapsto  [u:v:W] =\Big[ \rho\eta \,: -2\xi\eta : (\rho^2 \eta + 2b\, \xi^2 \eta-2\xi^3) \eta \Big]\,,
\end{equation}
and by mapping points $\mathsf{T}_1: [\xi:\eta:\rho] =[0:1:0]$ and $\mathsf{O}: [0:0:1]$ to $[u:v:W]=[1:0:1]$ and $[1:0:-1]$, respectively. 
\end{lemma}
\qed
\par For the elliptic curve $\mathcal{E}$ in Equation~(\ref{eqn:quotientEC}), the flex-point is the point at infinity $[\xi:\eta:\rho] =[0:0:1]$ which is also the base point $\mathsf{O}$ for the elliptic-curve group law. The point  $\mathsf{T}_1: [\xi:\eta:\rho] =[0:1:0]$ is a non-trivial two-torsion point. We have the following:
\begin{proposition}
\label{cor:bielliptic_quotient}
The plane bielliptic genus-three curve
\begin{equation}
\label{eqn:D}
   \mathcal{D}: \quad \Big(w^2 - a_2(u,v)\Big)^2  =  u^4 + b \, u^2v^2 + a^2 v^4  \,,
\end{equation}
where $[u:v:w] \in \mathbb{P}^2$, $a_2$ is a homogeneous polynomial of degree two, and $a, b$ are $K$-rational numbers such that $a (b^2-4a^2) \not =0$, admits the bielliptic involution $\tau:[u:v:w] \mapsto [u:v:-w]$ and a degree-two cover given by 
\begin{equation}
 \pi^{\mathcal{D}}_{\mathcal{Q}}: \; \mathcal{D} \to \mathcal{Q}, \qquad [u:v:w] \mapsto [u:v:W=  w^2 - a_2(u,v)] \,,
\end{equation}
onto the genus-one curve $\mathcal{Q}$ in Equation~(\ref{eqn:Qbiquadratic}). The branch locus of the bielliptic involution $\tau$ is isomorphic via $\varphi$ to a collection of points $\{\mathrm{pt}_1$, $\mathrm{pt}_2$, $\mathrm{pt}_3$, $\mathrm{pt}_4\} \subset \mathcal{E}$ in Equation~(\ref{eqn:iso_curves}) satisfying
\begin{equation}
 \xi^3 - a_2\Big( \rho, - 2 \xi \Big) \, \eta - (b^2-4a^2) \, \xi \eta^2 =  0 \,.
\end{equation}
In particular, we have $\sum_{i=1}^4 \mathrm{pt}_i=\mathsf{O}$. Conversely, the elliptic curve $\mathcal{E}$ in Equation~(\ref{eqn:iso_curves}) and $\{\mathrm{pt}_1$, $\mathrm{pt}_2$, $\mathrm{pt}_3$, $\mathrm{pt}_4\} \subset \mathcal{E}$ with $\sum_{i=1}^4 \mathrm{pt}_i=\mathsf{O}$ determine Equation~(\ref{eqn:D}) uniquely.
\end{proposition}
\begin{proof}
The first part follows by explicit computation using Lemma~\ref{cor:ECgeneral} and the group law on $\mathcal{E}$. Conversely, the elliptic curve in Equation~(\ref{eqn:quotientEC0}) is isomorphic to the general genus-one quotient curve given by Lemma~\ref{lem:normal_form_BEC} iff we impose $h=0$. The condition $h=0$ allows us to express the coefficients $c_2 \alpha^2, c_3 \alpha^4, c_4\alpha^6$ with $c_0=\alpha^2$ as simple rational functions of $A, B, c_1$. The general isomorphism $\varphi: \mathcal{E} \to \mathcal{Q}$ is given by
\begin{equation}
[\xi:\eta:\rho] \mapsto  [u:v:W] =\Big[ (2c_1 \xi + \rho)\eta \,: -2\xi\eta : (\rho^2 \eta + 2b \xi^2 \eta-2\xi^3) \eta \Big]\,,
\end{equation}
such that
\begin{equation}
 \mathcal{Q}: \quad W^2 = \left \lbrace\begin{array}{l} u^4 + b \, u^2v^2 + a^2 \, v^4  \\
    +\, c_1 \Big(2u+c_1 v\Big)\Big(2u^2+2c_1uv+(c_1^2+b)v^2\Big)v\end{array}\right. \,,
\end{equation}
where $[u:v:W] \in \mathbb{P}(1,1,2)$, $c_1 \in K$ is an arbitrary coefficient. 
\par The branch locus on $\mathcal{E}$ in Equation~(\ref{eqn:quotientEC0}) uniquely defines a conic. This conic is given by
\begin{equation}\label{eqn:K}
 \mathcal{K}: \quad (1-\gamma)  \xi^2 + 4\beta \gamma  \rho\eta + 4 \alpha  \xi\eta - (1+\gamma) (b^2-4a^2)  \eta^2 =0 \,,
\end{equation}
with $\alpha, \beta, \gamma \in K$. If a plane curve of degree $n$ intersects an elliptic curve in $3n$ points, then these points always sum up using the group law of the elliptic curve $\mathcal{E}$ in Equation~(\ref{eqn:iso_curves}). In our case, we expect six points $\mathrm{pt}_1, \dots, \mathrm{pt}_6 \in \mathcal{E}$ such that $[\mathrm{pt}_1 + \dots + \mathrm{pt}_6  - 6 \mathsf{O}]= 0 \in \operatorname{Pic}^0(\mathcal{E})$ as it is the divisor class of $\operatorname{Div}(\mathcal{K}/\mathcal{L}^6)$ where $\mathcal{L}: \eta=0$ is the flex-line. However, the conic and the elliptic curve intersect at $\eta=0$; one checks this computing the resultant of $\mathcal{K}$ and the defining equation of $\mathcal{E}$. From Equation~(\ref{eqn:K}) one checks that the intersection at $\eta=0$ has order two whence $\mathrm{pt}_5=\mathrm{pt}_6=\mathsf{O}$. Therefore, the remaining four points $\mathrm{pt}_1, \dots, \mathrm{pt}_4$ satisfy $[\mathrm{pt}_1 + \dots + \mathrm{pt}_4  - 4 \mathsf{O}]= 0$. We set $a_2(u,v)=\gamma \big(u+(\beta+c_1)v\big)^2-(\alpha+\beta^2\gamma-\gamma b/2) v^2$, and the branching locus then satisfies
\begin{equation}
 \xi^3 - (b^2-4a^2) \, \xi \eta^2 =  a_2\Big( 2 c_1\xi +\rho, - 2 \xi \Big) \, \eta\,.
\end{equation}
In turn, the plane genus-three curve is given by setting
\begin{equation}
\begin{split}
 w^2  = W + a_2(u,v) = - \Big(\xi^3 - (b^2-4a^2) \xi \eta^2\Big)\, \eta +  a_2\Big( 2 c_1\xi +\rho, - 2 \xi \Big) \,\eta^2 \,.
 \end{split}
\end{equation}
Since $\Delta_{\mathcal{E}}=\Delta_{\mathcal{Q}}=16 a^2 (b^2-4a^2)^2$, we can set $c_1=0$ without loss of generality.
\end{proof}
We have the following:
\begin{remark}\label{lem:3points}
The point $\mathsf{O}$ is a branch point if and only if $a_2(u,0)=u^2$. We then write $a_2(u,v)=(u+\beta v)^2-(\alpha+\beta^2-b/2) v^2$ with $\alpha, \beta \in K$. The remaining points of the branch locus lie on the intersection of $\mathcal{E}$ with the line $2\alpha \xi + 2\beta \rho - (b^2-4a^2)\eta=0$. If the point $\mathsf{O}$ is in the branch locus of $\pi^{\mathcal{D}}_{\mathcal{Q}}$, then the remaining points $\{\mathrm{pt}_1$, $\mathrm{pt}_2$, $\mathrm{pt}_3\}$ satisfy $\sum_{i=1}^3 \mathrm{pt}_i=\mathsf{O}$ on $\mathcal{E}$. 
\end{remark}
On $\mathcal{E}$ we have different involutions acting on points $p\in \mathcal{E}$: (1) the hyperelliptic involution $\imath^{\mathcal{E}}: p \mapsto -p$ given by $[\xi:\eta:\rho] \mapsto  [\xi:\eta:-\rho]$;  (2) the involution $\imath^{\mathcal{E}}_{\mathsf{T}_1}: p \mapsto p + \mathsf{T}_1$ obtained by translation by two-torsion $\mathsf{T}_1$ and given by
\begin{equation}
 \label{dual_isog_involution}
 \imath^{\mathcal{E}}_{\mathsf{T}_1}:\quad [\xi:\eta:\rho]\mapsto \left[ (b^2-4a^2)\, \xi\eta: \xi^2: - (b^2-4a^2)\, \rho\eta \right]\,;
\end{equation} 
(3) the composition $\imath^{\mathcal{E}}  \circ \imath^{\mathcal{E}}_{\mathsf{T}_1} =  \imath^{\mathcal{E}}_{\mathsf{T}_1} \circ \imath^{\mathcal{E}}: p \mapsto -p + \mathsf{T}_1$. We have the following:
\begin{lemma}\label{lem:involutions}
The involutions act  on $\mathcal{Q}$ as follows
\begin{equation}
\begin{split}
 \varphi \circ \imath^{\mathcal{E}}:& \quad  [u:v:W] \mapsto \Big[ - u: v: W\Big] = \Big[ u: -v: W\Big]\,,\\
 \varphi \circ \imath^{\mathcal{E}}_{\mathsf{T}_1}:& \quad  [u:v:W] \mapsto \Big[ - u: v: - W\Big]= \Big[ u: -v: - W\Big] \,,\\
  \varphi \circ \left(\imath^{\mathcal{E}}  \circ \imath^{\mathcal{E}}_{\mathsf{T}_1}\right):& \quad  [u:v:W] \mapsto \Big[ u: v: - W\Big] \,.
\end{split}
\end{equation}
\end{lemma}
\begin{proof}
The proof follows by computation.
\end{proof}
\subsection{Biquadratic quotients}\label{sec:odd_case}
\label{ssec4}
We discuss the case in Proposition~\ref{cor:bielliptic_quotient} for a branch locus on $\mathcal{E}$ in Equation~(\ref{eqn:quotientEC}) that consists of the points $\mathsf{O}: [\xi:\eta:\rho] =[0:0:1]$ and $\mathrm{pt}_1$, $\mathrm{pt}_2$, $\mathrm{pt}_3$ with $\sum_{i=1}^3 \mathrm{pt}_i=\mathsf{O}$. We have the following:
\begin{lemma}
\label{lem:adding_points}
Given two $K$-rational points $q_1, q_2  \in \mathcal{E}$ such that $[u:v:W]=[R_i:1:S_i]=\varphi(q_i)$ for $1\le i\le 2$ with $R_1^2 \not =R_2^2$, the point $q_3=-q_1-q_2 \in \mathcal{E}$ is a $K$-rational point with $[R_3:1:S_3]=\varphi(q_3)$, and we have for $a, b$ in Equation~(\ref{eqn:quotientEC}) the relations
\begin{equation}
\label{eqn:coeffs_3a}
\begin{split}
 a^2  = R_1^2R_2^2 + \frac{R_1^2S_2^2-R_2^2 S_1^2}{R_1^2-R_2^2} \,,\quad b= - \frac{R_1^4-R_2^4-S_1^2+S_2^2}{R_1^2-R_2^2} \,,
 \end{split} 
\end{equation} 
and 
\begin{equation}
\begin{split}
\label{eqn:coeffs_3b}
  R_3 = \frac{R_2S_1-R_1S_2}{R_1^2-R_2^2} \,,\quad S_3 = - R_1 R_2 + \frac{(R_1S_1-R_2S_2)(R_2S_1-R_1S_2)}{R_1^2-R_2^2}  \,.
  \end{split} 
\end{equation} 
\end{lemma}
\begin{proof}
Given two different $K$-rational points $\mathrm{pt}_1, \mathrm{pt}_2$ on the elliptic curve $\mathcal{E}$ in Equation~(\ref{eqn:quotientEC}) such that  $[R_1:1:S_1]=\varphi(q_1)$ and $[R_2:1:S_2]=\varphi(q_2)$ with $R_1^2 \not =R_2^2$, the point $q_3=-q_1-q_2$ is $K$-rational. Since the coordinates $[R_i:1:S_i]$ for $1\le i \le3$ label points on $\mathcal{Q}$ satisfying Equation~(\ref{eqn:Qbiquadratic}), we can easily derive Equation~(\ref{eqn:coeffs_3a}). We compute the coordinates for all points on $\mathcal{E}$ and check using the elliptic-curve group law that $q_1+q_2+q_3=\mathsf{O}$. 
\end{proof}
\par  For $\epsilon_1, \epsilon_2, \epsilon_3 \in \{\pm 1\}$ and $\varepsilon=(\epsilon_1, \epsilon_2,\epsilon_3)$ and two distinct $K$-rational points $q_1, q_2  \in \mathcal{E}$ such that $[u:v:W]=[R_i:1:S_i]=\varphi(q_i)$ for $1\le i\le 2$ with $R_1^2 \not =R_2^2$, we define the plane bielliptic genus-three curves $\mathcal{D}^{\varepsilon}$ given by
\begin{equation}\label{eqn:De1e2e3}
\scalemath{0.8}{
\begin{aligned}
w^4 -  2 w^2 \left(  u^2 -  \Big(\epsilon_1 R_1+\epsilon_2 R_2 +\epsilon_3 \frac{\epsilon_2 S_1- \epsilon_1 S_2}{\epsilon_1 R_1-\epsilon_2 R_2}\Big) \, uv
 + \left(\epsilon_1\epsilon_2 R_1R_2- \epsilon_3 \frac{R_1 S_2-R_2 S_1}{\epsilon_1R_1 - \epsilon_2 R_2} \right) \, v^2 \right)\\
-2 \left( \epsilon_1 R_1+\epsilon_2 R_2 + \epsilon_3 \frac{\epsilon_2 S_1-\epsilon_1 S_2}{\epsilon_1 R_1-\epsilon_2 R_2}\right) \Big(u-\epsilon_1 R_1 v\Big)\Big(u-\epsilon_2 R_2 v\Big)\Big(u-\epsilon_3 \frac{R_2S_1-R_1S_2}{R_1^2-R_2^2} v\Big) \, v =0\,,
 \end{aligned}}
\end{equation}
with $[u:v:W] \in \mathbb{P}(1,1,2)$. We have the following:
\begin{lemma}
\label{prop:bielliptic_curve_3points}
Given two $K$-rational points $q_1, q_2 \in \mathcal{E}$ such that  $[R_i:1:S_i]=\varphi(q_i)$  for $1\le i\le 2$ with $R_1^2 \not =R_2^2$, we set $q_3=-q_1-q_2 \in \mathcal{E}$ with $[R_3:1:S_3]=\varphi(q_3)$. The following holds:
\begin{enumerate}
\item For all $\epsilon_1, \epsilon_2, \epsilon_3 \in \{\pm 1\}$,  the plane genus-three curves $\mathcal{D}^{\varepsilon}$ in Equation~(\ref{eqn:De1e2e3}) admit the bielliptic involution $\imath^{\mathcal{D}}_b:[u:v:w] \mapsto [u:v:-w]$ and the degree-two quotient map $\pi^{\mathcal{D}^{\varepsilon}}_{\mathcal{Q}}$ given by
\begin{equation}
\label{eqn:quotient_3}
\begin{split}
 \pi^{\mathcal{D}^{\varepsilon}}_{\mathcal{Q}} :  \quad w  \mapsto W =  & \, w^2 - u^2 +  \Big(\epsilon_1 R_1+\epsilon_2 R_2 +\epsilon_3 \frac{\epsilon_2 S_1- \epsilon_1 S_2}{\epsilon_1 R_1-\epsilon_2 R_2}\Big)  \, uv \\
 & - \left(\epsilon_1\epsilon_2 R_1R_2- \epsilon_3 \frac{R_1 S_2-R_2 S_1}{\epsilon_1R_1 - \epsilon_2 R_2} \right) \, v^2 \,,
 \end{split}
\end{equation}
onto the curve $\mathcal{Q}$ in Equation~(\ref{eqn:Qbiquadratic}), isomorphic to $\mathcal{E}$ in Equation~(\ref{eq:EC}). 
\item The branch points $ \mathrm{pt}_i \in \mathcal{E}$ with $1\le i\le 4$ of $\pi^{\mathcal{D}^{\varepsilon}}_{\mathcal{Q}}$ are given by:
\begin{equation}
 \begin{array}{c|rrr|rrrr}
  \# &\epsilon_1 & \epsilon_2 & \epsilon_3 & \mathrm{pt}_1 & \mathrm{pt}_2 & \mathrm{pt}_3 & \mathrm{pt}_4 \\
  \hline
  1& 1	& 1	& 1		& q_1		& q_2		&  q_3		& \mathsf{O} \\	
  2&-1	& -1	& -1		&-q_1		&-q_2		& -q_3		& \mathsf{O} \\
  3& 1	&-1	& -1		& q_1		& \mathsf{T}_1+q_2	& \mathsf{T}_1+q_3	& \mathsf{O} \\
  4& -1	& 1	& 1		& -q_1		& \mathsf{T}_1-q_2		& \mathsf{T}_1-q_3		& \mathsf{O} \\	
  5& -1	&-1	& 1		& \mathsf{T}_1+q_1	& \mathsf{T}_1+q_2	& q_3		& \mathsf{O} \\
  6&  1	&1	& -1		& \mathsf{T}_1-q_1		& \mathsf{T}_1-q_2		& -q_3		& \mathsf{O} \\
  7&-1	&1	& -1		& \mathsf{T}_1+q_1	& q_2		& \mathsf{T}_1+q_3	& \mathsf{O} \\
  8&1		& -1	& 1		& \mathsf{T}_1-q_1		& -q_2		& \mathsf{T}_1-q_3		& \mathsf{O} 	
 \end{array}
 \end{equation}
 \item With respect to the elliptic-curve group law, we have $\sum_{i=1}^3 \mathrm{pt}_i=\mathsf{O}$.
\end{enumerate}
\end{lemma}
\begin{proof}
Using Equations~(\ref{eqn:coeffs_3a}), (\ref{eqn:coeffs_3b}), $\mathcal{Q}$ in Equation~(\ref{eqn:Qbiquadratic}) can be written as
\begin{equation}
\label{eqn:Q_3pts}
 \begin{split}
 & \mathcal{Q}:   \quad W^2  =  2 \, \Big(R_1+R_2 + \frac{S_1-S_2}{R_1-R_2}\Big)  \, (u-R_1 v) (u-R_2 v) (u-R_3 v)\, v \\
 &+ \left( u^2 - \Big(R_1+R_2 + \frac{S_1-S_2}{R_1-R_2}\Big) \, uv +  \left(R_1R_2-  \frac{R_1 S_2-R_2 S_1}{R_1 - R_2} \right) \,v^2\right)^2 \,.
  \end{split}
 \end{equation}
Regardless of what signs $\epsilon_1, \epsilon_2, \epsilon_3$ are chosen in Equation~(\ref{eqn:De1e2e3}), the bielliptic quotient is always the same, namely it coincides with the curve $\mathcal{Q}$ in Equation~(\ref{eqn:Q_3pts}). (1) is then immediate; for (2) one checks that the branch points of the map $\pi^{\mathcal{D}^{\varepsilon}}_{\mathcal{Q}}$ in Equation~(\ref{eqn:quotient_3}) are the points with coordinates $[\epsilon_1 R_1:1:\epsilon_2 \epsilon_3 S_1]$, $[\epsilon_2 R_2:1:\epsilon_1 \epsilon_3 S_2]$, $[\epsilon_3 R_3:1:\epsilon_1 \epsilon_2 S_3]$ and $\mathsf{O}$. Lemma~\ref{lem:involutions} provides the geometric interpretation for these branch points; (3) follows from Lemma~\ref{lem:adding_points}.
\end{proof}
\par We consider the plane bielliptic curve $\mathcal{D}$ given by
\begin{equation}
\label{eqn:master_curve_3points}
\mathcal{D}: \quad \Big( e \big(w^2 - u^2\big) -  c uv - d v^2 \Big)^2 = e^2\Big(u^4 + b  u^2 v^2 + a^2 v^4\Big) \,,
\end{equation}
with $[u:v:w] \in \mathbb{P}^2$ and $a, b, c, d, e \in K$. We discuss the singular locus of Equation~(\ref{eqn:master_curve_3points}). We have the following:
\begin{lemma}
\label{lem:smooth}
The plane genus-three curve in Equation~(\ref{eqn:master_curve_3points}) is irreducible and non-singular if and only if $\Delta_{\mathcal{E}} \Delta_{\mathcal{D}} \not =0$ where
\begin{equation}
\label{eqn:singular_locus}
\begin{split}
\Delta_{\mathcal{E}} = & \; 16 \, a^2 \big(b^2-4a^2\big)^2 \,, \\
\Delta_{\mathcal{D}} = & \; - \left(  \left( c^2-be^2-4\, de \right)^2 -12 \, d e^2 \left( be+d \right)  \right)^3 \ \\
&\; + \Big( 54 \, a  c^2 e^4 - c^6+ 3 \left( be + 4 d \right) c^4 e\\
&\quad - 3 \left( b^2e^2+2 \, bde+10 \, d^2 \right) c^2e^2+ \left( be-2\,d \right)^{3}e^3 \Big)^{2} \,.
\end{split}
\end{equation}
\end{lemma}
\begin{proof}
For $e=0$ Equation~(\ref{eqn:master_curve_3points}) is singular and $\Delta_{\mathcal{D}}=0$. We assume $e\not =0$. Then, we can set $e=1$ since rescaling $c\mapsto ce$ and $d\mapsto de$ eliminates $e$ from the equation. One checks that for $b^2-4a^2=0$, Equation~(\ref{eqn:master_curve_3points}) factors and the curve is reducible. We assume $b^2-4a^2=0\not =0$. One checks that the curve in Equation~(\ref{eqn:master_curve_3points}) has a singular point with $w=0$  iff $\Delta_{\mathcal{D}}=0$. $\Delta_{\mathcal{D}}$ is the iterated (reduced) discriminant of Equation~(\ref{eqn:master_curve_3points}) with respect to $w$ and $u$ (or $v$). For $w \not =0, u=0$, the curve in Equation~(\ref{eqn:master_curve_3points}) has a singular point if $a=0$. We assume $\Delta_{\mathcal{E}} \not =0$. Then, there are no singular points with $w \not =0, v\not =0$ and $\Delta_{\mathcal{E}} \Delta_{\mathcal{D}} \not =0$.
\end{proof}
\begin{remark}
\label{rem:smooth}
If $c=0$ or $e=0$, then $\Delta_{\mathcal{D}} =0$ in Equation~(\ref{eqn:singular_locus}).
\end{remark}
\par The genus-one curve $\mathcal{Q}$ in Equation~(\ref{eqn:Qbiquadratic}) is isomorphic via $\varphi$ in Equation~(\ref{eqn:iso_curves}) to the elliptic curve $\mathcal{E}$ in Equation~(\ref{eqn:quotientEC}). We now make the latter coincide with the elliptic curve in Proposition~\ref{prop:normal_form} to obtain the following:
\begin{proposition}
\label{cor:connection_3points}
Let $\mathcal{D}$ be the plane bielliptic curve given by
\begin{equation}
\label{eqn:master_curve_3points_b}
\mathcal{D}: \quad \left( w^2 - u^2 -  \frac{c}{e} \, uv - \frac{d}{e} \, v^2 \right)^2 = u^4 + b  u^2 v^2 + a^2 v^4 \,,
\end{equation}
with $[u:v:w] \in \mathbb{P}^2$, $a, b$ given in Equation~(\ref{eqn:coeffs_ab}), and coefficients
\begin{equation}
 \label{eqn:coeffs_cde}
 c  = c\big(\lambda_0,\lambda_1,\lambda_2,\lambda_3, l\big) \,,\quad
 d  = d\big(\lambda_0,\lambda_1,\lambda_2,\lambda_3, l\big) \,,\quad
 e  = e\big(\lambda_0,\lambda_1,\lambda_2,\lambda_3, l\big) \,,
\end{equation}
with polynomials $c, d, e$ given in Appendix~\ref{App:coeffs} such that $\Delta_{\mathcal{E}} \Delta_{\mathcal{D}} \not =0$. Then, the curve $\mathcal{D}$ is smooth and irreducible, and admits the involution $\tau:[u:v:w] \mapsto [u:v:-w]$ and the degree-two cover
\begin{equation}
\label{eqn:double_cover}
 \pi^{\mathcal{D}}_{\mathcal{Q}}: \; \mathcal{D} \to \mathcal{Q}\,, \quad [u:v:w] \mapsto \left[u:v:W=  w^2 - u^2 - \frac{c}{e} \, uv - \frac{d}{e} \, v^2\right] \,,
\end{equation}
onto $\mathcal{Q} \cong \mathcal{E}$ with branch points $\{ \mathsf{O}, 2\mathsf{p}_1, \mathsf{p}_1+  \mathsf{p}_2, -3\mathsf{p}_1 - \mathsf{p}_2\} \subset \mathcal{E}$ where $\mathcal{E}$ is the smooth elliptic curve given in Equation~(\ref{eqn:quotientEC}) with $\Delta_{\mathcal{E}} \not =0$ and the $K$-rational points $\mathsf{p}_1, \mathsf{p}_2$ in Lemma~\ref{lem:adding}.
\end{proposition}
\begin{proof}
It follows from Lemma~\ref{lem:smooth} that $\mathcal{D}$ is smooth and irreducible, and from Remark~\ref{rem:smooth} that Equation~(\ref{eqn:double_cover}) is well defined. We apply Lemma~\ref{prop:bielliptic_curve_3points} to the situation encountered in Proposition~\ref{prop:normal_form} with $q_1=2\mathsf{p}_1$, $q_2=\mathsf{p}_1+ \mathsf{p}_2$ where the $K$-rational points $\mathsf{p}_1, \mathsf{p}_2$ were given in Lemma~\ref{lem:adding}. Using the isomorphism $\varphi: \mathcal{E}\to \mathcal{Q}$ we compute the coordinates $[u:v:W]=[R_i:1:S_i]=\varphi(q_i)$ for $1\le i\le 2$. We then use Equations~(\ref{eqn:coeffs_3a}) and~(\ref{eqn:coeffs_3b}) to obtain formulas for the coefficients $c, d, e$. 
\end{proof}
\begin{remark}
Replacing $l \mapsto -l$ is equivalent to $\mathsf{p}_2 \mapsto -\mathsf{p}_2$ due to Lemma~\ref{lem:adding}. Thus, Proposition~\ref{cor:connection_3points} generalizes to branch points $\{ \mathsf{O}, 2\mathsf{p}_1, \mathsf{p}_1+ \epsilon_2 \mathsf{p}_2, -3\mathsf{p}_1 - \epsilon_2 \mathsf{p}_2\} \subset \mathcal{E}$ with $\epsilon_2 \in \{\pm 1\}$ when replacing $l \mapsto \epsilon_2 l$ in Equations~(\ref{eqn:coeffs_cde}).
\end{remark}
\begin{remark}
It follows from Lemma~\ref{lem:involutions} that inversion $q \mapsto -q$ on the elliptic curve is equivalent to $u \mapsto -u$. Moreover, for $[R_i:1:S_i] \mapsto [-R_i:1:S_i]=\varphi(-q_i)$ with $1\le i\le 2$ we have $[R_3:1:S_3] \mapsto [-R_3:1:S_3]$ in Equations~(\ref{eqn:coeffs_3b}). Therefore, a bielliptic plane quartic with branch points $\{ \mathsf{O}, -2\mathsf{p}_1, -\mathsf{p}_1- \epsilon_2 \mathsf{p}_2, 3\mathsf{p}_1 + \epsilon_2 \mathsf{p}_2\}$ is obtained by setting $c \mapsto -c$ in Equation~(\ref{eqn:master_curve_3points}).
\end{remark}
\par We also briefly discuss the existence of an additional involution for the plane genus-three curve $\mathcal{D}$ in Equation~(\ref{eqn:master_curve_3points}). We have the following:
\begin{lemma}
\label{lem:extra_auto}
The bielliptic plane genus-three curve $\mathcal{D}$ in Proposition~\ref{cor:connection_3points} admits an additional involution of the form
\[
  [u:v:w] \mapsto [\alpha^2 v: u : \alpha w] \,,
\]
iff $a=\pm d/e$ and $\alpha^2=d/e$. In particular, such an involution exists for $\mathcal{D}$ if $\lambda_0\lambda_2=\lambda_1\lambda_3$ or $\lambda_0\lambda_3=\lambda_1\lambda_2$.
\end{lemma}
\begin{proof}
The first statement is immediate. The second follows when computing $e^2a-d^2$ in terms of $\lambda_0, \dots, \lambda_3, l$.
\end{proof}
\begin{remark}
For $\lambda_0\lambda_1=\lambda_2\lambda_3$ we find $\Delta_{\mathcal{D}}=0$ in Equation~(\ref{eqn:singular_locus}), and the curve $\mathcal{D}$ is singular. This is easily understood when observing that the construction of $\mathcal{D}$ in Proposition~\ref{cor:connection_3points} depends on two Weierstrass points corresponding to $\lambda_0$ and $\lambda_1$.
\end{remark}
\section{Proof of Theorem~\ref{thm:main}}
\label{sec:proof}
Until now we constructed a bi-double cover of $\mathbb{P}^1$ introducing the curves $\mathcal{H}$, $\mathcal{C}$, and $\mathcal{E} \cong \mathcal{Q}$ of genus three, two, and one in Section~\ref{sec3}, and  provided a precise geometric characterization of plane bielliptic genus-three curves $\mathcal{D}$, their bielliptic quotients, and the associated branch loci in Section~\ref{sec4}.  We now combine the results of the previous sections to prove our main theorem. In this Section we will prove Corollary~\ref{cor:smooth_irreducible} which implies that, under certain conditions, the bielliptic genus-three curve $\mathcal{D}$ will be smooth and irreducible. Moreover we will prove in Theorem~\ref{thm:main2} the existence of the $(1,2)$-isogeny between the Prym variety of $\mathcal{D}$ and the Jacobian variety of a smooth genus-two curve $\mathcal{C}$, by using Theorem~\ref{thm:Barth}, Propositions~\ref{prop:Divisors},~\ref{prop:Garbagnati2},~\ref{prop:Kum},~\ref{cor:bielliptic_quotient},~\ref{cor:connection_3points}.
\par We first determine on which fibers of the elliptic fibration with section $(\pi, \mathsf{O})$ given by Equation~(\ref{eqn:EFS}) on the abelian surfaces $\mathfrak{A}$ with $(1,2)$-polarization line bundle $\mathcal{L}$ the branch locus -- with respect to of the action induced by $-\mathbb{I}$ on $\mathfrak{A}$ -- consists of four points $\{\mathrm{pt}_1$, $\mathrm{pt}_2$, $\mathrm{pt}_3$, $\mathrm{pt}_4\} \subset \mathcal{E}$ such that $\sum_{i=1}^4 \mathrm{pt}_i=\mathsf{O}$.  A normal form for the elliptic fibration and the generators $\{\mathsf{O}, \mathsf{S}_1, \mathsf{S}_2, \mathsf{S}_3 \}$ of the Mordell-Weil group was provided in Corollary~\ref{cor:pencilQ} and Section~\ref{ssec:KummerPencils}. As explained in Section~\ref{ssec:AbSrfc}, Barth's Theorem~\ref{thm:Barth} asserts that  $\mathfrak{A}$ is naturally isomorphic to the Prym variety $\operatorname{Prym}(\mathcal{D},\pi^{\mathcal{D}}_{\mathcal{E}})$ of a smooth genus-three curve $\mathcal{D} \in |\mathcal{L}|$ with bielliptic involution $\tau$ such that $-\mathbb{I}$ restrict to $\tau$, the linear pencil $|\mathcal{L}|$ has precisely $T(\mathcal{L})=\{ \mathsf{P}_0, \mathsf{P}_1, \mathsf{P}_2, \mathsf{P}_3 \}$ as base points. The blow up in the base points is equivalent to the elliptic fibration with section $(\pi, \mathsf{O})$ with sections $\{\mathsf{O}, \mathsf{S}'_1, \mathsf{S}'_2, \mathsf{S}'_3 \}$ such that  the divisor classes $\{K_0, K_1, K_2, K_3\}$ given by
\begin{equation}
 K_0 = [\mathsf{O}]\, \quad K_1 = [ \mathsf{S}'_1] \,,\quad K_2 = [\mathsf{S}'_2] \,, \quad K_3 = [\mathsf{S}_3'] \,,
\end{equation}
are the four exceptional curves of the blow up; see Proposition~\ref{prop:Divisors}.  To that end, the sum of the sections $\{ \mathsf{S}'_1, \mathsf{S}'_2, \mathsf{S}'_3 \}$ representing the divisor classes $K_1, K_2, K_3$ in Proposition~\ref{prop:MW} has to vanish. This will happen in certain smooth and certain singular fibers of the elliptic fibration, and we are interested in the former. We have the following:
\begin{proposition}
\label{prop:SpecialFibers}
Table~\ref{tab:special_fibers} lists all points in the base curve of the elliptic fibration with section $(\pi,\mathsf{O})$ on $\operatorname{Kum}(\mathfrak{A})$ where the sum of sections $\{ \mathsf{S}'_1, \mathsf{S}'_2, \mathsf{S}'_3 \}$ representing divisor classes $K_1, K_2, K_3$ vanishes. Table~\ref{tab:special_fibers} is based on the four possible choices for $\{ \mathsf{S}'_1, \mathsf{S}'_2, \mathsf{S}'_3 \}$ determined by Proposition~\ref{prop:MW}. The polynomials $p_4(s_0, s_1)$ and $p_2^{(1,2,3)}(s_0, s_1)$ are the polynomials of degree $4$ and $2$, respectively, given by
\begin{equation}
\begin{split}
p_4(s_0, s_1) & = 2 \lambda_0  \lambda_1\lambda_2\lambda_3 l s_1^4 -  \lambda_0\lambda_1\lambda_2\lambda_3 (\lambda_0+\lambda_1+\lambda_2+\lambda_3)s_1^3 s_0 \\
& +( \lambda_0\lambda_1\lambda_2\lambda_3 + \lambda_0\lambda_1\lambda_2 + \lambda_0\lambda_1\lambda_3 + \lambda_0\lambda_2\lambda_3)s_1s_0^3-2l s_0^4\,,\\
p_2^{(1,2,3)}(s_0, s_1)& = \lambda_0\lambda_1\lambda_2\lambda_3(\lambda_0+\lambda_1-\lambda_2-\lambda_3) s_1^2 - 2 l (\lambda_0\lambda_1-\lambda_2\lambda_3) s_0 s_1\\
&+\lambda_0\lambda_1(\lambda_2 + \lambda_3)-\lambda_2\lambda_3(\lambda_0+\lambda_1)  \,,
\end{split}
\end{equation}
and $p_2^{(2,1,3)}$ and $p_2^{(3,1,2)}$ are obtained by interchanging $\lambda_1 \leftrightarrow \lambda_2$ and $\lambda_1 \leftrightarrow \lambda_3$, respectively, in $p_2^{(1,2, 3)}$. The parameters $\lambda_0, \lambda_1, \lambda_2, \lambda_3, l$ are the moduli of a general genus-two curve\footnote{We remind the reader that we write $\lambda_0$ and $l^2=\lambda_0\lambda_1\lambda_3\lambda_3$ rather than substituting in $\lambda_0=1$.} given in Lemma~\ref{Picard}.\end{proposition}
\begin{proof}
The conditions $2\mathsf{S}_1\pm \mathsf{S}_3=\mathsf{T}_i$ and $2\mathsf{S}_2\pm \mathsf{S}_3=\mathsf{T}_i$ for $0 \le i \le 3$ result from making the last column in the table in Proposition~\ref{prop:Divisors} vanish. Staying away from singular fibers of the elliptic fibration described in Section~\ref{ssec:EF}, the rest of the statement follows from explicit computation using the group law on the smooth elliptic fibers.
\end{proof}
\begin{table}
\scalemath{0.62}{
\begin{tabular}{c|c||c|c|c|c||c}
 \# & & $\mathsf{O}=\mathsf{T}_0$	& $\mathsf{T}_1$	& $\mathsf{T}_2$ 	& $\mathsf{T}_3$ & $\{\mathsf{O}, \mathsf{S}'_1, \mathsf{S}'_2, \mathsf{S}'_3\}$ \\[0.2em]
\hline\hline
1 & $\begin{array}{l} \text{relation:} \\ \text{constraint:} \end{array}$
& $\begin{array}{c} 2\mathsf{S}_1+\mathsf{S}_3=\mathsf{O} \,, \\ p_4(s_0,s_1)=0 \end{array}$
& \cellcolor{blue!25} $\begin{array}{c} 2\mathsf{S}_1+\mathsf{S}_3=\mathsf{T}_1 \,, \\ p_2^{(1,2,3)}(s_0,s_1)=0 \end{array}$
& \cellcolor{blue!25} $\begin{array}{c} 2\mathsf{S}_1+\mathsf{S}_3=\mathsf{T}_2 \,, \\ p_2^{(2,1,3)}(s_0,s_1)=0 \end{array}$
& \cellcolor{blue!25} $\begin{array}{c} 2\mathsf{S}_1+\mathsf{S}_3=\mathsf{T}_3 \,, \\ p_2^{(3,1,2)}(s_0,s_1)=0 \end{array}$
& $\begin{array}{c}\pm\left\lbrace \mathsf{O}, -\mathsf{S}_1+\mathsf{S}_2+\mathsf{T}_i, \right.\\\left. -\mathsf{S}_1-\mathsf{S}_2+\mathsf{T}_i, \, 2\mathsf{S}_1\right\rbrace\end{array}$\\[0.8em]
\hline\hline
2 & $\begin{array}{l} \text{relation:} \\ \text{constraint:} \end{array}$
& $\begin{array}{c} 2\mathsf{S}_1-\mathsf{S}_3=\mathsf{O} \,, \\ p_4(s_0,-s_1)=0 \end{array}$
&$\begin{array}{c} 2\mathsf{S}_1-\mathsf{S}_3=\mathsf{T}_1 \,, \\ p_2^{(1,2,3)}(s_0,-s_1)=0 \end{array}$
&$\begin{array}{c} 2\mathsf{S}_1-\mathsf{S}_3=\mathsf{T}_2 \,, \\ p_2^{(2,1,3)}(s_0,-s_1)=0 \end{array}$
&$\begin{array}{c} 2\mathsf{S}_1-\mathsf{S}_3=\mathsf{T}_3 \,, \\ p_2^{(3,1,2)}(s_0,-s_1)=0 \end{array}$
& $\begin{array}{c}\pm\left\lbrace \mathsf{O}, -\mathsf{S}_1+\mathsf{S}_2+\mathsf{T}_i, \right.\\\left. -\mathsf{S}_1-\mathsf{S}_2+\mathsf{T}_i, \, 2\mathsf{S}_1\right\rbrace\end{array}$\\[0.8em]
\hline\hline
3 & $\begin{array}{l} \text{relation:} \\ \text{constraint:} \end{array}$
& $\begin{array}{c} 2\mathsf{S}_2+\mathsf{S}_3=\mathsf{O} \,, \\ p_4(l s_1,s_0)=0 \end{array}$
& $\begin{array}{c} 2\mathsf{S}_2+\mathsf{S}_3=\mathsf{T}_1 \,, \\ p_2^{(1,2,3)}(l s_1,s_0)=0 \end{array}$
& $\begin{array}{c} 2\mathsf{S}_2+\mathsf{S}_3=\mathsf{T}_2 \,, \\ p_2^{(2,1,3)}(l s_1,s_0)=0 \end{array}$
& $\begin{array}{c} 2\mathsf{S}_2+\mathsf{S}_3=\mathsf{T}_3 \,, \\ p_2^{(3,1,2)}(l s_1,s_0)=0 \end{array}$
& $\begin{array}{c}\pm\left\lbrace \mathsf{O}, -\mathsf{S}_2+\mathsf{S}_1+\mathsf{T}_i, \right.\\\left. -\mathsf{S}_2-\mathsf{S}_1+\mathsf{T}_i, \, 2\mathsf{S}_2\right\rbrace\end{array}$\\[0.8em]
\hline\hline
4 &  $\begin{array}{l} \text{relation:} \\ \text{constraint:} \end{array}$
& $\begin{array}{c} 2\mathsf{S}_2-\mathsf{S}_3=\mathsf{O} \,, \\ p_4(l s_1,-s_0)=0 \end{array}$
& $\begin{array}{c} 2\mathsf{S}_2-\mathsf{S}_3=\mathsf{T}_1 \,, \\ p_2^{(1,2,3)}(l s_1,-s_0)=0 \end{array}$
& $\begin{array}{c} 2\mathsf{S}_2-\mathsf{S}_3=\mathsf{T}_2 \,, \\ p_2^{(2,1,3)}(l s_1,-s_0)=0 \end{array}$
& $\begin{array}{c} 2\mathsf{S}_2-\mathsf{S}_3=\mathsf{T}_3 \,, \\ p_2^{(3,1,2)}(l s_1,-s_0)=0 \end{array}$
& $\begin{array}{c}\pm\left\lbrace \mathsf{O}, -\mathsf{S}_2+\mathsf{S}_1+\mathsf{T}_i, \right.\\\left. -\mathsf{S}_2-\mathsf{S}_1+\mathsf{T}_i, \, 2\mathsf{S}_2\right\rbrace\end{array}$
\medskip
\end{tabular}}
\caption{Special fibers in the elliptic fibration $(\pi,\mathsf{O})$ on $\operatorname{Kum}(A)$}
\label{tab:special_fibers}
\end{table} 
The marked cells in Table~\ref{tab:special_fibers} determine points in the base curve of the elliptic fibration with section $(\pi: \mathscr{E} \to \mathbb{P}^1,\mathsf{O})$ where  the sum of sections $\{\mathsf{O}, \mathsf{S}'_1, \mathsf{S}'_2, \mathsf{S}'_3 \}$ vanishes. In particular, these points can be explicitly expressed in terms of modular forms \emph{and} the sections in terms of the rational points $\mathsf{p}_1, \mathsf{p}_2$ in Lemma~\ref{lem:adding}. We have the following:
\begin{corollary}
\label{cor:special_points}
For the six points in $\mathbb{P}^1$ given by
\begin{equation}
\label{eqn:base_points}
  [s^*_0:s^*_1]=[  \big(\lambda_0+\lambda_i - \lambda_j - \lambda_k\big) l \ : \ \lambda_0 \lambda_i - \lambda_j \lambda_k 
  \pm  m^{(i,j,k)}  (\lambda_0-\lambda_j)(\lambda_0-\lambda_k)   ]\;,
\end{equation}  
the sections $\{\mathsf{O}, \mathsf{S}'_1, \mathsf{S}'_2, \mathsf{S}'_3 \}$ coincide with the points $\{ \mathsf{O}, 2\mathsf{p}_1, \mathsf{p}_1+  \mathsf{p}_2, -3\mathsf{p}_1- \mathsf{p}_2 \}$ in fibers $\mathscr{E}_{[s^*_0:s^*_1]}$ given by Equation~(\ref{eqn:Kummer_EF}).  Here, $m^{(i,j,k)}$ satisfies
\[ 
 (m^{(i,j,k)})^2 = \frac{(\lambda_i-\lambda_j)(\lambda_i-\lambda_k)}{(\lambda_0-\lambda_i)(\lambda_0-\lambda_j)} \,,
 \]
 for $\{i, j ,k\} =\{1,2,3\}$,  the point $2\mathsf{p}_1$ has coordinates
\begin{equation}
\label{eqn:2p1_back}
\begin{split}
  \xi  & =  \big(\Lambda_0+\Lambda_1-\Lambda_2-\Lambda_3\big)^2\,,\quad \eta  = 1\,,\\
  \rho & = \big(\Lambda_0+\Lambda_1-\Lambda_2-\Lambda_3\big) \big(\Lambda_0-\Lambda_1-\Lambda_2+\Lambda_3\big)\big(\Lambda_0-\Lambda_1+\Lambda_2+\Lambda_3\big)\,,
\end{split}
\end{equation}
and $\mathsf{p}_1 + \mathsf{p}_2$ has coordinates
\begin{equation}
\label{eqn:p1+p2_back}
\begin{split}
  \xi  &  = 4 ( \Lambda_0\Lambda_1+\Lambda_2\Lambda_3 - 2L)\,, \quad   \eta  = 1\,,\\
  \rho   & = 8 \big(  L (\Lambda_0 + \Lambda_1+\Lambda_2 + \Lambda_3) - \Lambda_0 \Lambda_1\Lambda_2 -\Lambda_0 \Lambda_1\Lambda_3-\Lambda_0 \Lambda_2\Lambda_3-\Lambda_1 \Lambda_2\Lambda_3\big)\,,
\end{split}
\end{equation}
and $\Lambda_i=\Lambda_i(s^*_0,s^*_1)$ for $0 \le i \le 3$ and $L=L(s^*_0,s^*_1)$ are given in Equation~(\ref{eqn:params_Lambdas}).
\end{corollary}
\begin{proof}
We check that the points in Equation~(\ref{eqn:base_points}) satisfy $p_2^{(1,2,3)}(  s^*_0, s^*_1 )=0$ in Proposition~\ref{prop:SpecialFibers}. The relation between the generators of the Mordell-Weil group $\operatorname{MW}(\pi,\mathsf{O})$ and the points $\mathsf{p}_1$ and $\mathsf{p}_2$ is given in Equation~(\ref{eqn:sections_EF}). Using Proposition~\ref{prop:Divisors} it follows
\[
 \mathsf{S}'_1 = 2\mathsf{S}_1 =2 \mathsf{p}_1 \,, \quad
 \mathsf{S}'_2 =\mathsf{S}_1 + \mathsf{S}_2 + \mathsf{S}_3= \mathsf{p}_1 +  \mathsf{p}_3  \,, \quad
 \mathsf{S}'_3 = \mathsf{S}_1 - \mathsf{S}_2 + \mathsf{S}_3= \mathsf{p}_1 +  \mathsf{p}_2 \,.
 \]
The condition $2\mathsf{S}_1+ \mathsf{S}_3=\mathsf{T}_i$ with $1\le i \le 3$ in Proposition~\ref{prop:SpecialFibers}, then implies $\sum  \mathsf{S}'_i = 2\mathsf{T}_i = \mathsf{O}$. Because  $p_2^{(2,1,3)}$ and $p_2^{(3,1,2)}$ are obtained by interchanging $\lambda_1 \leftrightarrow \lambda_2$ and $\lambda_1 \leftrightarrow \lambda_3$, the statements follow for the other points in the base curve as well. 
\end{proof}
\par  By replacing $\lambda_i \mapsto \Lambda_i(s_0,s_1)$ for $0 \le i \le 3$ and $l \mapsto L(s_0,s_1)$ in Proposition~\ref{cor:connection_3points}, we obtain coefficients $A, B, C, D, E$ from $a, b, c, d, e$; $A, B$ are given by Equation~(\ref{eqn:coeffs_AB}) and
\begin{equation}\label{eqn:pencil_quartic_coeffs}
\begin{split}
 C(s_0,s_1) & = \; c\Big(\Lambda_0(s_0,s_1), \, \Lambda_1(s_0,s_1), \,  \Lambda_2(s_0,s_1), \,  \Lambda_3(s_0,s_1), \,  L(s_0,s_1) \Big) \,,\\
 D(s_0,s_1) & = \; d\Big(\Lambda_0(s_0,s_1), \, \Lambda_1(s_0,s_1), \,  \Lambda_2(s_0,s_1), \,  \Lambda_3(s_0,s_1), \,  L(s_0,s_1) \Big) \,, \\
 E(s_0,s_1) & = \; e\Big(\Lambda_0(s_0,s_1), \, \Lambda_1(s_0,s_1), \,  \Lambda_2(s_0,s_1), \,  \Lambda_3(s_0,s_1), \,  L(s_0,s_1) \Big) \,.
\end{split}
\end{equation}
A pencil $\mathscr{D}$ of plane bielliptic genus-three curves $\mathscr{D}_{[s_0:s_1]}$ is then given by
\begin{equation}
\label{eqn:pencil_D}
\begin{split}
 & \Big( E(s_0,s_1) (w^2 - u^2) - C(s_0,s_1) \, uv - D(s_0,s_1)\,  v^2 \Big)^2 \\
 & =  \; E(s_0,s_1)^2 \Big(u^4 + B(s_0,s_1) \, u^2 v^2 + A^2(s_0,s_1) \, v^4\Big) \,,
\end{split}
\end{equation}
with $[u:v:w] \in \mathbb{P}^2$  and $[s_0:s_1] \in \mathbb{P}^1$. We have the immediate:
\begin{corollary}
\label{cor:smooth_irreducible}
The plane genus-three curves $\mathscr{D}_{[s_0:s_1]}$ in Equation~(\ref{eqn:pencil_D}) are irreducible and non-singular for all $[s_0:s_1] \in \mathbb{P}^1$ with $\Delta_{\mathscr{E}}(s_0,s_1)  \Delta_{\mathscr{D}}(s_0,s_1) \not =0$ where
\begin{equation}
\label{eqn:singular_locus_pencil}
\begin{split}
\Delta_{\mathscr{E}}(s_0,s_1)  = & \;16 \, A^2 \big(B^2-4A^2\big)^2 \,, \\
\Delta_{\mathscr{D}}(s_0,s_1)  = & \; - \left(  \left( C^2-B E^2-4 D E\right)^2 -12 D E^2 \left( B E+D \right)  \right)^3  \\
&\; + \Big( 54 A  C^2 E^4 - C^6+ 3 \left( BE + 4 D \right) C^4 E\\
&\quad  - 3 \left( B^2E^2+2 BDE+10 D^2 \right) C^2E^2+ \left( BE-2 D \right)^{3}E^3 \Big)^{2} \,,
\end{split}
\end{equation}
and $A=A(s_0,s_1)$, $B=B(s_0,s_1)$, etc. In particular, we then have $E(s_0,s_1)\not = 0$.
\end{corollary}
\begin{proof}
The proof follows from Lemma~\ref{lem:smooth} when replacing $a \mapsto A(s_0,s_1)$, $b \mapsto B(s_0,s_1)$, etc.
\end{proof}
\par Then, each smooth and irreducible curve $\mathscr{D}_{[s_0:s_1]}$ in the pencil $\mathscr{D}$ admits the bielliptic involution $\tau: [u:v:w] \in \mathbb{P}^2 \mapsto [u:v:-w]$ interchanging the sheets of the degree-two cover 
\begin{equation}
\begin{split}
 \pi^{\mathscr{D}}_{\mathscr{E}}: \quad \mathscr{D}_{[s_0:s_1]} & \to \mathscr{E}_{[s_0:s_1]} \\
 [u:v:w] & \mapsto \left[ u:v:W=  w^2 - u^2 - \frac{C(s_0,s_1)}{E(s_0,s_1)} \,   uv - \frac{D(s_0,s_1)}{E(s_0,s_1)} \,  v^2\right]\,,
\end{split}
\end{equation}
onto the elliptic curve $\mathscr{E}_{[s_0:s_1]} \cong \mathscr{D}_{[s_0:s_1]}/\langle \tau \rangle$ in Equation~(\ref{eqn:Kummer_EF}). It follows from Barth's Theorem~\ref{thm:Barth} that the Prym varieties for the bielliptic curves $\mathscr{D}_{[s_0:s_1]}$
are abelian surfaces with polarization of type $(1,2)$.
We have the following:
\begin{theorem}
\label{thm:main2}
The Prym varieties of the smooth plane bielliptic genus-three curves $\mathscr{D}_{[s^*_0:s^*_1]}$ obtained as fibers of the pencil in Equation~(\ref{eqn:pencil_D}) over $[s^*_0:s^*_1] \in \mathbb{P}^1$ in Equation~(\ref{eqn:base_points}) with $\Delta_{\mathscr{E}}(s^*_0,s^*_1)  \Delta_{\mathscr{D}}(s_0^*,s_1^*) \not =0$ admit a $(1,2)$-isogeny 
\[ 
\Psi: \operatorname{Prym}\left(\mathscr{D}_{[s^*_0:s^*_1]},\pi^{\mathscr{D}}_{\mathscr{E}}\right) \to \operatorname{Jac}(\mathcal{C})
\]
onto the principally polarized abelian surface $\operatorname{Jac}(\mathcal{C})$ for $\mathcal{C}$ in Equation~(\ref{Eq:Rosenhain_g2}).
\end{theorem}
\begin{proof}
It follows from Proposition~\ref{prop:Garbagnati2} that the elliptic fibration $\pi: \mathscr{E} \to \mathbb{P}^1$ in Equation~(\ref{eqn:Kummer_EF}) is the special elliptic fibration $\pi: \operatorname{Kum}(\mathfrak{A}) \to \mathbb{P}^1$ with section $\mathsf{O}$  in Proposition~\ref{prop:Garbagnati} on the Kummer surface of the abelian surface $\mathfrak{A}$ with a polarization of type $(1,2)$. By Barth's Theorem~\ref{thm:Barth} the elliptic fibration is induced by a pencil of bielliptic genus-three curves. The bielliptic involution $\tau$ has fixed points $\{ \mathsf{P}_0, \mathsf{P}_1, \mathsf{P}_2, \mathsf{P}_3 \}$. We proved in Proposition~\ref{prop:Divisors} that the branch points of the bielliptic involution are given by the sections $\{\mathsf{O}, \mathsf{S}'_1 ,\mathsf{S}'_2 , \mathsf{S}'_3\}$ that represent divisor classes 
\begin{equation}
 K_0 = [\mathsf{O}]\, \quad K_1 = [ \mathsf{S}'_1] \,,\quad K_2 = [\mathsf{S}'_2] \,, \quad K_3 = [\mathsf{S}_3'] \,.
\end{equation}
The double points are the images of the order-two points $\{ \mathsf{P}_0, \dots, \mathsf{P}_{15}\}$ on $\mathfrak{A}$, i.e., elements of $\mathfrak{A}[2]$, and the disjoint rational curves $\{ K_0, \dots, K_{15} \}$ are the exceptional divisors introduced in the blow-up process.  We proved in Corollary~\ref{cor:special_points} that over the points $[s^*_0:s^*_1] \in \mathbb{P}^1$ in Equation~(\ref{eqn:base_points}) the sections add up to zero with respect to the elliptic-curve group law and coincides with $\{ \mathsf{O}, 2\mathsf{p}_1, \mathsf{p}_1+  \mathsf{p}_2, -3\mathsf{p}_1- \mathsf{p}_2 \}$ in fibers $\mathscr{E}_{[s^*_0:s^*_1]}$. It follows from Proposition~\ref{cor:connection_3points} that a plane bielliptic genus-three curve covering $\mathscr{E}_{[s^*_0:s^*_1]}$ with the same branch locus is obtained from Equation~(\ref{eqn:master_curve_3points}) by replacing $\lambda_i \mapsto \Lambda_i(s^*_0,s^*_1)$ for $0 \le i \le 3$ and $l \mapsto L(s^*_0,s^*_1)$ in Proposition~\ref{cor:connection_3points}. By Proposition~\ref{cor:bielliptic_quotient} this model is unique. Moreover, it follows from \cite{MR1816214}*{Corollary~2.2} that any such smooth curve is the canonical model of a bielliptic non-hyperelliptic curve of genus three. It follows from Proposition~\ref{prop:Kum} that the polarization of type $(1,2)$ on the abelian surface $\mathfrak{A}$ is induced by an isogeny $\Psi: \mathfrak{A} \to \operatorname{Jac}(\mathcal{C})$ onto the principally polarized abelian surface $\operatorname{Jac}(\mathcal{C})$ for the genus-two curve in Equation~(\ref{Eq:Rosenhain_g2}).
\end{proof}
A tedious computation shows:
\begin{corollary}
The coefficients $A, B, C, D, E$ of the bielliptic genus-three curves $\mathscr{D}_{[s^*_0:s^*_1]}$ are polynomials in $\mathbb{Z}[\lambda_0, \lambda_1, \lambda_2, \lambda_3, l, l^{-1}, m^{(i,j,k)}]$.
\end{corollary}
We determine some symmetries of these functions. We have the following:
\begin{lemma}
For $\{i, j ,k\} =\{1,2,3\}$ and
\begin{equation}
  [s^*_0:s^*_1]=[  \big(\lambda_0+\lambda_i - \lambda_j - \lambda_k\big) l \ : \ \lambda_0 \lambda_i - \lambda_j \lambda_k 
  \pm  m^{(i,j,k)}  (\lambda_0-\lambda_j)(\lambda_0-\lambda_k)   ]\;,
\end{equation}
and $\lambda_0 \lambda_i=\lambda_j\lambda_k$, we have $C = D = E = 0$ in Equation~(\ref{eqn:pencil_D}).
\end{lemma}
\begin{corollary}
For $\{i, j ,k\} =\{1,2,3\}$ and
\begin{equation}
  [s^*_0:s^*_1]=[  \big(\lambda_0+\lambda_i - \lambda_j - \lambda_k\big) l : \lambda_0 \lambda_i - \lambda_j \lambda_k 
  \pm  m^{(i,j,k)}  (\lambda_0-\lambda_j)(\lambda_0-\lambda_k)   ]\;,
\end{equation}
and one additional relation given by
\begin{center}
\begin{tabular}{c|c}
$(i,j,k)$ & relation \\
\hline
$(1,2,3)$ & $\lambda_0\lambda_2=\lambda_1\lambda_3$ or $\lambda_0\lambda_3=\lambda_1\lambda_2$ \\
$(2,1,3)$ & $\lambda_0\lambda_1=\lambda_2\lambda_3$  \\
$(3,1,2)$ & $\lambda_0\lambda_1=\lambda_2\lambda_3$  \\
\end{tabular}
\end{center}
the smooth and irreducible bielliptic plane genus-three curve $\mathscr{D}_{[s^*_0:s^*_1]}$ in Theorem~\ref{thm:main2} admits an additional involution of the form
\[
  [u:v:w] \mapsto [\alpha^2 v: u : \alpha w] \,,
\]
with $\alpha^2=D(s_0^*,s_1^*)/E(s_0^*,s_1^*)$. In particular, the Jacobian variety of the smooth genus-two curve $\mathcal{C}$ is two-isogenous to a product of two elliptic curves, i.e., $\operatorname{Jac}(\mathcal{C}) \sim_2 \mathcal{E}_1 \times \mathcal{E}_2$.
\end{corollary}
\begin{proof}
The first part follows from Lemma~\ref{lem:extra_auto} and a tedious computation after replacing $a \mapsto A(s^*_0,s^*_1)$, $b \mapsto B(s^*_0,s^*_1)$, etc. For the second part, we compute the Igusa-Clebsch invariants of $\mathcal{C}$, denoted by by $[ I_2 : I_4 : I_6 : I_{10} ] \in \mathbb{P}(2,4,6,10)$, using the same normalization as in \cites{MR3712162,MR3731039}. One can then ask what the Igusa invariants of a genus-two curve $\mathcal{C}$ defined by a sextic curve are in terms of $\underline{\tau}$  such that $(\underline{\tau}, \mathbb{I}_2) \in \mathrm{Mat}(2, 4;\mathbb{C})$ is the period matrix of the principally polarized abelian surface  $\mathrm{Jac}(\mathcal{C})$. This allows us to compute the Siegel modular forms $\psi_4$, $\psi_6$, $\chi_{10}$, $\chi_{12}$ for $\mathrm{Jac}(\mathcal{C})$, as introduced by Igusa in \cite{MR229643}. Igusa also proved in \cite{MR527830} that the ring of Siegel modular forms is generated by $\psi_4$, $\psi_6$, $\chi_{10}$, $\chi_{12}$ and by one more cusp form $\chi_{35}$ of odd weight $35$ whose square is the following polynomial \cite{MR229643}*{p.~\!849}. One checks that the additional relation implies $\chi_{35}(\mathcal{C})^2=0$. On the other hand it is well known that for any smooth genus-two curve $\mathcal{C}$ with $\chi_{35}(\mathcal{C})=0$, its Jacobian is two-isogenous to a product of two elliptic curves, i.e., $\operatorname{Jac}(\mathcal{C}) \sim_2 \mathcal{E}_1 \times \mathcal{E}_2$; see \cite{MR229643}.
\end{proof}
\begin{proof}[Proof of Main Theorem~\ref{thm:main}]
Corollary~\ref{cor:smooth_irreducible} proves that for $\Delta_{\mathscr{E}}(s^*_0,s^*_1)  \Delta_{\mathscr{D}}(s^*_0,s^*_1) \not =0$ the curves $\mathscr{D}_{[s^*_0:s^*_1]}$ are smooth and irreducible. Theorem~\ref{thm:main2} proves the existence of a $(1,2)$-isogeny. Lemma~\ref{lem:Picard} and Remark~\ref{rem:mu} provide explicit formulas for $\lambda_0$, $\lambda_1$, $\lambda_2$, $\lambda_3$, $l$, $m^{(i,j,k)}$ in terms of theta functions.
\end{proof}
\begin{appendix}
\section{Coefficients of plane bielliptic genus-three curves} \label{App:coeffs}
The plane bielliptic genus-three curve $\mathcal{D}$ in Proposition~\ref{cor:connection_3points} is given by
\begin{equation}
\left( w^2 - u^2 - \frac{c}{e}  uv - \frac{d}{e}  v^2 \right)^2 = u^4 + b  u^2 v^2 + a^2 v^4 \,,
\end{equation}
with $[u:v:w] \in \mathbb{P}^2$  and coefficients
\begin{equation}\label{eqn:coeffs_ab_app}
\begin{split}
 a & = (\lambda_0-\lambda_1)(\lambda_2-\lambda_3) \,, \\
 b & = 4 \lambda_0 \lambda_1+ 4 \lambda_2 \lambda_3-2 \lambda_0 \lambda_2 - 2  \lambda_0 \lambda_3 -  2\lambda_1 \lambda_2  - 2 \lambda_1 \lambda_3 \,,
\end{split}
\end{equation}
and
\begin{equation}\label{eqn:coeffs_c_app}
\begin{split}
c=& \; c(\lambda_0, \lambda_1, \lambda_2, \lambda_3, l)  = -4 \prod_{i=0}^1 \prod_{j=1}^2 (\lambda_i-\lambda_j)
\big( \sum_{i=0}^3 \lambda_i^2 - 2 \sum_{0\le i < j \le 3} \lambda_i \lambda_j + 8l \big) \,,
 \end{split}
\end{equation}
\begin{equation}\label{eqn:coeffs_d_app}
\begin{split}
d=& \; d(\lambda_0, \lambda_1, \lambda_2, \lambda_3, l) = -\lambda_0^5 \lambda_1 \lambda_2-\lambda_0^5 \lambda_1 \lambda_3+2 \lambda_0^5 \lambda_2 \lambda_3-4 \lambda_0^4 \lambda_1^2 \lambda_2-4 \lambda_0^4 \lambda_1^2 \lambda_3\\
&+5 \lambda_0^4 \lambda_1 \lambda_2^2+5 \lambda_0^4 \lambda_1 \lambda_3^2-\lambda_0^4 \lambda_2^2 \lambda_3-\lambda_0^4 \lambda_2 \lambda_3^2+10 \lambda_0^3 \lambda_1^3 \lambda_2+10 \lambda_0^3 \lambda_1^3 \lambda_3
-5 \lambda_0^3 \lambda_1^2 \lambda_2^2\\
&-2 \lambda_0^3 \lambda_1^2 \lambda_2 \lambda_3-5 \lambda_0^3 \lambda_1^2 \lambda_3^2-5 \lambda_0^3 \lambda_1 \lambda_2^3+\lambda_0^3 \lambda_1 \lambda_2^2 \lambda_3+\lambda_0^3 \lambda_1 \lambda_2 \lambda_3^2-5 \lambda_0^3 \lambda_1 \lambda_3^3\\
&-5 \lambda_0^3 \lambda_2^3 \lambda_3 +10 \lambda_0^3 \lambda_2^2 \lambda_3^2-5 \lambda_0^3 \lambda_2 \lambda_3^3-4 \lambda_0^2 \lambda_1^4 \lambda_2-4 \lambda_0^2 \lambda_1^4 \lambda_3-5 \lambda_0^2 \lambda_1^3 \lambda_2^2\\
&-2 \lambda_0^2 \lambda_1^3 \lambda_2 \lambda_3-5 \lambda_0^2 \lambda_1^3 \lambda_3^2+10 \lambda_0^2 \lambda_1^2 \lambda_2^3+10 \lambda_0^2 \lambda_1^2 \lambda_3^3-\lambda_0^2 \lambda_1 \lambda_2^4+\lambda_0^2 \lambda_1 \lambda_2^3 \lambda_3\\
&+\lambda_0^2 \lambda_1 \lambda_2 \lambda_3^3-\lambda_0^2 \lambda_1 \lambda_3^4+5 \lambda_0^2 \lambda_2^4 \lambda_3-5 \lambda_0^2 \lambda_2^3 \lambda_3^2-5 \lambda_0^2 \lambda_2^2 \lambda_3^3+5 \lambda_0^2 \lambda_2 \lambda_3^4-\lambda_0 \lambda_1^5 \lambda_2\\
&-\lambda_0 \lambda_1^5 \lambda_3+5 \lambda_0 \lambda_1^4 \lambda_2^2+5 \lambda_0 \lambda_1^4 \lambda_3^2-5 \lambda_0 \lambda_1^3 \lambda_2^3+\lambda_0 \lambda_1^3 \lambda_2^2 \lambda_3+\lambda_0 \lambda_1^3 \lambda_2 \lambda_3^2-5 \lambda_0 \lambda_1^3 \lambda_3^3\\
&-\lambda_0 \lambda_1^2 \lambda_2^4+\lambda_0 \lambda_1^2 \lambda_2^3 \lambda_3+\lambda_0 \lambda_1^2 \lambda_2 \lambda_3^3-\lambda_0 \lambda_1^2 \lambda_3^4+2 \lambda_0 \lambda_1 \lambda_2^5-2 \lambda_0 \lambda_1 \lambda_2^3 \lambda_3^2\\
&-2 \lambda_0 \lambda_1 \lambda_2^2 \lambda_3^3+2 \lambda_0 \lambda_1 \lambda_3^5-\lambda_0 \lambda_2^5 \lambda_3
-4 \lambda_0 \lambda_2^4 \lambda_3^2+10 \lambda_0 \lambda_2^3 \lambda_3^3-4 \lambda_0 \lambda_2^2 \lambda_3^4\\
&-\lambda_0 \lambda_2 \lambda_3^5+2 \lambda_1^5 \lambda_2 \lambda_3-\lambda_1^4 \lambda_2^2 \lambda_3-\lambda_1^4 \lambda_2 \lambda_3^2-5 \lambda_1^3 \lambda_2^3 \lambda_3+10 \lambda_1^3 \lambda_2^2 \lambda_3^2\\
&-5 \lambda_1^3 \lambda_2 \lambda_3^3+5 \lambda_1^2 \lambda_2^4 \lambda_3-5 \lambda_1^2 \lambda_2^3 \lambda_3^2
-5 \lambda_1^2 \lambda_2^2 \lambda_3^3+5 \lambda_1^2 \lambda_2 \lambda_3^4-\lambda_1 \lambda_2^5 \lambda_3-4 \lambda_1 \lambda_2^4 \lambda_3^2\\
&+10 \lambda_1 \lambda_2^3 \lambda_3^3-4 \lambda_1 \lambda_2^2 \lambda_3^4-\lambda_1 \lambda_2 \lambda_3^5 + \Big(8 \lambda_0^4 \lambda_1-4 \lambda_0^4 \lambda_2-4 \lambda_0^4 \lambda_3-8 \lambda_0^3 \lambda_1^2+4 \lambda_0^3 \lambda_2^2\\
&+4 \lambda_0^3 \lambda_3^2-8 \lambda_0^2 \lambda_1^3+8 \lambda_0^2 \lambda_1^2 \lambda_2+8 \lambda_0^2 \lambda_1^2 \lambda_3
-4 \lambda_0^2 \lambda_1 \lambda_2^2-4 \lambda_0^2 \lambda_1 \lambda_3^2+4 \lambda_0^2 \lambda_2^3\\&
-4 \lambda_0^2 \lambda_2^2 \lambda_3-4 \lambda_0^2 \lambda_2 \lambda_3^2+4 \lambda_0^2 \lambda_3^3+8 \lambda_0 \lambda_1^4-4 \lambda_0 \lambda_1^2 \lambda_2^2-4 \lambda_0 \lambda_1^2 \lambda_3^2-4 \lambda_0 \lambda_2^4\\
&+8 \lambda_0 \lambda_2^2 \lambda_3^2-4 \lambda_0 \lambda_3^4-4 \lambda_1^4 \lambda_2-4 \lambda_1^4 \lambda_3+4 \lambda_1^3 \lambda_2^2+4 \lambda_1^3 \lambda_3^2+4 \lambda_1^2 \lambda_2^3-4 \lambda_1^2 \lambda_2^2 \lambda_3\\
&-4 \lambda_1^2 \lambda_2 \lambda_3^2+4 \lambda_1^2 \lambda_3^3-4 \lambda_1 \lambda_2^4
+8 \lambda_1 \lambda_2^2 \lambda_3^2-4 \lambda_1 \lambda_3^4+8 \lambda_2^4 \lambda_3\\
&-8 \lambda_2^3 \lambda_3^2-8 \lambda_2^2 \lambda_3^3+8 \lambda_2 \lambda_3^4\Big) l\,,
 \end{split}
\end{equation}
\begin{equation}\label{eqn:coeffs_e_app}
\begin{split}
e=&\; e(\lambda_0, \lambda_1, \lambda_2, \lambda_3, l)  = (\lambda_0 + \lambda_1 - \lambda_2 - \lambda_3)
	\Big(  2\lambda_0^2\lambda_1^2 -\lambda_0^3 \lambda_1  -2\lambda_0^2\lambda_1\lambda_2
	\\ &-3\lambda_0^2\lambda_2\lambda_3-\lambda_0\lambda_1^3 -2\lambda_0\lambda_1^2\lambda_2 -2\lambda_0\lambda_1^2\lambda_3
	+3\lambda_0\lambda_1\lambda_2^2+3\lambda_0\lambda_1\lambda_3^2\\
&+2\lambda_0\lambda_2^2\lambda_3+2\lambda_0\lambda_2\lambda_3^2-3\lambda_1^2\lambda_2\lambda_3
	+2\lambda_1\lambda_2^2\lambda_3+2\lambda_1\lambda_2\lambda_3^2+\lambda_2^3\lambda_3\\
& -2\lambda_0^2\lambda_1\lambda_3-2\lambda_2^2\lambda_3^2+\lambda_2\lambda_3^3 + 4 \big(\lambda_0^2+\lambda_1^2-\lambda_2^2-\lambda_3^2\big) l \Big) \,.
 \end{split}
\end{equation}
Here, we assumed that the curve $\mathcal{D}$ is irreducible and smooth such that $\Delta_{\mathcal{E}} \Delta_{\mathcal{D}} \not =0$ (and $e\not =0$ in particular) where

\begin{equation}
\label{eqn:singular_locus_app}
\begin{split}
\Delta_{\mathcal{E}} = & \; 16 \, a^2 \big(b^2-4a^2\big)^2 \,,  \\
\Delta_{\mathcal{D}} = & \; - \left(  \left( c^2-be^2-4\, de \right)^2 -12 \, d e^2 \left( be+d \right)  \right)^3 \ \\
&\; + \Big( 54 \, a  c^2 e^4 - c^6+ 3 \left( be + 4 d \right) c^4 e\\
&\quad - 3 \left( b^2e^2+2 \, bde+10 \, d^2 \right) c^2e^2+ \left( be-2\,d \right)^{3}e^3 \Big)^{2} \,.
\end{split}
\end{equation}
\medskip
\par The plane bielliptic genus-three curves $\mathscr{D}_{[s^*_0:s^*_1]}$ in Theorem~\ref{thm:main} are given by
\begin{equation}
\label{eqn:pencil_D_app}
\left( w^2 - u^2 - \frac{C(s^*_0,s^*_1)}{E(s^*_0,s^*_1)} \, uv - \frac{D(s^*_0,s^*_1)}{E(s^*_0,s^*_1)}\,  v^2 \right)^2 = u^4 + B(s^*_0,s^*_1) \, u^2 v^2 + A^2(s^*_0,s^*_1) \, v^4 \,,
\end{equation}
with $[u:v:w] \in \mathbb{P}^2$  and $[s_0:s_1] \in \mathbb{P}^1$, and coefficients
\begin{equation}
\label{eqn:coeffs_AB_app}
\begin{split}
 A(s_0,s_1) & =\big(\Lambda_0(s_0,s_1)-\Lambda_1(s_0,s_1)\big)\big(\Lambda_2(s_0,s_1)-\Lambda_3(s_0,s_1)\big) \,,\\
 B(s_0,s_1) & = 4 \Lambda_0(s_0,s_1) \, \Lambda_1(s_0,s_1)+ 4 \Lambda_2(s_0,s_1) \, \Lambda_3(s_0,s_1)-2 \Lambda_0(s_0,s_1) \, \Lambda_2(s_0,s_1) \\
 &  - 2  \Lambda_0(s_0,s_1)  \, \Lambda_3(s_0,s_1)  -  2\Lambda_1(s_0,s_1) \, \Lambda_2(s_0,s_1)   - 2 \Lambda_1(s_0,s_1) \, \Lambda_3(s_0,s_1)  \,,
\end{split}
\end{equation}
and
\begin{equation}\label{eqn:coeffs_CDE_app}
\begin{split}
 C(s_0,s_1) & = \; c\Big(\Lambda_0(s_0,s_1), \, \Lambda_1(s_0,s_1), \,  \Lambda_2(s_0,s_1), \,  \Lambda_3(s_0,s_1), \,  L(s_0,s_1) \Big) \,,\\
 D(s_0,s_1) & = \; d\Big(\Lambda_0(s_0,s_1), \, \Lambda_1(s_0,s_1), \,  \Lambda_2(s_0,s_1), \,  \Lambda_3(s_0,s_1), \,  L(s_0,s_1) \Big) \,, \\
 E(s_0,s_1) & = \; e\Big(\Lambda_0(s_0,s_1), \, \Lambda_1(s_0,s_1), \,  \Lambda_2(s_0,s_1), \,  \Lambda_3(s_0,s_1), \,  L(s_0,s_1) \Big) \,,
\end{split}
\end{equation}
where the functions $c, d, e$ are given in Equations~(\ref{eqn:coeffs_c_app})--(\ref{eqn:coeffs_e_app}) with $\lambda_0, \lambda_1, \lambda_2, \lambda_3, l$ replaced by $\Lambda_0, \Lambda_1, \Lambda_2, \Lambda_3, L$ with
\begin{equation}
\label{eqn:Lambdas_app}
 \Lambda_i(s_0,s_1) =   \frac{(s_0+ \lambda_i s_1)^2}{\lambda_i}  \,, \qquad  L(s_0,s_1) = \frac{\prod_{i=0}^3 (s_0+ \lambda_i s_1)}{l} \,, 
\end{equation}
for $0 \le i \le 3$. Moreover, the special point $[s_0:s_1]=[s^*_0:s^*_1]$ to be used in Equation~(\ref{eqn:pencil_D_app}) is given by
\begin{equation}
\label{eqn:sol_s_app}
  [s^*_0:s^*_1]=\Big[  \big(\lambda_0+\lambda_i - \lambda_j - \lambda_k\big) l  :  (\lambda_0 \lambda_i - \lambda_j \lambda_k)
  \pm  m^{(i,j,k)}  (\lambda_0-\lambda_j)(\lambda_0-\lambda_k)  \Big]\;,
\end{equation}  
for all $\{ i, j, k\} = \{1,2,3\}$, and
\begin{equation}\label{eqn:lambdas_app}
\begin{aligned}
\lambda_0=1 \,, \quad \lambda_1 = \frac{\theta_1^2\theta_3^2}{\theta_2^2\theta_4^2} \,, \quad \lambda_2 = \frac{\theta_3^2\theta_8^2}{\theta_4^2\theta_{10}^2}\,, \quad \lambda_3 = \frac{\theta_1^2\theta_8^2}{\theta_2^2\theta_{10}^2}\,,\quad
 l =  \frac{\theta_1^2\theta_3^2\theta_8^2}{\theta_2^2\theta_4^2\theta_{10}^2}  \,,\\
 m^{(1,2,3)} =  \frac{\theta_1\theta_3\theta^2_6}{\theta_2\theta_4\theta^2_5}\,, \quad
  m^{(2,1,3)} =  i \frac{\theta_3\theta_8\theta^2_6}{\theta_4\theta_{10}\theta^2_7}\,,\quad
 m^{(3,1,2)} =  \frac{\theta_1\theta_8\theta^2_6}{\theta_2\theta_{10}\theta^2_9}\,,
\end{aligned}
\end{equation}
and the ten even theta functions $\theta_i^2=\theta_i^2(0,\tau)$  with zero elliptic argument, modular argument $\tau \in \mathbb{H}_2/\Gamma_2(2)$, and $1\le i \le10$ follow the same standard notation for even theta functions as the one used in \cites{MR0141643, MR0168805, MR3712162, MR3731039, Clingher:2018aa}. 
Since the curves $\mathscr{D}_{[s^*_0:s^*_1]}$ are assumed to be irreducible and smooth we have $\Delta_{\mathscr{E}}(s^*_0,s^*_1)  \Delta_{\mathscr{D}}(s_0^*,s_1^*) \not =0$ where
\begin{equation}
\label{eqn:singular_locus_pencil_app}
\begin{split}
\Delta_{\mathscr{E}}(s_0,s_1)  = & \;16 \, A^2 \big(B^2-4A^2\big)^2 \,, \\
\Delta_{\mathscr{D}}(s_0,s_1)  = & \; - \left(  \left( C^2-B E^2-4 D E\right)^2 -12 D E^2 \left( B E+D \right)  \right)^3  \\
&\; + \Big( 54 A  C^2 E^4 - C^6+ 3 \left( BE + 4 D \right) C^4 E\\
&\quad  - 3 \left( B^2E^2+2 BDE+10 D^2 \right) C^2E^2+ \left( BE-2 D \right)^{3}E^3 \Big)^{2} \,,
\end{split}
\end{equation}
and $A=A(s_0,s_1)$, $B=B(s_0,s_1)$, etc.%
\end{appendix} 
\bibliographystyle{amsplain}
\bibliography{ref}{}
\end{document}